\documentclass[a4paper,12pt]{article}
\usepackage{marginnote}
\usepackage{hyperref}
\hypersetup{
	colorlinks,
	linkcolor={red!50!black},
	citecolor={blue!50!black},
	urlcolor={blue!80!black}
}
\usepackage{amsmath,amssymb,amsfonts,amscd, graphicx, latexsym, verbatim, multirow, color,extarrows}
\graphicspath{{figures/}}
\usepackage{relsize}

\usepackage{xfp}			
\usepackage{ifthen}			
\usepackage[shortlabels]{enumitem}
\usepackage{spath3}			
\usepackage{quoting,xparse}	

\usepackage{mathrsfs}
\usepackage[font={scriptsize}]{caption}
\usepackage{tikz}
\usepackage{tikz-cd}
	\usetikzlibrary{calc}
	\usetikzlibrary{intersections}
	\usetikzlibrary[decorations.pathmorphing,mindmap]
	\usetikzlibrary{decorations.markings}
	\usetikzlibrary{decorations.pathreplacing}
\usepackage{framed}
\usepackage[marginparwidth=20mm]{geometry}
\usepackage{ stmaryrd }
\definecolor{shadecolor}{rgb}{.9, .9, .9}
\definecolor{gr}{rgb}{.2, 1, .2}

\NewDocumentCommand{\bywhom}{m}{
	{\nobreak\hfill\penalty50\hskip1em\null\nobreak
		\hfill\mbox{\normalfont(#1)}%
		\parfillskip=0pt \finalhyphendemerits=0 \par}%
}
\NewDocumentEnvironment{commnote}{m}
{\begin{quoting}[
		indentfirst=true,
		leftmargin=\parindent,
		rightmargin=\parindent]\itshape
		\color{red!50!black}}
	{\bywhom{#1}\end{quoting}}

\newcommand{\circsize}{3pt}
\newcommand{\vY}[1]
	{\draw[fill=black] (0,0) circle (#1);}
\newcommand{\vI}[1]{
	\begin{scope}[scale=.8]
	\draw[fill=white] (0,0) circle (#1);
	\draw (-.7*#1,-.7*#1) -- (.7*#1,.7*#1);
	\draw (.7*#1,-.7*#1) -- (-.7*#1,.7*#1);
	\end{scope}
	}
\tikzset{
	emphline/.style=
	{line width=2mm,yellow},
	hedge/.style n args={1}{
		postaction={decorate},
		decoration={markings,
			mark= at position 0 with \vY{#1},
			mark= at position .9995 with \vI{#1} }  
	},
	hedge/.default=3px,
	fedge/.style n args={1}{
		postaction={decorate},
		decoration={markings,
			mark= at position 0 with \vY{#1},
			mark= at position .5 with \vI{#1},
			mark= at position .999 with \vY{#1}
		}  
	},
	fedge/.default=3pt,
	fedgeshiftmid/.style n args={1}{
		postaction={decorate},
		decoration={markings,
			mark= at position 0 with \vY{#1},
			mark= at position .5 with \vI{#1},
			mark= at position 1 with \vY{#1}
		}  
	},
		brace/.style={
			decoration={brace, mirror},
			decorate
	}
}
\newcommand{\bctr}[3]{barycentric cs:#1=1,#2=1,#3=1} 	
\newcommand{\TwinEdges}[4]								
{
	\coordinate (X) at (intersection of #1--#2 and #3--#4);
	\draw[halfa] (#1)--(X);
	\draw[halfb] (X)--(#2);
	\draw
	[cross,circle, fill=white,name]
	(X) circle (2pt) node {};
}
\newcommand{\TwinEdgesBold}[4]							
{
	\coordinate (X) at (intersection of #1--#2 and #3--#4);
	\draw[emphline] (#1)--(#2);
	\draw[halfa,ultra thick] (#1)--(X);
	\draw[halfb,ultra thick] (X)--(#2);
	\draw
	[cross,circle, fill=white,name]
	(X) circle (2pt) node {};
}
\makeatletter	
\tikzset
{anchor at angle/.code=
	{\pgfpointtransformed{\pgfpointpolar{#1}{1pt}}
		\pgfmathanglebetweenpoints{\pgfqpoint{-\pgf@x}{-\pgf@y}}{\pgfpointorigin}
		\def\tikz@anchor{\pgfmathresult}
	}
}
\makeatother
\tikzset			
{
	line1/.style=
	{lightgray},
	line2/.style={},
	halfa/.style=
	{ultra thick,blue},
	halfb/.style=
	{ultra thick,green!50!black},
	cross/.style={path picture={
			\draw[black]
			(path picture bounding box.south east) -- (path picture bounding box.north west) (path picture bounding box.south west) -- (path picture bounding box.north east);
	}}
}

\newcommand{\NiL}[1]{\fpeval{round(2^{#1})}} 	
\newcommand{\vertex}[2]{(\fpeval{(2*{\inteval{#2}}-1)*2/\NiL{\fpeval{#1}}},\Step*\fpeval{#1})}

\newcommand{\edge}[3][gray]
{
	\draw[#1, fedge] #2 -- #3;
}

\newcommand{\ghostfork}[2]
{
	\draw[dotted] \vertex{#1}{#2} -- \vertex{1+#1}{2*#2};
	\draw[dotted] \vertex{#1}{#2} -- \vertex{1+#1}{2*#2-1};
}

\newcommand{\halfedge}[3][gray]
{
	\draw[#1, hedge] #2 -- #3;
}

\newcommand{\halfedgebezier}[4][gray]
{
	\draw[#1,hedge] #2  .. controls #4 ..  #3;
}

\usetikzlibrary{mindmap,matrix,arrows,decorations.pathmorphing,patterns,fadings}
\usepackage{subfig, pgf}
\usepackage{amsfonts, amssymb, amsmath, amsthm}
\usepackage[normalem]{ulem} 
\usepackage[cmtip,all]{xy}
\newcommand{\longsquiggly}{\xymatrix{{}\ar@{~>}[r]&{}}}

\newcommand{\tens}{\raisebox{.35mm}{$\mathsmaller{\mathsmaller{\ttimes}}$}}
\newcommand{\F}{\mathcal F}
\newcommand{\G}{\mathcal G }

\renewcommand{\mod}{\mathsf{Mod}}
\newcommand{\mcg}{\mathbf{MCG}}
\newcommand{\aut}{\mathsf{Aut}}
\newcommand{\out}{\mathsf{Out}}

\newcommand{\mor}{\textrm{Mor}}
\newcommand{\obj}{\textrm{Obj}}

\newcommand{\bull}{{\mathlarger{\bullet}}}
\newcommand{\oasterisk}{\bull}

\newcommand{\ccdot}{\!\cdot\!}

\newcommand{\ttimes}{{\mathsmaller \otimes}}

\newcommand{\B}{{\mathcal B}}
\renewcommand{\H}{{\mathcal H}}

\newcommand{\Z}{{\mathsf Z}}

\newcommand{\Q}{{\mathsf Q}}
\newcommand\nt[1]{\textcolor{red}{{\small{#1}}}}

\newcommand{\T}{{\mathcal{T}}}

\newcommand{\psl}{ \mathsf{PSL}_2 (\Z)    }
\newcommand{\ppsl}{ \mathsf{PPSL}_2 (\Z)    }

\definecolor{green}{RGB}{117, 165, 50}

\newcommand{\pimcg}{\textcolor{magenta}{\mathbf{\Pi MG}}}
\newcommand{\omg}{\textcolor{blue}{\mathbf{\Omega MG}}}
\newcommand{\oomg}{\textcolor{blue}{\mathsf{OMG}}}
\newcommand{\pmg}{\textcolor{magenta}{\mathsf{PMG}}}

\newtheorem{theorem}{Theorem}[section] 
\newtheorem{lemma}[theorem]{Lemma}     
\newtheorem{corollary}[theorem]{Corollary}
\newtheorem{proposition}[theorem]{Proposition}
\theoremstyle{definition}
\newtheorem{definition}[theorem]{Definition}
\newtheorem{example}[theorem]{Example}
\newtheorem*{remark}{Remark}
\definecolor{shadecolor}{rgb}{.9, .9, .9}
\usepackage{framed}

\newcommand{\sherh}[1]{\fboxsep=1pt\setlength{\fboxrule}{0pt}
	\begin{center}
		\fbox{\colorbox{green}{
				\begin{minipage}[t]{14.5cm}
					#1
				\end{minipage}
			}
		}
\end{center}}
\renewcommand{\sherh}[1]{}

\title{Mapping class groupoids and Thompson's groups}
\author{
	Mustafa Topkara\footnote{
		{Department of Mathematics, Mimar Sinan Fine Arts University.}
		{Cumhuriyet Mah. Silahşör Cad. No:71, Şişli,}
		{\.{I}stanbul, Turkey}
		} 
		\and
	A. Muhammed Uluda\u{g}\footnote{
		{Department of Mathematics, Galatasaray University.}
		{\c{C}{\i}ra\u{g}an Cad. No. 36, 34349 Be\c{s}ikta\c{s},}
		{\.{I}stanbul, Turkey}}\and
	\"Ozge \"Ulkem$^\dag$\footnote{
        {Institute of Mathematics, Academia Sinica.  Astronomy-Mathematics
Building, Taipei 10617, Taiwan}}\and
	Ayberk Zeytin$^\dag$
}

\begin{document}
	\tikzset{->-/.style={decoration={
				markings,
				mark=at position #1 with {\arrow{>}}},postaction={decorate}}}
	
	\maketitle
	
	
	\begin{abstract}
		We introduce a groupoid $\pimcg$, called the fundamental modular groupoid, which is a variant of Penner's mapping class groupoid. We study how it relates to the surface mapping class groups and Thompson's group $\mathsf T$. 
We also introduce larger groupoid $\omg$, which is related to outer automorphisms of free groups and Thompson's group $\mathsf V$ in a similar manner.				
	\end{abstract}
	
		
	\tableofcontents

	\section{Introduction}
	\label{sec:introduction}
	Let $S=S^g_n$ be an oriented surface of genus $g$ and with $n\geq 1$ punctures. We assume that $S$ is of finite type ($g, n<\infty$)  but this hypothesis will be relaxed later.
	
	Our aim in this paper is to introduce the (disconnected) groupoid $\pimcg$, called the {\it fundamental modular groupoid} and study its connections to 
	\begin{itemize} 
		\item Penner's mapping class groupoid and {${\mod(S)}$}   (mapping class group of $S$) 
		\item {${\mathsf T}$}  (Thompson's group under its guise as the group of partial automorphisms of the infinite binary planar tree).
	\end{itemize}
	Roughly speaking, the object set of $\pimcg$ is the set $\mathrm{MGR}_\ell$ of certain ribbon graphs $\G$ with an enumeration of their edges (called a labeling). Morphisms of $\pimcg$ consists of flip-induced isomorphisms between the fundamental groupoids of these ribbon graphs.
	
	Isotropy groups of $\pimcg$ yield a new class of ``{mapping class-Thompson hybrid groups}''. One has the diagram
	\begin{equation}
	\begin{tikzcd}
\pmg \arrow[r,hook] \arrow[d, two heads]
& \pimcg \arrow[d, two heads] \\
\mod(S)  \arrow[r,hook]
& \mcg_\ell(S)
\end{tikzcd}
	\end{equation}
	where 
	$\mcg_\ell(S)$ is a variation Penner's mapping class groupoid of $S$; and $\pmg$ is the isotropy group of an appropriate connected component of $\pimcg$.

	In Section 2, we review two  versions of mapping class groupoids (due to Penner, \cite{universal}) $\mcg_1$ and $\mcg_2$ whose objects are triangulations with a distinguished oriented edge (d.o.e) and spines with a d.o.e, respectively.  These two groupoids are canonically isomorphic via duality.
	
	In Section 3, we introduce a third version of the mapping class groupoid $\mcg_\ell(S)$ (which is not isomorphic to the previous two versions). This is a variation of another version (again due to Penner,  \cite{decorated}) of the mapping class groupoid in the sense that the objects are graphs that are labeled via enumerations of their half-edges, and not just edges as in Penner's version. We then introduce the groupoid $\pimcg(S)$, whose morphisms are same as $\mcg_\ell(S)$ and whose isomorphisms are fundamental groups of labeled trivalent ribbon graphs (called {\it modular graphs} in the paper). This latter groupoid can be defined for arbitrary labeled modular graphs (i.e. not only for those graphs which are embedded as spines in a surface of finite type).
	
	Section 4 is devoted to the study of this larger groupoid $\pimcg$. In Section 5 we show that the isotropy group (denoted $\pmg(\F)$) of the connected component of $\pimcg$ containing the Farey tree $\F$, is isomorphic to the near automorphism group of $\F$, with the canonical surjection to the partial automorphism group $\aut^\infty(\F)$. This latter group is known to be isomorphic to Thompson's group $\mathsf T$ and is a subgroup of Neretin's spheromorphism group; denoted $\aut^\infty(\F^c)$ here. (\cite{yves}, \cite{vlad})

In Section 6, we initiate the study of  the connected components of $\pimcg$ containing the bush-like modular graphs, and we finish by stating a conjecture about their isotropy groups.	
	
Our construction of $\pimcg$ allows us to enrich the groupoid with some extra morphisms, called shuffles, and construct the extension
$\omg$ of $\pimcg$. Denoting its isotropy group of the connected component of  $\F$ by $\oomg(\F)$, we show that $\oomg(\F)$ canonically surjects onto Thompson's group $\mathsf V$. When  $\G$ is finite, $\oomg(\G)$ is a finite extension of the group 
$\out(\mathsf F)$ of outer automorphisms of the free group $\mathsf F\simeq\pi_1(\G)$ of finite rank. 
There is a considerable literature dedicated to the groups $\out(\mathsf F)$, see~\cite{vogtmann}.
In our setting, Thompson's group $\mathsf V$ appears as an analogue of $\out(\mathsf F)$, where $\mathsf F$ is a free group. Moreover, the groups $\oomg(\G)$ with infinite $\G$ provide generalizations of $\out(\mathsf F)$. Is there a cell complex, akin to  Culler-Vogtmann's outer space, on which $\oomg(\G)$ acts?  In any case, it is of interest to find presentations of $\oomg(\G)$ (for finite and infinite $\G$).

	
	%
	\paragraph{Notation.} We use sans serif fonts $\Z, \mathsf T \dots$ to denote groups, bold letters $\mathbf X$, $\mathbf S\dots$ to denote groupoids, script fonts $\mathscr C$, $\mathscr D\dots$ to denote general categories and finally calligraphic letters $\mathcal F, \G  \dots $ to denote graphs. 
	
	We define and use the following notations in the sequel:\\
	
	\bigskip
	\noindent
	
	\hspace{-10mm}
	{\small
		\begin{tabular}{lll}
			Symbol&Name&Page\\ \hline 
			$S=S_n^g$& Surface of genus $g$ with $n>0$ punctures & \pageref{Sng}\\
			$\mod(S)$&Mapping class group of $S$&\pageref{mcgsng}\\
			$\mathrm{TRN}(S)$&Set of isotopy classes of ideal triangulations of $S$&\pageref{sicidts}\\
			$\overrightarrow{\mathrm{TRN}}(S)$&
			\!\!\!\begin{tabular}{l} Set of isotopy classes of ideal triangulations of $S$\\ with a distinguished oriented edge\end{tabular}
			&\pageref{sicidtsdoe}\\
			$\mathrm{Isot}(X,x)$&
			The isotropy group of object $x$ in groupoid $X$.
			&\pageref{isotrrr}\\
			$\mcg_1(S)$&Mapping class groupoid of $S$, triangulation version&\pageref{mcgtv}\\
			$|\overrightarrow{\mathrm{TRN}}(S)|$&
			Set of combinatorial types of elements of $\overrightarrow{\mathrm{TRN}}(S)$&\pageref{sctetrn}\\
			$\mathrm{SPN}(S)$&Set of isotopy classes of spines of $S$&\pageref{sicsSng}\\
			$\overrightarrow{\mathrm{SPN}}(S)$&
			\!\!\!\begin{tabular}{l} Set of isotopy classes of spines of $S$ with a distinguished \\oriented edge\end{tabular}
			&\pageref{sicsSngdoe}\\
			$|\overrightarrow{\mathrm{SPN}}(S)|$&
			Set of combinatorial types of spines of $\overrightarrow{\mathrm{SPN}}(S)$&\pageref{spnsng}\\
			$\mcg_2(S)$&Mapping class groupoid of $S$, spine version&\pageref{mcgrsv}\\
			${\mathrm{SPN}^*}(S)$& Set of isotopy classes of pairs $[(\G , f)\hookrightarrow S]$&\pageref{spnstar}\\
			$\G^{*}$& Non-Hausdorff model of the topological graph $\G$&\pageref{gstar}\\
			$\mathrm{MGR}(S)$&Set of modular graphs obtained from spines&\pageref{mgrspine}\\
			$\overrightarrow{\mathrm{MGR}}(S)$&Set of modular graphs with a distinguished non-oriented edge&\pageref{mgrspinedue}\\
			$\mathrm{SPN}_\ell(S)$&Set labeled modular graph spines of $S$&\pageref{lmspine}\\
			$\mcg_\ell(S)$&Mapping class groupoid of $S$, labeled spine version&\pageref{mcgrell}\\
			${\Pi}_1^{\G^*}$& Fundamental groupoid of $\G^{*}$ &\pageref{tildapi1} \\
			$\widetilde{\Pi}_1^{\G^*}$& Graph fundamental groupoid of $\G^{*}$&\pageref{graphpi1}\\
			$\pimcg(S)$&Fundamental modular groupoid of $S$&\pageref{fmgrs}\\
			$\pimcg$&Fundamental modular groupoid&\pageref{fmgr}\\
			$\pmg(\G)$& Isotropy group of $\G$ in $\pimcg$ &\pageref{modgdef}\\
			$\mathsf T\simeq \mathsf{PPL}_2(\mathsf Z), \,\mathsf F,\mathsf V$& Thompson's groups&\pageref{thmpgr}\\
			$\mathsf{PPSL}_2(\mathsf Z) \simeq\mathsf T $&Group of pw-$\psl$ homeomorphisms of $S^1$ with $\Q$-brake points &\pageref{thmpgr}\\
			$\F$& Farey tree &\pageref{fareytreee}\\
			$\aut^\infty(\F)\simeq \mathsf T$& Group of partial $\F$-automorphisms &\pageref{germsof}\\
			$\aut^\infty(\F^c)$& Neretin's spheromorphism group&\pageref{nerretin}\\
			$\mathbf{NMGR}_\ell$& The groupoid of near isomorphisms of modular graphs&\pageref{nmgrl}\\
			$\mathbf{PMGR}_\ell$& The groupoid of partial isomorphisms of modular graphs&\pageref{pmgrl}\\
			$\omg$&Outer modular groupoid&\pageref{oomg}\\
			$\oomg(\G)$& Isotropy group of $\G$ in $\omg$ &\pageref{oomg}\\
$\mathbf{PUN}$& The groupoid of bijections of puncture sets of modular graphs&\pageref{purrpunn}\\
$\mathbf{PUR}$& The puncture preserving subgroupoid of $\pimcg$ &\pageref{purrpunn}
	\end{tabular}}

	\section{Two versions of Mapping Class Groupoids}
	\label{sec:mcg}
	
	Here we recall the definitions of the mapping class groupoids introduced by Penner \cite{universal} as exposed by Mosher \cite{mosher}. Let $S=S_n^g$ be an orientable surface of genus $g$ and with $n\geq 1$ punctures.\label{Sng} Denote by $\mod(S)$ 
	the mapping class group of $S$ which is defined to be the group of orientation-preserving homeomorphisms of $S$ modulo isotopies \label{mcgsng} (where homeomorphisms are allowed to exchange punctures).
	
	In this section, we assume that $S$ is of finite type, i.e. $g,n<\infty$. Our aim is to eventually relax this assumption. Moreover, we assume that $S$ has negative Euler characteristic, i.e.
	$\chi(S)=2-2g-n<0$.
	
	\begin{definition}\label{ideal} 
		{\rm An {\it ideal arc} of $S$ is  an embedded arc connecting
			punctures in $S$, which is not homotopic to a point relative to punctures. 
			An {\it ideal cell decomposition of $S$} is a collection of ideal arcs so that each region complementary to arcs is a polygon
			with vertices among the punctures. A maximal ideal cell decomposition is called an {\it ideal triangulation}. }
	\end{definition}
	
	Denote the set of ideal triangulations of $S$ modulo isotopy as
	\begin{equation*}
		{\mathrm{TRN}}(S):=\{ \mbox{ideal triangulations of } S \} /\mbox{isotopies}
	\end{equation*}
	\label{sicidts}
	and those with a distinguished oriented edge (d.o.e.) modulo isotopy as \label{sicidtsdoe}
	\begin{equation*}
		\overrightarrow{\mathrm{TRN}}(S):=\{ \mbox{ideal triangulations of } S \mbox{ with a d.o.e.}\} /\mbox{isotopies.}
	\end{equation*}
	Note that $\mathrm{TRN}(S)$ and $\overrightarrow{\mathrm{TRN}}(S)$ are countably infinite discrete sets.
	
	From now on, `{triangulation}' means `{ideal triangulation}'. Due to the existence of triangulations with automorphisms, the obvious $\mod(S)$-action on $\mathrm{TRN}(S)$ is not free.
	On the other hand, the  $\mod(S)$-action on $\overrightarrow{\mathrm{TRN}}(S)$ is free \cite{mosher}.
	
	Suppose that $J$ is a set admitting a free action of a group $\mathsf H$ from the left. We can associate a groupoid  $\llbracket\mathsf H\backslash J\rrbracket$ to this action, whose 
	objects are $\mathsf H$-orbits of $J$ and morphisms are $\mathsf H$-orbits of $J\times J$, i.e.
	\begin{align}
		\obj(\llbracket\mathsf H\backslash J\rrbracket)&:=\mathsf H\backslash J=\{\mathsf Hx \,:\, x\in J\}, \notag\\
		\mor_{\llbracket\mathsf H\backslash  J\rrbracket}(\mathsf Hx,\mathsf Hy)&:=\{\mathsf H(x',y')\, |\, \mathsf H x'= \mathsf Hx, \,  \mathsf H y'= \mathsf Hy\}.\label{eq:ActCatMorSet}
	\end{align}
	The composition $\mathsf H(x,y)\circ \mathsf H(y,z)$ is defined as $\mathsf H(x,z)$.
	By construction, the isotropy groups of this groupoid are isomorphic to $ \mathsf H$. (Recall that, for any groupoid $\mathbf X$, and an object $x$ of $\mathbf X$, the set of morphisms from $x$ to $x$ constitute a group, which is called the {\it isotropy group of $\mathbf X$ at $x$} and denoted \label{isotrrr}
	$\mathrm{Isot}(\mathbf X,x)=\mor_{\mathbf X}(x,x)$).
	\begin{definition} The {\it mapping class groupoid (first version)}  \label{mcgtv} $\mcg_1(S)$ is the groupoid associated to the free 
		$\mod(S)$-action on $\overrightarrow{\mathrm{TRN}}(S)$, i.e.
		\begin{equation}
			\mcg_1(S):=\llbracket\mod(S)\backslash\overrightarrow{\mathrm{TRN}}(S)\rrbracket.
		\end{equation}
	\end{definition}
	Thus, the isotropy group of any element in $\mcg_1(S)$ is isomorphic to $\mod(S)$.
	The object set of $\mcg_1(S)$ is the set 
	$$
	|\overrightarrow{\mathrm{TRN}}(S)|:=\mod(S)\backslash \overrightarrow{\mathrm{TRN}}(S).
	$$ 
	Whenever two triangulations of $S$ are combinatorially equivalent, i.e. are isomorphic as cell-complexes, there exists a homeomorphisms of $S$ which carries one to the other. Hence, the set $|\overrightarrow{\mathrm{TRN}}(S)|$
	corresponds to `{combinatorial types}' of triangulations of $S$ with a d.o.e. \label{sctetrn} This set is finite if $S$ is of finite type.
	Morphisms of $\mcg_1(S)$ are the $\mod(S)$-orbits 
	\begin{align*}
		\mor_{\mcg_1}(|\mathcal T_1, \vec{e_1}|,&|\mathcal T_2, \vec{e_2}|)=\\
		&\Bigl\{ \mod(S)([\mathcal T_1^\prime, \vec{e_1^\prime}], [\mathcal T_2^\prime, \vec{e_2^\prime}]) \,\, \Bigl |\Bigr. \,\,
		|\mathcal T_1^\prime, \vec{e_1^\prime}|=
		|\mathcal T_1, \vec{e_1}|, \,
		|\mathcal T_2^\prime, \vec{e_2^\prime}|=
		|\mathcal T_2, \vec{e_2}|
		\Bigr\},
	\end{align*}
	where $[\mathcal T, \vec{e}]\in \overrightarrow{\mathrm{TRN}}(S)$ denotes the isotopy type of the triangulation-with-a-d.o.e. 
	$(\mathcal T, \vec{e})$,  and $|\mathcal T, \vec{e}|\in \obj(\mcg_1)$ denotes its {{`}{combinatorial type}'}.
	
	\subsection{Dual picture: spines}
	A (topological) {\it graph} $\G $ is a one-dimensional CW-complex comprised of vertices $V(\G )$ and edges
	$E(\G )$; a {\it ribbon graph} or ({\it fat graph}) is a topological graph together with a cyclic ordering of (half-)edges emanating from each vertex.

	Let $\G \hookrightarrow S$ be an embedding of a topological graph $\G $. We say that $\G \hookrightarrow S$ is a {\it spine} of $S$ 
	if it is dual to an ideal cell decomposition $\mathcal T$ of $S$. 
	In other words, the image of $\G  \hookrightarrow S$ is obtained by putting the vertices of $\G $ inside the cells of $\mathcal T$, 
	picking some paths connecting these vertices by passing through a unique ideal arc of the cell decomposition, and finally by identifying the edges of $\G $ by these paths.

	Every spine $\G \hookrightarrow S$ is a strong deformation retract of $S$. 
	Note that $\G $ acquires  a natural ribbon graph structure from the positive orientation of $S$ via the embedding $\G \hookrightarrow S$. Note also that a spine dual to an ideal triangulation is a trivalent ribbon graph. 
	
	Denote the set of isotopy classes of trivalent ribbon graph spines of $S$ as \label{sicsSng}
	\begin{equation}
		\mathrm{SPN}(S):=\bigl\{\varphi:\G \hookrightarrow S \, |\, \varphi \mbox{ is a spine}\bigr\}/\mbox{isotopy}
	\end{equation}
	and denote the set of isotopy classes of trivalent ribbon graph spines with a d.o.e.  as \label{sicsSngdoe}
	\begin{equation}
		\overrightarrow{\mathrm{SPN}}(S):=\bigl\{\varphi: (\G , \vec{e}) \hookrightarrow S \, \, |\, \varphi  \mbox{ is a spine with a d.o.e.}\bigr\}/\mbox{isotopy}.
	\end{equation}
	Since we assume $S$ to be of finite type, $\mathrm{SPN}(S)$ and $\overrightarrow{\mathrm{SPN}}(S)$ are at most countably infinite discrete sets. Now, $\mod(S)$ acts by post-composition on $\mathrm{SPN}(S)$. This action is not free; however, it is free  on $\overrightarrow{\mathrm{SPN}}(S)$. Therefore we have the associated groupoid $\mcg_2(S)=\llbracket \mod(S)\backslash\overrightarrow{\mathrm{SPN}}(S)\rrbracket$, \label{mcgrsv} whose object set is the set of combinatorial types of trivalent ribbon graph spines with d.o.e. \label{spnsng}
	$$
	|\overrightarrow{\mathrm{SPN}}(S)|:=\mod(S)\backslash\overrightarrow{\mathrm{SPN}}(S).
	$$
	Whenever two spines of $S$ are isomorphic as ribbon graphs, there exists a
	homeomorphism of $S$ which carries one to the other. Hence, the set $|\overrightarrow{\mathrm{SPN}}(S)|$ corresponds to combinatorial types of spines of $S$ with a d.o.e.
	This is a finite set if $S$ is of finite type.

	Let $[(\G , \vec{e})\hookrightarrow S]\in \overrightarrow{\mathrm{SPN}}(S)$ denote the isotopy class of the spine-with-a-d.o.e. 
	$((\G , \vec{e})\hookrightarrow S)$ and let $|\G , \vec{e}|=\mod(S)[(\G , \vec{e})\hookrightarrow S]$ denote its combinatorial type.
	
	The set of morphisms from an object $|\G_1 , \vec{e_1}|$ to the object $|\G_2 , \vec{e_2}|$ in the category $\llbracket\mod(S)\backslash\overrightarrow{\mathrm{SPN}}(S)\rrbracket$ is the set of orbits of the form
	$$ \mod(S)\left([(\G_1^\prime , \vec{e_1^\prime})\hookrightarrow S], [(\G_2^\prime, \vec{e_2^\prime})\hookrightarrow S]\right)$$
	where $|\G_1',e_1'|=|\G_1,e_1|$ and $|\G_2',e_2'|=|\G_2,e_2|$,
	 as defined in Equation \ref{eq:ActCatMorSet}.

	\begin{figure}[h!]
		\begin{center}
			\begin{tikzpicture}[line cap=round,line join=round,>=triangle 45,x=.5cm,y=.5cm]
				\draw [line width=1pt, style=dashed,->-=.5] (-1,0)-- (2,0);
				\draw [line width=.3pt, style=dashed] (-2,2)-- (-1,0);
				\draw [line width=.3pt, style=dashed] (-1,0)-- (-2,-2);
				\draw [line width=.3pt, style=dashed] (2,0)-- (3,2);
				\draw [line width=.3pt, style=dashed] (2,0)-- (3,-2);
				\draw [line width=1pt] (-2.5,0)-- (.5,2);
				\draw [line width=1pt] (.5,2)-- (3.5,0);
				\draw [line width=1pt] (-2.5,0)-- (.5,-2);
				\draw [line width=1pt] (.5,-2)-- (3.5,0);
				\draw [line width=1pt,->-=.5] (.5,-2)-- (.5,2);

				\begin{scriptsize}
					\draw [fill=red] (-1,0) circle (1.5pt);
					\draw [fill=red] (2,0) circle (1.5pt);
					\draw [fill=red] (-2,2) circle (1.5pt);
					\draw [fill=red] (-2,-2) circle (1.5pt);
					\draw [fill=red] (3,2) circle (1.5pt);
					\draw [fill=red] (3,-2) circle (1.5pt);
					\draw [fill=white] (.5,2) circle (2pt);
					\draw [fill=white] (.5,-2) circle (2pt);
					\draw [fill=white] (-2.5,0) circle (2pt);
					\draw [fill=white] (3.5,0) circle (2pt);
					
				\end{scriptsize}
			\end{tikzpicture}
		\end{center}
		\caption{How to choose the d.o.e. in the graph dual to a triangulation. Unbroken arcs constitute the triangulation, dashed lines constitute the dual spine.}
		\label{dualflip}
	\end{figure}

	By duality, we have bijections
	\begin{equation*}
		\mathrm{SPN}(S)\longleftrightarrow \mathrm{TRN}(S), \quad 
		\overrightarrow{\mathrm{SPN}}(S)\longleftrightarrow\overrightarrow{\mathrm{TRN}}(S)
	\end{equation*}
	which are compatible with the $\mod(S)$-action.  Hence we have the duality isomorphism
	\begin{equation*}
		\mcg_1(S)=\llbracket\mod(S)\backslash\overrightarrow{\mathrm{TRN}}(S)\rrbracket\simeq \mcg_2(S)=\llbracket\mod(S)\backslash\overrightarrow{\mathrm{SPN}}(S)\rrbracket.
	\end{equation*}

	\subsection{Flips}
	Given an ideal triangulation $\mathcal T$ of $S$ with an arc $f$ of $\mathcal T$, if $f$ separates two distinct triangles, one obtains a new ideal triangulation by replacing the given arc by the other diagonal of the quadrilateral formed by the neighboring arcs as depicted in Figure~\ref{flipfigure}. This operation is well-defined on the set of isotopy classes and is called a {\it flip}. In case one has the same triangle on both sides of $f$, as in Figure~\ref{degenerate}, then by definition the flip operation leaves the triangulation unmodified. 
	This defines an involution of the set ${\mathrm{TRN}^*}(S)$ of pairs of isotopy classes $[\mathcal T,f]$, where 
	$\mathcal T\in {\mathrm{TRN}}(S)$ and $f$ is an arc of $\mathcal T$.
	By duality there is an operation on trivalent ribbon graph spines, which replaces an H-shaped part of a spine $\G $ of $S$ by an I-shaped graph, thereby producing a new spine $\G '$ of $S$ as depicted in Figure~\ref{flipfigure}. This operation is also called a {\it flip}, which defines the dual involution \label{spnstar}
	$$
	\phi: {\mathrm{SPN}^*}(S) \to {\mathrm{SPN}^*}(S),
	$$ 
	${\mathrm{SPN}^*}(S)$ being the set of isotopy classes of pairs $[(\G , f)\hookrightarrow S]$, where 
	$\G  \hookrightarrow S$ is a spine and $f$ is an edge of $\G $.

	\begin{figure}[h!]
		\begin{center}
			\begin{tikzpicture}[line cap=round,line join=round,>=triangle 45,x=.5cm,y=.5cm]
				\draw [line width=.3pt, style=dashed] (-1,0)-- (2,0);
				\draw [line width=.3pt, style=dashed] (-2,2)-- (-1,0);
				\draw [line width=.3pt, style=dashed] (-1,0)-- (-2,-2);
				\draw [line width=.3pt, style=dashed] (2,0)-- (3,2);
				\draw [line width=.3pt, style=dashed] (2,0)-- (3,-2);
				\draw [line width=1pt] (-2.5,0)-- (.5,2);
				\draw [line width=1pt] (.5,2)-- (3.5,0);
				\draw [line width=1pt] (-2.5,0)-- (.5,-2);
				\draw [line width=1pt] (.5,-2)-- (3.5,0);
				\draw [line width=1pt] (.5,-2)-- (.5,2);

				\begin{scriptsize}
					\draw [fill=black] (-1,0) circle (1.5pt);
					\draw [fill=black] (2,0) circle (1.5pt);
					\draw [fill=red] (-2,2) circle (1.5pt);
					\draw [fill=red] (-2,-2) circle (1.5pt);
					\draw [fill=red] (3,2) circle (1.5pt);
					\draw [fill=red] (3,-2) circle (1.5pt);
					\draw [fill=white] (.5,2) circle (2pt);
					\draw [fill=white] (.5,-2) circle (2pt);
					\draw [fill=white] (-2.5,0) circle (2pt);
					\draw [fill=white] (3.5,0) circle (2pt);
					
				\end{scriptsize}
			\end{tikzpicture}
			\hspace{1cm}\raisebox{.9cm}{\huge $\rightsquigarrow$} \hspace{1cm}
			\begin{tikzpicture}[line cap=round,line join=round,>=triangle 45,x=.4cm,y=.4cm]
				\draw [line width=.3pt, style=dashed] (0,-.5)-- (0,1.5);
				\draw [line width=.3pt, style=dashed] (-3,3)-- (0,1.5);
				\draw [line width=.3pt, style=dashed] (0,1.5)-- (3,3);
				\draw [line width=.3pt, style=dashed] (0,-.5)-- (-3,-2);
				\draw [line width=.3pt, style=dashed] (0,-.5)-- (3,-2);
				\draw [line width=1pt] (-4,0.5)-- (0,-2);
				\draw [line width=1pt] (-4,0.5)-- (0,3);
				\draw [line width=1pt] (4,0.5)-- (0,-2);
				\draw [line width=1pt] (4,0.5)-- (0,3);
				\draw [line width=1pt] (4,0.5)-- (-4,0.5);
				
				\begin{scriptsize}
					\draw [fill=black] (0,-.5) circle (1.5pt);
					\draw [fill=black] (0,1.5) circle (1.5pt);
					\draw [fill=red] (-3,3) circle (1.5pt);
					\draw [fill=red] (3,3) circle (1.5pt);
					\draw [fill=red] (-3,-2) circle (1.5pt);
					\draw [fill=red] (3,-2) circle (1.5pt);
					\draw [fill=white] (-4,0.5) circle (2pt);
					\draw [fill=white] (4,0.5) circle (2pt);
					\draw [fill=white] (0,3) circle (2pt);
					\draw [fill=white] (0,-2) circle (2pt);
					
				\end{scriptsize}
			\end{tikzpicture}
		\end{center}
		
		\caption{A triangulation flip and its effect on the dual trivalent ribbon graph spine (shown by dashed lines). Hollow nodes denotes punctures, and thick lines denotes ideal arcs.}
		\label{flipfigure}
	\end{figure}
	If $f$ is an edge of  $[\G  \hookrightarrow S]\in {\mathrm{SPN}}(S)$ and 
	$\phi([\G  \hookrightarrow S],f)=([\G ' \hookrightarrow S],f')$ then denote the result of a flip on the edge $f$ of $\G$ as
	\begin{equation}
		\phi_f([\G  \hookrightarrow S]):= 
		[\G ' \hookrightarrow S].
	\end{equation}
	
	\begin{figure}[h!]
		\begin{center}
			
			\begin{tikzpicture}[line cap=round,line join=round,>=triangle 45,x=.4cm,y=.4cm]
				\draw [line width=1pt, style=dashed] (-8,0.5)-- (-3,0.5);
				\draw [line width=1pt] (2,0.5)-- (-1.3,0.5);
				
				\begin{scriptsize}
					\draw (-1.7,0.5) ellipse (1.5cm and .9cm);
					\draw [style=dashed] (-1.3,0.5) ellipse (.7cm and .5cm);
					
					\draw [fill=white] (-1.3,0.5) circle (2pt);
					\draw [fill=black] (-3,0.5) circle (2pt);
					\draw [fill=white] (2,0.5) circle (2pt);
				\end{scriptsize}
			\end{tikzpicture}
		\end{center}
		
		\caption{A degenerate triangulation and its dual spine.}
		\label{degenerate}
	\end{figure}
	
	Since $\phi$ commutes with the $\mod(S)$-action it  induces an involution
	\begin{equation}\label{fflip}
		\phi: |{\mathrm{SPN}^*}(S)|\to |{\mathrm{SPN}^*}(S)|,
	\end{equation}
	where $|{\mathrm{SPN}^*}(S)|$ is the set of combinatorial types of spines with a choice of an edge (corresponding to the edge the flip is to be/has been applied to; note that such a choice has no relation with d.o.e.'s). Moreover, if $\phi(|\G,f|)=|\G',f'|$ then there is an obvious pairing between the edge sets of $\G $ and $\G '$
	\begin{equation}
		\hat \phi_f: E(\G ) \longrightarrow E(\G ')
	\end{equation}
	such that $\hat \phi_f(f)=f'$.

	We have the following combinatorial result:
	\begin{lemma}\label{whitehead} (Whitehead~\cite{decorated}, see also \cite{hatcher}.)
		Any two elements of ${\mathrm{TRN}}(S)$ are connected via a finite sequence of flips. 
		Equivalently, any two elements of ${\mathrm{SPN}}(S)$ are connected via a finite sequence of flips.
	\end{lemma}
	Note that this lemma is obviously not valid for surfaces of infinite type; e.g. spines in Figure~\ref{cccp} are not connected by any finite sequence of flips.

	\paragraph{Flips of spines with a d.o.e.}
	The flip of a trivalent ribbon graph spine $[(\G , \vec{e}) \hookrightarrow S]$ with a d.o.e. 
	is a flip of the underlying spine which maps the d.o.e.s in the obvious manner: if d.o.e. is not the flipped edge, then it is simply preserved; if it is, then rotated counterclockwise 
	(if the d.o.e. is the loop edge in Figure~\ref{degenerate}, then we leave the orientation as it is).
	This defines a map 
	\begin{equation}
		\Phi: \overrightarrow{\mathrm{SPN}}^*(S) \to \overrightarrow{\mathrm{SPN}}^*(S),
	\end{equation}
	which is of order four.
	Observe that the $\mod(S)$-orbit
	\begin{equation}
		\mod(S)\bigl ([(\G , \vec{e}) \hookrightarrow S], 
		[(\G', \vec{e'}) \hookrightarrow S]\bigr)
	\end{equation}
	is an element of $\mcg_2(S)$, 
	where 
	$[\G'\hookrightarrow S]=
	\phi_f([\G \hookrightarrow S])
	$
	and $\vec{e'}$ is the edge ${\hat{\phi}_f(e)}$ 
	with orientation obtained by the counterclockwise rotation described above.

	By Whitehead's lemma $\mcg_2(S)$ is generated by these elements, again called `{flips}' (or `{elementary moves}' in~\cite{mosher}) and {`}{d.o.e. moves}{'}, i.e. orbits
	\begin{equation}
		\mod(S)\bigl ([(\G , \vec{e_1}) \hookrightarrow S], [(\G , \vec{e_2}), \hookrightarrow S]\bigr),
	\end{equation}
	where the two embeddings of $\G$ are the same, and only the choice of the d.o.e. differs
	(see \cite{mosher}, where they are called `{relabelings}'). In fact, with the exception $S_0^3$ and $S_1^1$, the groupoid $\mcg_2(S)$ is already generated by flips, see \cite{decorated}. In particular, d.o.e. moves can be expressed in terms of flips.
	This generation is not free, i.e. there are some relations obeyed by flips. An example is the famous pentagon relation, see Figure~\ref{fig12:pentagram}.
	
		\begin{figure}[h]
	\begin{center}
	\begin{tikzpicture}[auto,scale=0.75]
		\def\Radius{2cm}
		
		
		\node at (0,-3){(1)};
		\coordinate (P0) at (90:\Radius);
		\foreach \i in {1,...,5}
		{
			\def\angle{90+\i*72}
			\coordinate (P\i) at (\angle:\Radius) {};
			\pgfmathtruncatemacro{\j}{\i-1}
			\draw  
			[line1, name path global/.expanded=L\i]
			(P\j)-- 
			(P\i);
			\coordinate (Q\i) at (\angle-32:1.2*\Radius) {};
			\node[anchor at angle=-\angle] at (P\i){
			};
			\node[anchor at angle=-\angle-32] at (Q\i){
			};
			\filldraw[black] (P\i) circle (1pt);
			\filldraw (Q\i) circle (1pt);
		}
		\draw[line1,name path global/.expanded=L6] (P1) -- (P3);
		\draw[line1,name path global/.expanded=L7] (P1) -- (P4);

		\coordinate (R1) at (\bctr{P1}{P2}{P3});
		\coordinate (R2) at (\bctr{P1}{P4}{P3});
		\coordinate (R3) at (\bctr{P1}{P4}{P5});

		\foreach [count=\i] \q/\r in {1/3, 2/1, 3/1, 4/2, 5/3}
		{\pgfmathtruncatemacro{\j}{\i-1}
			\TwinEdges{Q\i}{R\r}{P\j}{P\i}
		}
		\TwinEdgesBold{R1}{R2}{P1}{P3}
		\TwinEdges{R2}{R3}{P1}{P4}
		\path (R1) --  node{$ e $}(R2);
		\path (R2) --  node[anchor=west]{$ f $}(R3);
		\foreach \r in {1,2,3}
		{\filldraw (R\r) circle (2pt);
			\node at (R\r){
			};
		}
		\foreach \i in {1,2,3,4,5}
		{\filldraw[lightgray] (P\i) circle (2pt);
			\filldraw (Q\i) circle (2pt);}
		
		
		\begin{scope}[shift={(7,0)}]
			\node at (0,-3){(2)};
			\coordinate (P0) at (90:\Radius);
			\foreach \i in {1,...,5}
			{
				\def\angle{90+\i*72}
				\coordinate (P\i) at (\angle:\Radius) {};
				\pgfmathtruncatemacro{\j}{\i-1}
				\draw  
				[line1, name path global/.expanded=L\i]
				(P\j)-- 
				(P\i);
				\coordinate (Q\i) at (\angle-32:1.2*\Radius) {};
				\node[anchor at angle=-\angle] at (P\i){
				};
				\node[anchor at angle=-\angle-32] at (Q\i){
				};
				\filldraw[black] (P\i) circle (1pt);
				\filldraw (Q\i) circle (1pt);
			}
			\draw[line1,name path global/.expanded=L6] (P2) -- (P4);
			\draw[line1,name path global/.expanded=L7] (P1) -- (P4);

			\coordinate (R1) at (\bctr{P4}{P2}{P3});
			\coordinate (R2) at (\bctr{P1}{P4}{P2});
			\coordinate (R3) at (\bctr{P1}{P4}{P5});

			\foreach [count=\i] \q/\r in {1/3, 2/2, 3/1, 4/1, 5/3}
			{\pgfmathtruncatemacro{\j}{\i-1}
				\TwinEdges{Q\i}{R\r}{P\j}{P\i}
			}
			\TwinEdges{R1}{R2}{P2}{P4}
			\TwinEdgesBold{R2}{R3}{P1}{P4}
			\path (R1) --  node{$ e $}(R2);
			\path (R2) --  node[anchor=west]{$ f $}(R3);
			\foreach \r in {1,2,3}
			{\filldraw (R\r) circle (2pt);
				\node at (R\r){
				};
			}
			\foreach \i in {1,2,3,4,5}
			{\filldraw[lightgray] (P\i) circle (2pt);
				\filldraw (Q\i) circle (2pt);}
		\end{scope}
		
		
		\begin{scope}[shift={(14,0)}]
			\node at (0,-3){(3)};
			\coordinate (P0) at (90:\Radius);
			\foreach \i in {1,...,5}
			{
				\def\angle{90+\i*72}
				\coordinate (P\i) at (\angle:\Radius) {};
				\pgfmathtruncatemacro{\j}{\i-1}
				\draw  
				[line1, name path global/.expanded=L\i]
				(P\j)-- 
				(P\i);
				\coordinate (Q\i) at (\angle-32:1.2*\Radius) {};
				\node[anchor at angle=-\angle] at (P\i){
				};
				\node[anchor at angle=-\angle-32] at (Q\i){
				};
				\filldraw[black] (P\i) circle (1pt);
				\filldraw (Q\i) circle (1pt);
			}
			\draw[line1,name path global/.expanded=L6] (P2) -- (P4);
			\draw[line1,name path global/.expanded=L7] (P2) -- (P5);

			\coordinate (R1) at (\bctr{P4}{P2}{P3});
			\coordinate (R2) at (\bctr{P1}{P5}{P2});
			\coordinate (R3) at (\bctr{P2}{P4}{P5});

			\foreach [count=\i] \q/\r in {1/2, 2/2, 3/1, 4/1, 5/3}
			{\pgfmathtruncatemacro{\j}{\i-1}
				\TwinEdges{Q\i}{R\r}{P\j}{P\i}
			}
			\TwinEdgesBold{R1}{R3}{P2}{P4}
			\TwinEdges{R3}{R2}{P2}{P5}
			\path (R1) --  node{$ e $}(R3);
			\path (R2) --  node[anchor=south]{$ f $}(R3);
			\foreach \r in {1,2,3}
			{\filldraw (R\r) circle (2pt);
				\node at (R\r){
				};
			}
			\foreach \i in {1,2,3,4,5}
			{\filldraw[lightgray] (P\i) circle (2pt);
				\filldraw (Q\i) circle (2pt);}
		\end{scope}
		
		
		\begin{scope}[shift={(0,-6)}]
			\node at (0,-3){(4)};
			\coordinate (P0) at (90:\Radius);
			\foreach \i in {1,...,5}
			{
				\def\angle{90+\i*72}
				\coordinate (P\i) at (\angle:\Radius) {};
				\pgfmathtruncatemacro{\j}{\i-1}
				\draw  
				[line1, name path global/.expanded=L\i]
				(P\j)-- 
				(P\i);
				\coordinate (Q\i) at (\angle-32:1.2*\Radius) {};
				\node[anchor at angle=-\angle] at (P\i){
				};
				\node[anchor at angle=-\angle-32] at (Q\i){
				};
				\filldraw[black] (P\i) circle (1pt);
				\filldraw (Q\i) circle (1pt);
			}
			\draw[line1,name path global/.expanded=L6] (P3) -- (P5);
			\draw[line1,name path global/.expanded=L7] (P2) -- (P5);

			\coordinate (R1) at (\bctr{P4}{P5}{P3});
			\coordinate (R2) at (\bctr{P1}{P5}{P2});
			\coordinate (R3) at (\bctr{P2}{P3}{P5});

			\foreach [count=\i] \q/\r in {1/2, 2/2, 3/3, 4/1, 5/1}
			{\pgfmathtruncatemacro{\j}{\i-1}
				\TwinEdges{Q\i}{R\r}{P\j}{P\i}
			}
			\TwinEdges{R1}{R3}{P3}{P5}
			\TwinEdgesBold{R3}{R2}{P2}{P5}
			\path (R1) --  node{$ e $}(R3);
			\path (R2) --  node[anchor=north]{$ f $}(R3);
			\foreach \r in {1,2,3}
			{\filldraw (R\r) circle (2pt);
				\node at (R\r){
				};
			}
			\foreach \i in {1,2,3,4,5}
			{\filldraw[lightgray] (P\i) circle (2pt);
				\filldraw (Q\i) circle (2pt);}
		\end{scope}
		
		
		\begin{scope}[shift={(7,-6)}]
			\node at (0,-3){(5)};
			\coordinate (P0) at (90:\Radius);
			\foreach \i in {1,...,5}
			{
				\def\angle{90+\i*72}
				\coordinate (P\i) at (\angle:\Radius) {};
				\pgfmathtruncatemacro{\j}{\i-1}
				\draw  
				[line1, name path global/.expanded=L\i]
				(P\j)-- 
				(P\i);
				\coordinate (Q\i) at (\angle-32:1.2*\Radius) {};
				\node[anchor at angle=-\angle] at (P\i){
				};
				\node[anchor at angle=-\angle-32] at (Q\i){
				};
				\filldraw[black] (P\i) circle (1pt);
				\filldraw (Q\i) circle (1pt);
			}
			\draw[line1,name path global/.expanded=L6] (P3) -- (P5);
			\draw[line1,name path global/.expanded=L7] (P1) -- (P3);

			\coordinate (R1) at (\bctr{P4}{P5}{P3});
			\coordinate (R2) at (\bctr{P1}{P3}{P5});
			\coordinate (R3) at (\bctr{P2}{P3}{P1});

			\foreach [count=\i] \q/\r in {1/2, 2/3, 3/3, 4/1, 5/1}
			{\pgfmathtruncatemacro{\j}{\i-1}
				\TwinEdges{Q\i}{R\r}{P\j}{P\i}
			}
			\TwinEdgesBold{R1}{R2}{P3}{P5}
			\TwinEdges{R2}{R3}{P1}{P3}
			\path (R1) --  node[anchor=south]
			{$ e $}(R2);
			\path (R2) --  node[anchor=170]{$ f $}(R3);
			\foreach \r in {1,2,3}
			{\filldraw (R\r) circle (2pt);
				\node at (R\r){
				};
			}
			\foreach \i in {1,2,3,4,5}
			{\filldraw[lightgray] (P\i) circle (2pt);
				\filldraw (Q\i) circle (2pt);}
		\end{scope}
		
		
		\begin{scope}[shift={(14,-6)}]
			\node at (0,-3){(6)};
			\coordinate (P0) at (90:\Radius);
			\foreach \i in {1,...,5}
			{
				\def\angle{90+\i*72}
				\coordinate (P\i) at (\angle:\Radius) {};
				\pgfmathtruncatemacro{\j}{\i-1}
				\draw  
				[line1, name path global/.expanded=L\i]
				(P\j)-- 
				(P\i);
				\coordinate (Q\i) at (\angle-32:1.2*\Radius) {};
				\node[anchor at angle=-\angle] at (P\i){
				};
				\node[anchor at angle=-\angle-32] at (Q\i){
				};
				\filldraw[black] (P\i) circle (1pt);
				\filldraw (Q\i) circle (1pt);
			}
			\draw[line1,name path global/.expanded=L6] (P1) -- (P4);
			\draw[line1,name path global/.expanded=L7] (P1) -- (P3);

			\coordinate (R1) at (\bctr{P4}{P5}{P1});
			\coordinate (R2) at (\bctr{P1}{P3}{P4});
			\coordinate (R3) at (\bctr{P2}{P3}{P1});

			\foreach [count=\i] \q/\r in {1/1, 2/3, 3/3, 4/2, 5/1}
			{\pgfmathtruncatemacro{\j}{\i-1}
				\TwinEdges{Q\i}{R\r}{P\j}{P\i}
			}
			\TwinEdges{R1}{R2}{P1}{P4}
			\TwinEdges{R2}{R3}{P1}{P3}
			\path (R1) --  node[anchor=20]
			{$ e $}(R2);
			\path (R2) --  node[anchor=120]{$ f $}(R3);
			\foreach \r in {1,2,3}
			{\filldraw (R\r) circle (2pt);
				\node at (R\r){
				};
			}
			\foreach \i in {1,2,3,4,5}
			{\filldraw[lightgray] (P\i) circle (2pt);
				\filldraw (Q\i) circle (2pt);}
		\end{scope}
	\end{tikzpicture}
	\end{center}
	\caption{At each step, the flip is applied on the highlighted edge. After five flips, the twin edges represented by $ e $ and $ f $ are transposed.}	
	\label{fig12:pentagram}
	\end{figure}
	
	Note that one may define the dual notion of the flip of a triangulation with a d.o.e. as well. To see how two flips with d.o.e.s are related, see Figure~\ref{dualflip}.
	
	\section{Fundamental modular groupoid of a surface} 
	Let $S^g_n$ be a surface. In the previous section, we have described two versions of mapping class groupoids:
	\begin{equation}
		\mcg_1(S)=\llbracket\mod(S)\backslash\overrightarrow{\mathrm{TRN}}(S)\rrbracket\simeq \mcg_2(S)=\llbracket\mod(S)\backslash\overrightarrow{\mathrm{SPN}}(S)\rrbracket,
	\end{equation}
	whose set of objects are combinatorial triangulations with a d.o.e. and combinatorial spines with a d.o.e., respectively. Our aim in this section is to introduce a new version, the (connected) fundamental modular groupoid of $S$, denoted $\pimcg(S)$, whose objects are the fundamental groupoids (in the `graph' sense) of combinatorial spines with a d.o.e.. 
	
	\subsection{Combinatorial graphs as non-Hausdorff spaces}
	Let $\G$ be a trivalent ribbon graph. We shall model its combinatorial type by a non-Hausdorff topological space. To achieve this, we first make $\G$ into a bipartite ribbon graph, by putting a degree-2 vertex (denoted \raisebox{.5mm}{$\ttimes$} in Figure~\ref{bip}) in the middle of each edge of $\G$. The resulting bipartite graph is 2-3-valent, i.e. it has vertices of two types, one type always of degree 2, other type always of degree 3, and such that $ \G $ has no edge between two vertices of the same type. 
	\begin{figure}[h!]
		\begin{center}
			\begin{tikzpicture}[line cap=round,line join=round,>=triangle 45,x=.5cm,y=.5cm]
				\draw [line width=.3pt] (-1,0)-- (2,0);
				\draw [line width=.3pt] (-2,2)-- (-1,0);
				\draw [line width=.3pt] (-1,0)-- (-2,-2);
				\draw [line width=.3pt] (2,0)-- (3,2);
				\draw [line width=.3pt] (2,0)-- (3,-2);

				\begin{scriptsize}
					\draw [fill=black] (-1,0) circle (1.5pt);
					\draw [fill=black] (2,0) circle (1.5pt);
					\draw [fill=black] (-2,2) circle (1.5pt);
					\draw [fill=black] (-2,-2) circle (1.5pt);
					\draw [fill=black] (3,2) circle (1.5pt);
					\draw [fill=black] (3,-2) circle (1.5pt);
					
				\end{scriptsize}
			\end{tikzpicture}
			\hspace{1cm}\raisebox{.9cm}{\huge $\rightsquigarrow$} \hspace{1cm}
			\begin{tikzpicture}[line cap=round,line join=round,>=triangle 45,x=.5cm,y=.5cm]
				\draw [line width=.3pt] (-1,0)-- (2,0);
				\draw [line width=.3pt] (-2,2)-- (-1,0);
				\draw [line width=.3pt] (-1,0)-- (-2,-2);
				\draw [line width=.3pt] (2,0)-- (3,2);
				\draw [line width=.3pt] (2,0)-- (3,-2);

				\begin{scriptsize}
					\draw [fill=black] (-1,0) circle (1.5pt);
					\draw [fill=white, color=white] (.5,0) circle (1.9pt);\node at (0.5,0) {$\ttimes$};
					\draw [fill=black] (2,0) circle (1.5pt);
					\draw [fill=black] (-2,2) circle (1.5pt);
					\draw [fill=white, color=white] (-1.5,1) circle (1.9pt);\node at (-1.5,1) {$\ttimes$};
					\draw [fill=black] (-2,-2) circle (1.5pt);
					\draw [fill=white, color=white] (-1.5,-1) circle (1.9pt);\node at (-1.5,-1) {$\ttimes$};
					\draw [fill=black] (3,2) circle (1.5pt);
					\draw [fill=white, color=white] (2.5,1) circle (1.9pt);\node at (2.5,1) {$\ttimes$};
					\draw [fill=black] (3,-2) circle (1.5pt);
					\draw [fill=white, color=white] (2.5,-1) circle (1.9pt);\node at (2.5,-1) {$\ttimes$};

				\end{scriptsize}
			\end{tikzpicture}
		\end{center}
		
		\caption{A trivalent graph  and the corresponding bipartite graph.}
		\label{bip}
	\end{figure}
	
	This bipartition operation is necessary to handle the loops of the trivalent graph. In fact, it is very natural to represent spines of punctured surfaces by bipartite graphs, see~\cite{panorama}. 
	For example, a choice of a d.o.e. for the trivalent graph is equivalent to a choice of just a distinguished unoriented edge (d.u.e. for short) of the corresponding bipartite graph; by taking the half-edge pointed by the d.o.e., see Figure~\ref{fig:d.o.e.d.u.e.}.
	
	From now on, we will assume that our graphs are bipartite 2-3-valent graphs. Considering 2-3-valent graphs instead of trivalent graphs changes nothing about the definition of the sets ${\mathrm{SPN}(S)}$, $\overrightarrow{\mathrm{SPN}(S)}$ and the groupoid $\mcg_1(S)$ (except the natural correspondence between d.o.e.'s and d.u.e.'s remarked above).
	
	\begin{figure}[h!]
		\begin{center}
			\begin{tikzpicture}[line cap=round,line join=round,>=triangle 45,x=.5cm,y=.5cm]
				\draw [line width=.9pt, color=blue, ->] (-1,0)-- (2,0);
				\draw [line width=.3pt] (-2,2)-- (-1,0);
				\draw [line width=.3pt] (-1,0)-- (-2,-2);
				\draw [line width=.3pt] (2,0)-- (3,2);
				\draw [line width=.3pt] (2,0)-- (3,-2);

				\begin{scriptsize}
					\draw [fill=black] (-1,0) circle (1.5pt);
					\draw [fill=black] (2,0) circle (1.5pt);
					\draw [fill=black] (-2,2) circle (1.5pt);
					\draw [fill=black] (-2,-2) circle (1.5pt);
					\draw [fill=black] (3,2) circle (1.5pt);
					\draw [fill=black] (3,-2) circle (1.5pt);
					
				\end{scriptsize}
			\end{tikzpicture}
			\hspace{1cm}\raisebox{.9cm}{\huge $\rightsquigarrow$} \hspace{1cm}
			\begin{tikzpicture}[line cap=round,line join=round,>=triangle 45,x=.5cm,y=.5cm]
				\draw [line width=.3pt] (-1,0)-- (0.5,0);
				\draw [line width=.9pt, color=blue] (0.5,0)-- (2,0);
				\draw [line width=.3pt] (-2,2)-- (-1,0);
				\draw [line width=.3pt] (-1,0)-- (-2,-2);
				\draw [line width=.3pt] (2,0)-- (3,2);
				\draw [line width=.3pt] (2,0)-- (3,-2);

				\begin{scriptsize}
					\draw [fill=black] (-1,0) circle (1.5pt);
					\draw [fill=white, color=white] (.5,0) circle (1.9pt);\node at (0.5,0) {$\ttimes$};
					\draw [fill=black] (2,0) circle (1.5pt);
					\draw [fill=black] (-2,2) circle (1.5pt);
					\draw [fill=white, color=white] (-1.5,1) circle (1.9pt);\node at (-1.5,1) {$\ttimes$};
					\draw [fill=black] (-2,-2) circle (1.5pt);
					\draw [fill=white, color=white] (-1.5,-1) circle (1.9pt);\node at (-1.5,-1) {$\ttimes$};
					\draw [fill=black] (3,2) circle (1.5pt);
					\draw [fill=white, color=white] (2.5,1) circle (1.9pt);\node at (2.5,1) {$\ttimes$};
					\draw [fill=black] (3,-2) circle (1.5pt);
					\draw [fill=white, color=white] (2.5,-1) circle (1.9pt);\node at (2.5,-1) {$\ttimes$};

				\end{scriptsize}
			\end{tikzpicture}
		\end{center}
		
		\caption{A distinguished oriented edge (d.o.e.) and the corresponding distinguished unoriented half-edge (d.u.e.)}
		\label{fig:d.o.e.d.u.e.}
	\end{figure}

	We model a graph $\G$ by the topological space $\G^*$ \label{gstar}
	whose set of points is $E(\G) \cup V(\G)$.  
	Open sets of $\G^*$ are generated by edge singletons and  stars of vertices, i.e. sets of the form $v\cup s(v)$, where 
	$v\in V(\G )$ is a vertex and $s(v)$ is the set of edges incident to $v$, see Figure~\ref{fig:modular/space}. 
	
	\label{modular}
	
	The points of the space $\G^* $ from $E(\G )$ will be called `{edges}' and those from $V(\G )$ will be called the `{vertices}' of $\G^* $. Note that a ribbon graph structure on the trivalent graph $\G$ induces a similar structure on the space $\G^*$, i.e. a cyclic ordering of $s(v)$ for every vertex $v$.  
	
	Recall that a 2-3-valent ribbon graph $\G$, together with its ribbon graph structure was called a `{modular graph}'\footnote{The naming `{modular graph}' is related to the fact that these spaces (together with the ribbon structure) are in 1-1 correspondence with the torsion-free subgroups of the modular group $\psl$, see~\cite {UZD}} in~\cite {UZD}. Accordingly,
	we will call the space $\G^*$ obtained from a such graph a `{modular graph}' as well, by abuse of language.
	
	
	\begin{figure}[h!]
		\begin{center}
			
			\begin{tikzpicture}[line cap=round,line join=round,>=triangle 45,x=.5cm,y=.5cm]
				\draw [line width=.3pt, color=red] (-1,0)-- (0.5,0);
				\draw [line width=.3pt, color=red] (0.5,0)-- (2,0);
				\draw [line width=.3pt, color=red] (-2,2)-- (-1,0);
				\draw [line width=.3pt, color=red] (-1,0)-- (-2,-2);
				\draw [line width=.3pt, color=red] (2,0)-- (3,2);
				\draw [line width=.3pt, color=red] (2,0)-- (3,-2);

				\begin{scriptsize}
					\draw [fill=black] (-1,0) circle (1.5pt);
					\draw [fill=white, color=white] (.5,0) circle (1.9pt);\node at (0.5,0) {$\ttimes$};
					\draw [fill=black] (2,0) circle (1.5pt);
					\draw [fill=black] (-2,2) circle (1.5pt);
					\draw [fill=white, color=white] (-1.5,1) circle (1.9pt);\node at (-1.5,1) {$\ttimes$};
					\draw [fill=black] (-2,-2) circle (1.5pt);
					\draw [fill=white, color=white] (-1.5,-1) circle (1.9pt);\node at (-1.5,-1) {$\ttimes$};
					\draw [fill=black] (3,2) circle (1.5pt);
					\draw [fill=white, color=white] (2.5,1) circle (1.9pt);\node at (2.5,1) {$\ttimes$};
					\draw [fill=black] (3,-2) circle (1.5pt);
					\draw [fill=white, color=white] (2.5,-1) circle (1.9pt);\node at (2.5,-1) {$\ttimes$};

				\end{scriptsize}
			\end{tikzpicture}
			\hspace{1cm}\raisebox{.9cm}{\huge $\rightsquigarrow$} \hspace{1cm}
			\raisebox{6mm}{\begin{tikzpicture}[line cap=round,line join=round,>=triangle 45,x=.5cm,y=.5cm]

					\begin{scriptsize}
						\draw [fill=black] (-0.95,0.95) circle (1.5pt);
						\draw [fill=black] (-0.95,-0.95) circle (1.5pt);
						\draw [fill=red] (-0.75,0.75) circle (2.5pt);
						\draw [fill=red] (-0.3,-0.25) circle (2.5pt);
						\draw [fill=red] (-0.3,0.25) circle (2.5pt);\node at (-0.57,0.46) {$\ttimes$};%
						\draw [fill=red] (-0.75,-0.75) circle (2.5pt);\node at (-0.57,-0.46) {$\ttimes$};%
						\draw [fill=black] (-.16,0) circle (1.5pt);%
						\draw [fill=red] (0.14,0) circle (2.5pt);%
						\draw [fill=white, color=white] (.5,0) circle (1.9pt);\node at (0.5,0) {$\ttimes$};%
						\draw [fill=red] (0.85,0) circle (2.5pt);%
						\draw [fill=black] (1.13,0) circle (1.5pt);%
						\draw [fill=red] (1.25,0.25) circle (2.5pt);\node at (1.45,.55) {$\ttimes$};%
						\draw [fill=red] (1.75,0.75) circle (2.5pt);%
						\draw [fill=red] (1.75,-0.75) circle (2.5pt);%
						\draw [fill=red] (1.25,-0.25) circle (2.5pt);\node at (1.45,-.55) {$\ttimes$};%
						\draw [fill=black] (1.9,1) circle (1.5pt);%
						\draw [fill=black] (1.9,-1) circle (1.5pt);%
						
					\end{scriptsize}
			\end{tikzpicture}}
		\end{center}
		
		\caption{A 2-3-valent graph and the corresponding modular graph - red points on the right correspond to the edges of the graph on the left.}
		\label{fig:modular/space}
	\end{figure}

	Denote the set of modular graphs obtained from the spines of $S$ (called the {\it modular spines of $S$}) by \label{mgrspine}
	\begin{eqnarray}
		\mathrm{MGR}(S):=
		\{\G^* : \G \hookrightarrow S \} /\mbox{ribbon structure preserving homeomorphisms}
	\end{eqnarray}
	and those with a d.u.e.  by \label{mgrspinedue}
	\begin{eqnarray}\label{mgrsss}
		\overrightarrow{\mathrm{MGR}}(S):=
		\{ (\G^* , {e}) : (\G,{e}) \hookrightarrow S \} 
		{\Big /}
		\begin{array}{l}
			\mbox{homeomorphisms of modular graphs }\\ \mbox{preserving  d.u.e. and ribbon structure.}
		\end{array}
	\end{eqnarray}
	Thus, elements of $\overrightarrow{\mathrm{MGR}}(S)$ can be described as the modular graphs of genus $g$ with $n$ punctures and with a distinguished edge.
	
	Elements of these sets will be denoted respectively by $|\G^*|$ and $|\G^*,e|$.
	If $S$ is of finite type, then these sets are finite.
	There are bijections between the sets of combinatorial types
	\begin{eqnarray}
		\mathrm{MGR}(S)  \longleftrightarrow |\mathrm{SPN}(S)|:=\mod(S)\backslash\mathrm{SPN}(S) \\
		\overrightarrow{\mathrm{MGR}}(S)\longleftrightarrow |\overrightarrow{\mathrm{SPN}}(S)|:=\mod(S)\backslash\overrightarrow{\mathrm{SPN}}(S).
	\end{eqnarray}

	\subsection{A variation of Penner's mapping class groupoid} 
	
	In this section, we will adapt Penner's definition of the Mapping Class Groupoid (see \cite{decorated}) to labeled modular spines.
	By a {\it labeled modular spine},  we mean a modular spine 
	$\mathcal G$ of $S$ together with a bijection (called a {\it labeling} on $\mathcal G$)
	\begin{align}\label{labellling}
		\ell: E(\mathcal G)\to \{1, 2, \dots, 6g-6+3n\}\times\{+,-\}
	\end{align}
	such that if $e$ and $e'$ are two half-edges meeting at a degree-2 vertex, then for some $k$ either 
	$\ell(e)=(k,+)$ and $\ell(e')=(k,-)$ or vice versa.

		\begin{figure}[h]
		\begin{center}
			
			\begin{tikzpicture}[scale=0.5, transform shape]
				\node [circle,draw, minimum size=4cm
				,label=above: {$(1,+)$},label=below:{$(1,-)$}
				] {};
				\filldraw[black] (2,0) circle (6pt);				
				\begin{scope}[shift={(-2,0)}]
					\vI{6pt}
				\end{scope} 
				\draw (2,0) -- (8,0) node[pos=0.75,above]{(2,-)} node[pos=0.2, above]{(2,+)};
				\begin{scope}[shift={(5,0)}]
					\vI{6pt}
				\end{scope}  
				\filldraw[black] (8,0) circle (6pt);
				\node [circle,draw, minimum size=4cm
				,label=above: {$(3,-)$},label=below:{$(3,+)$}
				] at (10,0) {};
				\begin{scope}[shift={(12,0)}]
					\vI{6pt}
				\end{scope} 
				
			\end{tikzpicture}
			
		\end{center}
		\caption{A labeled modular spine of genus 0 and 3 punctures. }
	\end{figure}
	
	$\mod(S)$ acts freely on the set $\mathrm{SPN}_\ell(S)$ of all labeled modular spines of $\mathcal G$, see \cite{decorated}. \label{lmspine}(Thanks to the $\pm$ signs we don't need to consider the case where $S$ is the punctured torus separately as in~\cite{decorated}).

	\begin{definition}\label{mcgrell}
		The {\it mapping class groupoid} (third version) $\mcg_\ell(S)$ is the category associated to the free  
		$\mod(S)$ action on the set $\mathrm{SPN}_\ell(S)$ of isotopy classes of labeled modular spines in $S$:
		\[ \mcg_\ell(S):=\llbracket\mod(S) \backslash \mathrm{SPN}_\ell(S)\rrbracket. \]
	\end{definition}
	\begin{definition}
		Let $(\G, \ell)$ be a labeled spine and $f$  a half-edge of $\G$. A {\it labeled flip} 
		of $(\G, \ell)$ on $f$ is an operation which gives us a new labeled spine $(\G', \ell')$ 
		(and a half-edge $f'$). Let $e$ be the edge which meets $f$ at a degree-2 vertex.
		We distinguish two cases:
		\begin{itemize}
			\item In case $e$ and $f$ do not meet at a second vertex, then the flip is described by 
			changing the spine locally around $f$ 
			as shown in Figure~\ref{xcxcxxx}
			and keeping the rest as it is.
			The labeling is preserved out of this local picture and for each $x\in\{a,b,c,d,e,f\}$
			the edge $x'$ in the picture on the right is given the same label as the edge $x$ on the left.
			\item 
			In case $e$ and $f$ do meet at a second vertex (as in Fig.~\ref{ozge}),
			then
			$(\G',\ell'):=(\G, \ell)$.
		\end{itemize}
\begin{figure}[h]
		\begin{center}
			\begin{tikzpicture}[line cap=round,line join=round]
				\draw (0,0) circle (0.8) ; 
\node [circle,draw, minimum size= 1.6cm, label=left: {$ e$},label=right:{$f$}
   ] {};
				\filldraw[black] (0,-0.8) circle (2pt)  ;
				\draw[hedge] (0,-0.8) -- (0,-1.8) ;
				\draw[dotted] (0,-1.8) -- (0,-2.8);
				\draw[fill=white] (0,0.8) circle (2pt);
				\begin{scope}[shift={(0,.8)}]
					\vI{3pt}
				\end{scope} 
			\end{tikzpicture}
			
			\caption{The case of two half-edges meeting at two vertices.}
			\label{ozge}
		\end{center}
\end{figure}
		$\mod(S)$-orbit of a labeled flip is called the {\it class of the flip.}
		Note that paired half-edges $e$ and $f$ define the same flip.
	\end{definition}
	
	\begin{figure}[h!] 
		\begin{center}
			\begin{tikzpicture}[line cap=round,line join=round,>=triangle 45,x=.5cm,y=.5cm]
				\draw [line width=1pt, color=blue] (-1,0)-- (.5,0) node[font=\scriptsize,midway,above]{$f$} ; 
				\draw [line width=1pt, color=blue] (.5,0)-- (2,0) node[font=\scriptsize,midway,above]{$e$} ; 
				\draw  [fill=white, color=white] (.5,0) circle (1.9pt);\node at (.5,0){$\ttimes$};
				
				\draw [line width=1pt] (-2,2)-- (-1,0) node[font=\scriptsize,midway,above]{\,\, $a$} ;  
				\draw [line width=1pt] (-1,0)-- (-2,-2) node[font=\scriptsize,midway,below]{ \,\, $b$} ; 
				\draw [line width=1pt] (2,0)-- (3,2) node[font=\scriptsize,midway,above]{$c$ \,\,} ;
				\draw [line width=1pt] (2,0)-- (3,-2) node[font=\scriptsize,midway,below]{$d$ \,\,} ; 
				\begin{scriptsize}
					\draw [fill=black] (-1,0) circle (3pt);
					\draw [fill=black] (2,0) circle (3pt);
					\draw  [fill=white, color=white] (-2,2) circle (1.9pt);\node at (-2,2){$\ttimes$};
					\draw  [fill=white, color=white]  (-2,-2) circle (1.9pt);\node at (-2,-2) {$\ttimes$};
					\draw  [fill=white, color=white]  (3,2) circle (1.9pt);\node at  (3,2){$\ttimes$};
					\draw  [fill=white, color=white]  (3,-2) circle (1.9pt);\node at (3,-2){$\ttimes$};
				\end{scriptsize}
			\end{tikzpicture}
			\hspace{1cm}\raisebox{.9cm}{\huge ${\rightsquigarrow}$ } \hspace{1cm}
			\raisebox{-.1cm}{\begin{tikzpicture}[line cap=round,line join=round,>=triangle 45,x=.4cm,y=.4cm]
					\draw [line width=1pt, color=blue] (0,-1)-- (0,.5) node[font=\scriptsize,midway,left]{$f'$} ;
					\draw [line width=1pt, color=blue] (0,.5)-- (0,2) node[font=\scriptsize,midway,left]{$e'$} ;
					\draw  [fill=white, color=white] (0,.5) circle (1.9pt);\node at (0,.5){$\ttimes$};
					
					\draw [line width=1pt] (-2,3)-- (0,2) node[font=\scriptsize,midway,above]{$a'$} ;
					\draw [line width=1pt] (0,2)-- (2,3) node[font=\scriptsize,midway,above]{$c'$} ;
					\draw [line width=1pt] (0,-1)-- (-2,-2) node[font=\scriptsize,midway,below]{$b'$} ;
					\draw [line width=1pt] (0,-1)-- (2,-2) node[font=\scriptsize,midway,below]{$d'$} ;
					\begin{scriptsize}
						\draw [fill=black] (0,-1) circle (3pt);
						\draw [fill=black] (0,2) circle (3pt);
						\draw  [fill=white, color=white]  (-2,3) circle (1.9pt);\node at (-2,3){$\ttimes$};
						\draw  [fill=white, color=white]  (2,3) circle (1.9pt);\node at (2,3){$\ttimes$};
						\draw  [fill=white, color=white]  (-2,-2) circle (1.9pt);\node at (-2,-2){$\ttimes$};
						\draw  [fill=white, color=white]  (2,-2) circle (1.9pt);\node at(2,-2) {$\ttimes$};
					\end{scriptsize}
			\end{tikzpicture}}
		\end{center}
		\caption{A flip of a modular spine.}
		\label{xcxcxxx}
	\end{figure}

	This defines an order-4 map 
	\begin{equation}
		\Psi: |{\mathrm{SPN_\ell^*}}(S)| \to |{\mathrm{SPN_\ell^*}}(S)|,
	\end{equation}
	where $*$ indicates the choice of an edge, which the labeled flip will be applied to.
	Each such labeled flip, where $f$ is the flipped edge,
	induces a pairing
	\begin{equation}\label{psisapka}
		\hat\psi_f: E(\G) \to E(\G')
	\end{equation}
	sending each edge to the edge with the same label (i.e. $\psi(f)=f'$, $\psi(e)=e'$; $\psi(a)=a'$, $\psi(b)=b'$, $\psi(c)=c'$, $\psi(d)=d'$  and similarly for the remaining edges of $\G$ and $\G'$).
	\begin{theorem}
		$\mcg_\ell(S)$ is generated by classes of flips and $\mod(S)$-orbits of permutations of spine labels,
		sending the paired edges $(m, \pm)$ to the paired edges $(n,\pm)$ possibly changing the signs. 
	\end{theorem}
	\begin{proof}
		By the Whitehead lemma, flips are transitive on the set of spines. Since we've also added all possible permutations of labels, these elements generate the groupoid.
	\end{proof}

	\subsection{Fundamental groupoids of modular graphs}
	In this section and in the next one, we adopt Brown's terminology~\cite{brown}
	about fundamental groupoids.
	
	Let $\G$ be a modular graph (i.e. a 2-3-valent ribbon graph). Let $\G^*$ be the space defined in page~\pageref{modular}.
	There is the canonical ``{retraction}'' map
	\begin{equation}
		f_{\G }: \G  \to \G^*,
	\end{equation}
	which sends vertices to vertices and each point $x$ on an edge $e\in E(\G )$ to the point $e\in \G^*$.
	The fundamental groupoid $\Pi_1^{\G^*}$ is defined in the usual way, that is, its objects are the points of $\G^*$, and morphisms are homotopy classes of paths. So $f_{\G }$ induces a morphism of groupoids 
	\begin{equation}
		f_{\G }^\flat:\Pi_1^{\G } \to \Pi_1^{\G^*}
	\end{equation}
	with the left-inverse
	\begin{equation}
		f_{\G }^\sharp: \Pi_1^{\G^*}{ \to \Pi_1^{\G }},
	\end{equation}
	defined as follows: If a path starts and ends at a vertex-point of ${\G^* }$, then one may lift it to ${\G }$ in a canonical way. 
	If one or both endpoints is an edge-point 
	$e$ of ${\G^* }$, then lift it to the path on ${\G }$, whose corresponding endpoint(s) is(are) the midpoint(s) of the edge $e$. Note that  $f_{\G }^\sharp$ is not induced by any continuous map 
	$\G^* \to \G $.
	
	Observe that $f_{\G }^\flat$ is an equivalence of groupoids with the inverse $f_{\G }^\sharp$.
	Indeed, $f_{\G }^\flat\circ f_{\G }^\sharp:\Pi_1^{\G^*}{ \to \Pi_1^{\G^* }}$ is the identity morphism. On the other hand, 
	the morphism $f_{\G }^\sharp \circ f_{\G }^\flat:\Pi_1^{\G}{ \to \Pi_1^{\G }}$ is homotopic (i.e. naturally transforms) to the identity, via
	\begin{equation}
		\gamma\in \Pi_1^{\G }(x,y) \to \alpha_t \cdot\gamma \cdot \beta_t \in  \Pi_1^{\G }(x_t,y_t)
	\end{equation}
	where $x_t=x-(x-1/2)t$ moves towards the midpoint of the edge containing $x$ if $x$ is not a vertex and $x_t\equiv x$ if $x$ is a vertex; similarly
	$y_t=y-(y-1/2)t$ moves towards the midpoint of the edge containing $y$  if $y$ is not a vertex and $y_t\equiv y$ if $y$ is a vertex; 
	$\alpha_t(s) =x_t(1-s)+sx$ is a path from $x_t$ to $x$ inside the edge containing $x$,
	and  $\beta_t(s)=y(1-s)+sy_t$ is a path from $y$ to $y_t$ inside the edge containing $y$.
	
	Note that this homotopy restricts to identity on the subgroupoid of $\Pi_1^{\G }$ which consists of paths whose endpoints lies on vertices or midpoints of ${\G }$, i.e. it may be considered as a deformation retract at the fundamental groupoid level. 
	
	Hence, the groupoids $\Pi_1^{\G }$ and  $\Pi_1^{\G^*}$ are equivalent 
	under an equivalence  determined uniquely  by $f_\G$.
	
	Denote by \label{tildapi1}
	\begin{equation}
		\widetilde{\Pi}_1^{\G^*} < \Pi_1^{\G^*}
	\end{equation}
	the subgroupoid of homotopy classes of paths starting and terminating at an edge-point of ${\G^*}$. It is clear that $\widetilde{\Pi}_1^{\G^*}$ is a sub-category of $\Pi_1^{\G^*}$. To see that two mentioned categories are equivalent, we'll show that their skeletons, see \cite[\S~4]{abstract/concrete/categories}, are isomorphic. As both are connected groupoids, their skeletons are merely their fundamental groups. Elements of $\widetilde{\Pi}_1^{\G^*}$ can be represented as sequences of edges of ${\G^*}$,
	$$\gamma=(e_1, e_2, \dots, e_k)$$
	where $e_i, e_{i+1}$ are incident at a single vertex. If $e_i, e_{i+1}$ are incident at two vertices (so that they constitute a loop) then we replace  the comma between $e_i$ and $e_{i+1}$ by a semicolon as in 
	$(\dots, e_i; e_{i+1}, \dots)$ if the path moves from  $e_i$ to $e_{i+1}$ through a 2-valent vertex.
	
	
	%
	%
	%
	%
	%
	\begin{figure}[h!]
		\begin{center}
			\begin{tikzpicture}[line cap=round,line join=round,>=triangle 45,x=.5cm,y=.5cm]
				\draw [line width=1pt] (-0.3,0)-- (2,0);
				\draw [line width=1pt] (2,0)-- (3,2);
				\draw [line width=1pt] (2,0)-- (3,-2);
				\draw [fill=white, color=white](-0.3,0) circle (1.9pt);
				\node at (-0.3,0) {$\ttimes$};
				\node at (-1.05,0) {$c$};
				\draw [fill=black] (2,0) circle (3pt);
				\draw [fill=white, color=white](3,2) circle (1.9pt);
				\node at (3,2) {$\ttimes$};
				\node at (3.5,2.5) {$a$};
				\draw [fill=white, color=white] (3,-2) circle (1.9pt);
				\node at (3,-2) {$\ttimes$};
				\node at (3.5,-2.5) {$b$};
			\end{tikzpicture}
		\end{center}
		
		\caption{Edges meeting at a vertex of degree 3}\label{nonpath}
	\end{figure}
	
	For an element $g = (e_0,e_1,\ldots,e_{k-1},e_0)$ in the fundamental group of $\widetilde{\Pi}_1^{\G^*}$ based at an arbitrary edge $e_0$ the associated path :
	\begin{align*}
		\gamma_g \colon [0,1] & \to \mathcal{G}^*\\
		[0,1/(k+1)) &\to e_0 \\
		1/(k+1) &\to \mbox{ the vertex joining }e_0 \mbox{ to } e_1 \\
		(1/(k+1),2/(k+1)) &\to e_1 \\
		\vdots & \\
		((k-1)/(k+1),k/(k+1)) &\to e_{k-1} \\
		k/(k+1) &\to \mbox{ the vertex joining }e_{k-1} \mbox{ to } e_0 \\
		(k/(k+1),1] &\to e_0 
	\end{align*}
	induces an isomorphism from the isotropy group of $\widetilde{\Pi}_1^{\G^*}$  at $e_0$ to the isotropy group of $\Pi_1^{\G^*}$ at $e_0$. Let us also note that $\widetilde{\Pi}_1^{\G^*}$ is not the fundamental groupoid of some topological space. 
	
	The groupoid $\widetilde{\Pi}_1^{\G^*} $ is generated by paths of length 2 which are subject to the following relations :
	
	\begin{eqnarray*}\label{relations}
		(a,a)=(a), \quad (a,b)\cdot (b,a)=(a), \quad (a,b)\cdot (b,c) \cdot (c,a)=(a)
	\end{eqnarray*}
	and
	\begin{eqnarray*}
		(a;a)=(a), \quad (a;b)\cdot (b;a)=(a).  \end{eqnarray*}

	Note that $(a)$ is the identity of the isotropy group of $\widetilde{\Pi}_1^{\G^*}$ at $a$. The path $(b,a)$ is the inverse of $(a,b)$ and that the third relator can be read as (see Figure~\ref{nonpath})
	\begin{equation*}
		(a,b)\cdot (b,c) =(a,c).
	\end{equation*}
	We call $\widetilde{\Pi}_1^{\G^*}$ the {\it graph fundamental groupoid} of $\G^*$ as opposed to  $\Pi_1^{\G^*}$  which is the usual   fundamental groupoid of $\G^*$ from topology. \label{graphpi1} As they are equivalent, they have the same isotropy, which is $\pi_1(\G)$. 
	
	\subsection{Introduction to the fundamental modular groupoid }
	Now let $\eta_1:\G_1  \hookrightarrow S$ be a spine of $S$. Then since $\G_1 $ is a  deformation retract of $S$, there is an equivalence (i.e. homotopy) $\pi\eta_1$ of groupoids 
	between $\Pi_1^{\G_1 }$ and $\Pi_1^{\mathcal S}$  by the following result:
	\begin{lemma} (Brown~\cite{brown}, 6.5.10)
		If $f, g : X \to Y$ are homotopic maps of spaces, then the induced 
		morphisms $\pi f, \pi g : \Pi_1^X \to \Pi_1^Y$  are homotopic.
	\end{lemma}
	Hence one has equivalences
	\begin{equation}\label{seq}
		\widetilde{\Pi}_1^{\G_1^* } \equiv \Pi_1^{\G_1^* } \equiv \Pi_1^{\G_1} \equiv \Pi_1^{\mathcal S}
	\end{equation}

	If $\G_2\hookrightarrow S$ is another spine, then we have another equivalence $\Pi_1^{\G_2}\equiv\Pi_1^{\mathcal S}$, and so 
	there is an equivalence
	\begin{equation}
		\Pi_1^{\G_2}\equiv \Pi_1^{\G_1} \iff \Pi_1^{\G_2^*}\equiv \Pi_1^{\G_1^*}
		\iff \widetilde{\Pi}_1^{\G_2^*}\equiv \widetilde{\Pi}_1^{\G_1^*},
	\end{equation}
	and the latter equivalence is in fact an isomorphism as we describe below.

	We let 
	\begin{equation}
		\psi_f=\mod(S)\bigl ([(\G_1 , {\ell_1}),\hookrightarrow S], [(\G_2, \ell_2) \hookrightarrow S]\bigr)
	\end{equation}
	where $f$ is a half-edge of $\G_1$ denote the class of a flip inside the groupoid 
	$\mcg_\ell(S)=\llbracket \mod(S)\backslash \mathrm{SPN}_\ell(S)\rrbracket$.	This means that $(\G_2, \ell_2)$ is obtained from $(\G_1, \ell_1)$ via a labeled flip at $f$.
	To define the isomorphism $\widetilde{\Pi}_1^{\G_2^*}\simeq \widetilde{\Pi}_1^{\G_1^*}$, we make the following definition: 
	\begin{definition} \label{bridge}
		We define a map $\psi_f^*:\widetilde{\Pi}_1^{\G_1^*}\to \widetilde{\Pi}_1^{\G_2^*}$ by describing its effect on the generators, as follows:
		\begin{itemize}
			\item If the flipped edge $f$ is on a generator of the form $(f;e)$ then recalling that 
			$\G_1:=\G_2$ in this case,  $\psi_f^*:\widetilde{\Pi}_1^{\G_1^*}\to \widetilde{\Pi}_1^{\G_1^*}$ is defined to be the identity.
			\item Otherwise, recall that there is an induced pairing 
			$\hat\psi_f:E(\G_1^*)\to E(\G_2^*)$, where $f$ is the flipped edge (cf.~(\ref{psisapka}) and Fig.~\ref{xcxcxxx}).
			Let $(x,y)$ (resp. $(x;y)$) be a generator of 
			$\widetilde{\Pi}_1^{\G_1^*}$. 
			If neither $x$ nor $y$ is neighbor of the flipped edge $f$, then  $\psi_f^*$ sends the 
			generator $(x,y)$ (resp. $(x;y)$) of $\widetilde{\Pi}_1^{\G_1^*}$ to the generator 
			$(\hat \psi_f(x), \hat \psi_f(y))$ (resp. $(\hat \psi_f(x); \hat \psi_f(y))$) of $\widetilde{\Pi}_1^{\G_2^*}$. Otherwise, $\psi_f^*$ is defined on the generators seen in Fig.~\ref{xcxcxxx} as follows:
		\end{itemize}
		\begin{equation*}
			\begin{array}{|lll||lll|}
				\hline
				(a,f)&\to&(a',e',f')=(a',e')\cdot (e',f')	&(f,a)&\to&(f',e',a')=(f',e')\cdot(e',a') \\\hline
				(b,f)&\to&(b',f')	&(f,b)&\to&(f',b')\\\hline
				(c,e)&\to&(c',e')	& (e,c)&\to&(e',c')\\\hline
				(d,e)&\to&(d',f',e')=(d',f')\cdot(f',e')	&(e,d)&\to&(e',f',d')=(e',f')\cdot(f',d')\\\hline
				(f,e)&\to&(f',e')&(e,f)&\to&(e',f')\\
				\hline
			\end{array}
		\end{equation*}
	\end{definition}
	\begin{lemma}
		The map $\psi_f^*$ is an isomorphism.
	\end{lemma}
	\begin{proof}
		We investigate the effect of 
		$\psi_f^*\circ \psi_{f'}^*$ on the generator 
		$(a',c')$ of $\widetilde{\Pi}_1^{\G_2^*}$:
		\begin{align*}
			\psi_f^*(\psi_{f'}^*(a',c'))&=\psi_f^*\bigl((a,f,e,c)\bigr)\\
			&=\psi_f^*\bigl((a,f)\cdot (f,e)\cdot(e,c) \bigr)\\
			&=\psi_f^*\bigl((a,f)\bigr)\cdot \psi_f^*\bigl((f,e)\bigr)\cdot \psi_f^*\bigl((e,c)\bigr)\\
			&=(a',e')\cdot (e',f') \cdot (f',e')\cdot (e',c')\\
			&=(a',c').
		\end{align*}
		The last equality follows from the presentation of $\widetilde{\Pi}_1^{\G_2^*}$.
		For  other generators of $\widetilde{\Pi}_1^{\G_2^*}$, this is proved similarly. By symmetry, this works for the generators of $\widetilde{\Pi}_1^{\G_2^*}$; showing that $\psi_f^*$ and $\psi_{f'}^*$ are inverses of each other. Therefore, $\psi_f^*$ is an isomorphism.
	\end{proof}
	We adapt the definition of the set $\overrightarrow{\mathrm{MGR}}(S)$ given in (\ref{mgrsss}) to the labeled case as follows:
	\begin{eqnarray}
		{\mathrm{MGR}_\ell}(S):=
		\{ (\G^* , {\ell}) : (\G,{\ell}) \hookrightarrow S \} 
		{\Big /}
		\begin{array}{l}
			\mbox{homeomorphisms of modular graphs  }\\ \mbox{preserving the labeling  and ribbon structure.}
		\end{array}
	\end{eqnarray}
	
	Thus, elements of ${\mathrm{MGR}_\ell}(S)$
	can be described as the modular graphs of genus $g$ and with $n$ punctures, whose half-edges (red-points in Fig.~\ref{fig:modular/space}) are labeled as in (\ref{labellling}). These elements will be denoted as $|\G^*,\ell|$.

	Flip-induced isomorphisms give rise to the following groupoid construction:
	\label{fmgrs}
	\begin{definition}\label{fmg}
		{\rm The } {\it fundamental modular groupoid of }$S$ {\rm is the groupoid $\pimcg(S)$ whose object set is 
			${\mathrm{MGR}}_\ell(S)$. Morphisms of $\pimcg(S)$ are defined as}
			\begin{equation}\label{grrou}
				\mor_{\pimcg}(|\G^*_1 ,{\ell_1}|,|\G^*_2,{\ell_2}|):=
				\left\{ 
				\begin{array}{l} 
					\mbox{isomorphisms } \widetilde{\Pi}_1^{\mathcal \G^*_1}\to \widetilde{\Pi}_1^{\mathcal \G^*_2} \mbox{ induced}\\
					\mbox{by finite sequences of flips \& }\\
					\mbox{relabelings}
				\end{array}
				\right\},
		\end{equation} 
where an isomorphism $\widetilde{\Pi}_1^{\mathcal \G^*_1}\to \widetilde{\Pi}_1^{\mathcal \G^*_2} $ induced by a relabeling is the ``{identity}'' isomorphism sending each path to itself.
	\end{definition}
	Note that the set ${\mathrm{MGR}_\ell}(S)$ is in canonical bijection with the set $\mod(S) \backslash \mathrm{SPN}_\ell(S)$. We identify the object set of  $\pimcg(S)$ and that of $\mcg_\ell(S)$ via this bijection, and denote them as $|\G^*, \ell|$.
	
	A finite sequence of flips and relabelings on a labeled modular graph of a surface determines a morphism in both categories $\pimcg(S)$ and $\mcg_\ell(S)$. The following theorem describes the functor between these categories determined by this correspondence.
	
\begin{theorem}
	There is a functor $\pimcg(S) \to \mcg_\ell(S)$ which is identity on the object set and which associates graph fundamental groupoid isomorphisms in $\pimcg(S)$ to the morphism in $\pimcg(S)$ that are determined by the same sequence of flips and relabelings. On the isotropy groups, this yields a surjection
	\begin{equation}\label{mctoaut}
		\mathrm{Isot}({\pimcg}, |\G^*, \ell|) \to
		\mathrm{Isot}({\mcg_\ell},|\G^* ,\ell|) \simeq \mod(S).
	\end{equation}
	\label{thm:surjection}
\end{theorem}

\begin{proof}
	Each spine $G\hookrightarrow S$ is a deformation retract of $S$, hence determines an injection $\Pi_1(\G) \hookrightarrow \Pi_1(S)$ of fundamental groupoids	. Taking the composition of this injection with the injection $\widetilde{\Pi}_1^{\G^*} \hookrightarrow \Pi_1(\G)$, we can view $\widetilde{\Pi}_1^{\G^*}$ as a subgroupoid of $\Pi_1(S)$ via an embedding $\Phi_\G$, where we can identify each element of $\G^*$ with the embedding of midpoint $p_e$ of the corresponding edge $e$ of $\G$.
	
	Note that any morphism $|\G,\ell| \to |\G',\ell|$ in $\pimcg(S)$ is represented by a pair of spines $\G \hookrightarrow S$ and $\G' \hookrightarrow S$, related by a finite sequence of flips. Such a sequence also determines a correspondence $E(\G) \to E(\G')$ between the edge sets of the spines. Keeping track of the path taken by the midpoints $p_e$ through flips and isotopies, we obtain paths $\alpha_e$ from $p_e$ to $p'_{e'}$; where our morphism pairs the edge $e$ of $\G$ with the edge $e'$ of $\G'$. Now, given a path $\gamma$ from $e$ to $f$ in $\G^*$, we construct the path $(\Phi'_{\G'})^{-1}\left({\alpha_e}^{-1}\cdot\Phi_{\G}(\gamma)\cdot\alpha_f\right)$ 
	from $e'$ to $f'$ in $\widetilde{\Pi}_1^{\G'}$. This procedure describes a morphism $|\G,\ell| \to |\G',\ell|$ in $\mcg_\ell(S)$, which is the image of the original morphism in $\pimcg(S)$ under our functor.
	
	Functoriality of the procedure is clear since the morphisms in both categories are determined by a sequence of flips. To prove that the functor is indeed well-defined on the set of morphisms of the corresponding categories, we will first establish this fact on the isotropy groups. 
	
	We consider a sequence of flips transforming an embedding 
	$\phi:\G \hookrightarrow S$ to an  embedding 
	$\phi':\G \hookrightarrow S$, inducing $\Phi_\G:\widetilde{\Pi}_1^{\G^*}\to{\Pi}_1(S)$ and $\Phi'_\G:\widetilde{\Pi}_1^{\G^*}\to{\Pi}_1(S)$, respectively. Let us assume that the sequence of flips give identity morphism in $\mathrm{Isot}({\pimcg},|\G^* ,\ell|)$. Below, we will prove that this sequence also describes the identity morphism in $\mathrm{Isot}({\mcg_\ell},|\G^* ,\ell|)$. By assumption, we see that the sequence of flips results in identity map $\widetilde{\Pi}_1^{\G^*} \to \widetilde{\Pi}_1^{\G^*}$. Thus, in particular, each edge is mapped to itself. 
	
	Let $e,f \in E(\G)$ be two adjacent edges of $\G$, and let $\gamma$ represent the path $(e,f)$ in $\Pi_1^{\G^*}$. Then
	$$\gamma=(\Phi'_{\G})^{-1}\left(\alpha_e^{-1}\cdot\Phi_{\G}(\gamma)\cdot\alpha_f\right),
	$$
	and hence
	\begin{equation}\label{eq:deformation}
		\Phi'_{\G}(\gamma)=\alpha_e^{-1}\cdot\Phi_{\G}(\gamma)\cdot\alpha_f .
	\end{equation}
	Thus, $\Phi_\G(\gamma)$ can be deformed to $\Phi'_\G(\gamma)$ while moving the endpoints $p_e, p_f$ to $p_{e'}, p_{f'}$ via $\alpha_e,\alpha_f$; respectively. 
	
	\begin{figure}[!h]
		\centering
		\begin{tikzpicture}[scale=.8]
			\definecolor{isocol}{rgb}{0 , .5 , 0}
			\def\circsize{2}
			\tikzset{
				sl/.style={
					black
				},
				g1/.style={
					red,thick
				},
				g2/.style={
					blue,thick
				},
				snake it/.style={
					decorate, decoration=snake},
				arrowmark/.style 2 args={
					decoration={markings,mark=at position #1 with \arrow{#2}}},
				iso/.style={
					isocol,
					postaction={decorate},
					arrowmark={.5}{>},
					arrowmark={.9}{>},
					arrowmark={.15}{>}
				},
				isonode/.style={
					inner sep=.8*\circsize,
					circle, 
					fill,
					isocol
				},
				gnode1/.style={
					inner sep=.6*\circsize,
					circle, 
					fill,
					g1
				},
				gnode2/.style={
					inner sep=.6*\circsize,
					circle, 
					fill,
					g2
				}
			}
			\draw [sl] plot [smooth,tension=.7] coordinates {(12,2) (10,2.6) (6,4) (0,2) (0,-2) (6,-4) (12,-2) };
			\draw [sl] plot [smooth,tension=.6] coordinates {(3,2) (4, 1.5) (6, 1) (8, 1.5) (9, 2) };
			\draw [sl] plot [smooth,tension=.6] coordinates {(4, 1.5) (6, 2) (8, 1.5)};
			\foreach \i in {1,2}
			{\node[gnode\i] 
				(A\i) at   ( $(1,-1)   +(4*\i-4,-1.5*\i+ 1.5 ) $){};
				\node[gnode\i] 
				(B\i) at    ( $(3,0)    +(4*\i-4,-1.5*\i+ 1.5 ) $){};
				\node[gnode\i] 
				(C\i) at    ( $(4,1)    +(4*\i-4,-1.5*\i+ 1.5 ) $){};
				\node[gnode\i] 
				(D\i) at    ( $(4.5,-.8) +(4*\i-4,-1.5*\i+ 1.5 ) $){};
				\foreach \x/\y in {A\i/B\i,B\i/C\i,B\i/D\i}
				\draw[g\i] (\x) to
				[out=10, in=170, relative] coordinate[pos=.5](\x\y) (\y);
			}
			\draw[iso] (A1B1) 
			to [out=10, in=170, relative] 
			node[anchor=north]{$\alpha_e$} (A2B2);
			\draw[iso] (B1C1) to [out=10, in=170, relative] 
			node[anchor=south]{$\alpha_f$} (B2C2);
			\node[isonode, label={[isocol, anchor=90]below: $p_e$ }] at (A1B1){};
			\node[isonode, label={[isocol, anchor=130] $p'_{e}$ }] at (A2B2){};
			\node[isonode, label={[isocol, anchor=300] $p_f$ }] at (B1C1){};
			\node[isonode, label={[isocol, anchor=150]below: $p'_{f}$ }] at (B2C2){};
			
			\draw[iso] ( $(A1B1) + (-.05,.05)$ ) 
			to [out=10, in=170, relative] ( $(B1) + (-.05,.05)$ );
			\draw[iso] ( $(B1) + (-.05,.05)$ )
			to [out=10, in=170, relative]( $(B1C1) + (-.05,.05)$ );
			\node[anchor=330,isocol] at (B1) {$\Phi_\G(\gamma)$};
			\draw[iso] ( $(A2B2) + (-.05,.05)$ ) 
			to [out=10, in=170, relative] ( $(B2) + (-.05,.05)$ );
			\draw[iso] ( $(B2) + (-.05,.05)$ )
			to [out=10, in=170, relative]( $(B2C2) + (-.05,.05)$ );
			\node[anchor=330,isocol] at (B2) {$\Phi_\G'(\gamma)$};
		\end{tikzpicture}
		\caption{}
	\end{figure}
	
	At the common vertex of $e$ and $f$, they meet with a third edge $g$. The paths joining $p_e$ and $p_f$ to $p_g$ also deform to paths joining $p'_{e}$ a $p'_{f}$ to $p'_g$. For both deformations, $p_g$ travels via $\alpha_g$. Note that the 1-dimensional cycle $p_e \to p_f \to p_g \to p_e$ of shortest paths on the image of the graph is homologically trivial; so is the cycle $p'_e \to p'_f \to p'_g \to p'_e$. Hence, these two `triangles' bound disks. Also, they are isotopic via the image of a triangular prism, whose edges are mapped via $\alpha_e, \alpha_f$ and $\alpha_g$. If we close the top and bottom of this prism with the aforementioned triangles, we obtain the boundary of a solid triangular prism, which is topologically a sphere. Since the second homotopy group of the surface $S$ is trivial, the image of this boundary is contractible, thus the map can be extended to that of a solid triangular prism. If we restrict this map to the subset which is comprised of the unions of line segments joining the barycenter of each sectional triangle to its corners, we obtain the deformation of the piece of $\Phi_\G$ bounded by $p_e, p_f$ and $p_g$ to the relevant piece of $\Phi'_\G$.
	
	\begin{figure}
		\centering
		\begin{tikzpicture}[scale=.6]
			\tikzset{
				nx/.style={
					inner sep=2,
					circle, 
					fill,green!50!red
				},
				nxl/.style={
					thick,
					green!50!red
				}
			}
			\node[nx, label={[anchor=north east] $\alpha_e(t)$ }] 
			(A) at (0,0){};
			\node[nx, label={[anchor=north west] $\alpha_f(t)$ }] 
			(B) at (6,0){};
			\node[nx, label={[anchor=south] $\alpha_g(t)$ }] 
			(C) at (3,5){};
			\draw (A) 
			-- (B) 
			-- (C) 
			-- (A);
			\node[nx, label={[anchor=north]below: $v(t)$ }] 
			(center) at (\bctr{A}{B}{C}){};
			\draw[nxl] (A) -- (center);
			\draw[nxl] (B) -- (center);
			\draw[nxl] (C) -- (center);
			;		\end{tikzpicture}
		\caption{A section of the isotopy prism, at moment $t$. The vertices are mapped to $\alpha_e(t),\alpha_f(t),\alpha_g(t)$ and the edges are mapped to the deformation given by Equation \ref{eq:deformation}. The isotopy on around the vertex $v$ is given by the line segments joining the center $v(t)$ to the vertices.}
	\end{figure}
	
	We can repeat this construction around each (3-valent) vertex of the graph. Since the deformations obtained agree on the intersections (i.e. $p_e$ for each vertex $e$), they can be combined into a single deformation of $\G$. Hence, these two embeddings are isotopic and thus they represent the same element in $\mcg_\ell$. 
	
	Now, consider two sequences of flips describing the same morphism $|\G,\ell| \to |\G',\ell|$ in $\pimcg(S)$ and describing morphisms $\mu_1$ and $\mu_2$ in $\mcg_\ell(S)$, respectively. Then, $\mu_1\mu_2^{-1} \in \mathrm{Isot}({\mcg_\ell},|\G^* ,\ell|)$ is the identity, hence $\mu_1=\mu_2$. We conclude that the procedure is  well-defined. Since any morphism in $\pimcg(S)$ is described by a sequence of flips and such a sequence also gives a morphism in $\mcg_\ell(S)$, the functor is surjective between sets of morphisms.
\end{proof}

In order to investigate the kernel of this functor, we first describe a certain kind of automorphism on an arbitrary groupoid. Let $\mathbf X$ be a groupoid, let $a,b \in \obj(\mathbf X)$ and let $\gamma \in \mor_\mathbf{X}(a,b)$. We define the functor $\Xi_{\gamma}\in\aut(\mathbf X)$ as follows: On $\obj(\mathbf{X})$ it only exchanges $a$ and $b$, and leaves other objects as is. For a morphism $\alpha \in \mor_\mathbf{X}(x,y)$, the functor operates as follows:
$$
\Xi_\gamma(\alpha)= 
	\gamma^{-\epsilon(a)}\cdot  \alpha \cdot \gamma^{\epsilon(b)} 
$$
where $\epsilon$ is the function (depending on $\gamma$)
$$
\epsilon(x)=
\begin{cases}
+1 & x=a\\
-1 & x=b \\
0 & \mbox{otherwise}.
\end{cases}
$$
Concisely, this automorphism exchanges the objects $a$ and $b$, and extends any morphism which begins or ends at either of them using $\gamma$ to their new position. In case $a=b$, this yields the functor  
$$
\Xi_\gamma(\alpha)= 
\begin{cases}
	\gamma\cdot \alpha & \alpha\in \mathrm{Mor}(a,x) \quad (x\neq a)\\ 
	\alpha \cdot \gamma^{-1} & \alpha\in \mathrm{Mor}(x,a) \quad (x\neq a)\\ 
	\gamma\cdot  \alpha\cdot \gamma^{-1}& \alpha\in \mathrm{Mor}(a,a)\\
	\alpha & \mathrm{otherwise.}
\end{cases}
$$
In this case, the automorphism $\Xi_\gamma$ is ``{inner}'' in the sense that 
the induced outer automorphism on isotropy groups of $\mathbf X$ is trivial.
(If $\mathbf X$ is the fundamental groupoid of some topological space, then $\Xi_\gamma$ wildly violates the topology since it is not induced by a homeomorphism of the space in question, in general.)

In this paragraph, for simplicity $e$ and $f$ denote full-edges and not half edges. Consider a labeled graph embedding $\G \hookrightarrow S$ and two adjacent edges $e$ and $f$ of $\G$. By the pentagon relation described in Figure~\ref{fig12:pentagram}, we see that we can transpose these two  edges by five consecutive flips. We can then apply relabeling switching the labels of $e$ and $f$, keeping the other labels as is. This operation gives back the same labeled graph. Since the operation is confined to a simply-connected region, it is topologically trivial and hence gives the identity morphism at $|\G^*,\ell|$ in $\mcg_\ell(S)$, i.e. the identity element of the group $\mathrm{Isot}({\mcg_\ell},|\G^* ,\ell|)$. Nevertheless, we see that the same sequence of flips induce the nontrivial automorphism $\Xi_{\gamma}\in \aut(\widetilde{\Pi}_1^{\G^*})$ where $\gamma=(e,f)$ denotes the short path  from $e$ to $f$.  Consequently, the kernel of $
\mathrm{Isot}({\pimcg}, |\G^*, \ell|) \to
		\mod(S)$ is non-trivial in general.

By Theorem~\ref{thm:surjection}, we have a surjection 
$\mathrm{Isot}({\pimcg}, |\G^*,\ell|)\to\mod(S)$. 
It is known that $\mod(S)<\out(\pi_1(S))$ is the group of outer automorphisms which preserves the loops around the punctures. On the other hand, for any groupoid $\mathbf X$, there is a surjection
	$
	\aut(\mathbf{X})\to \out(\mathrm{Isot}(\mathbf{X},x))
	$
	for any object $x$ of $\mathbf X$. Whence
	\begin{equation}\label{miracle}
		\mathrm{Isot}({\pimcg}, |\G^*,\ell|)\to
		{\textcolor{green}
			{\aut(\widetilde{\Pi}_1^{\G^*})}}\to
		\out(\mathrm{Isot}({\widetilde{\Pi}_1^{\G^*}}, x))
		=
		\out(\pi_1(\G^* ,x))
	\end{equation}

\begin{remark}
The arrow on the right is obtained as follows: Let $\Xi: \mathbf X \to \mathbf X$ be a grupoid automorphism.
Let $\gamma\in \mathrm{Isot}(\mathbf X,x)$ for some object $x$.
Then $\Xi(\gamma) \in \mathrm{Isot}(\mathbf X,\Xi(x))$. Given a morphism $\beta\in \mathrm{Mor}(x,\Xi(x))$, we obtain $\beta\cdot \Xi(\gamma)\cdot \beta^{-1}\in \mathrm{Isot}(\mathbf X,x)
$. Hence we obtain an element of $\aut(\mathrm{Isot}(\mathbf X,x))$
\begin{align*}
	\Xi_\beta:  \mathrm{Isot}(\mathbf X,x) &\to \mathrm{Isot}(\mathbf X,x),\\
	\gamma &\mapsto \beta\cdot \Xi(\gamma)\cdot \beta^{-1} 
\end{align*}
If $\beta'$ is any other morphism in $\mathrm{Mor}(x,\Xi(x))$, then $\Xi_\beta$ and $\Xi_{\beta'}$ are conjugates. Hence we have a well-defined map $\aut( \mathbf X) \to \out(\mathrm{Isot}(\mathbf X,x))$.
\end{remark}

	\subsection{Punctures}\label{pnctrs}
	Recall that spines $\G \hookrightarrow S$ are endowed with a ribbon graph structure induced from the orientation of $S$.
	This induces a similar structure on $\G$, $\G^*$ and on $\widetilde{\Pi}_1^{\mathcal \G^*}$. Indeed, we define a (finite) {\it puncture} to be the homotopy class of a left-turn  path in 
	$\widetilde{\Pi}_1^{\mathcal \G^*}$, which starts and ends at the same object and has no self-intersections. Two overlapping punctures are considered to be identical. 
	
	An {\it infinite puncture} is an infinite nested sequence of homotopy classes of left-turning paths with non-repeating start and end points in $\widetilde{\Pi}_1^{\mathcal \G^*}$.

	Denote by ${\G}^\circ$ the set of punctures of $\G^*$. Hence to any modular graph we can associate two invariants $(g,n)$ where $g$ is its genus and $n$ is the cardinality of ${\G}^\circ$, the set of punctures (which may be finite or infinite). Note that $\G^*$ may have punctures even though $\G^*$ is simply connected, e.g. consider $\G=\F$, the Farey tree  (see Section 5.1). It has infinitely many punctures at the boundary.
	
	Let $\aut^\circ(\widetilde{\Pi}_1^{\G^*})$ be the group of groupoid isomorphisms preserving the set of finite punctures. Since a finite puncture is an element of $\widetilde{\Pi}_1^{\G^*}$, this definition makes sense.
		
		A quick experiment shows that flips preserve punctures, i.e. the set $\G_1^\circ$ is sent to the set
		$\G_2^\circ$ under a flip. Since the same is obviously true for permutations of labels, morphisms of $\pimcg$ preserve this extra structure,
		i.e. we may revise  (\ref{miracle}) as
		\begin{equation}\label{miracle2}
			\pmg(S)=\mathrm{Isot}({\pimcg}, |\G^*,\ell|)\to
			{\textcolor{green}
				{\aut^\circ(\widetilde{\Pi}_1^{\G^*})}}\to
			\out^\circ(\pi_1(\G^* ,x))
		\end{equation}
		where 
		$\out^\circ(\pi_1(\G^*,x))$ is the group of 
		automorphisms of the fundamental group
		preserving the conjugacy classes of punctures; modulo inner automorphisms. Since for $n,g<\infty$ the group $\mod(S)$ is known to be isomorphic to the group of puncture-preserving outer automorphisms of $\pi_1(S)$, the composition of the morphisms in~(\ref{miracle2}) defines an epimorphism
		\begin{equation}\label{modtoout2}
			\pmg(S)\to \out^\circ(\pi_1(\G^*,x)),
		\end{equation}
		for each pair $|\G^*, \ell|$  and for each edge-point $x$ of $\G^*$.

		\section{Full fundamental modular groupoid}
		In the previous section, under the hypothesis that $S$ is a surface of finite type (i.e. $g,n<\infty$), we have introduced the fundamental modular groupoid $\pimcg(S)$ of $S$, which is isomorphic to $\mcg_\ell(S)$. On the other hand, Definition~\ref{fmg} can be restated solely in terms of the modular graph $\G^*$, without referring to the surface $S$ or its type, as follows:
		\begin{definition}
			{\rm The } fundamental modular groupoid {\rm is the disconnected groupoid $\pimcg$ whose 
				\label{fmgr} object set is 
				${\mathrm{MGR}_\ell}$, which is defined to be set of all labeled modular graphs up to isomorphism. 
				Morphisms of $\pimcg$ are (exactly as in (\ref{grrou})  ):
				\begin{equation}\label{grrou1}
					\mor_{\pimcg}(|\G_1^*,{\ell_1}|,|\G_2^*,{\ell_2}|):=
					\left\{ 
					\begin{array}{l} 
						\mbox{isomorphisms } \widetilde{\Pi}_1^{\G_1^*}\to \widetilde{\Pi}_1^{\G_2^*} \mbox{ induced  }\\
						\mbox{by finite sequences of flips \&}\\
						\mbox{permutations of labels}
					\end{array}
					\right\}
			\end{equation}}
		\end{definition}


			Here, $|\G_1^*, \ell_1|$ and $|\G_2^*, \ell_2|$ are  in ${\mathrm{MGR}_\ell}$, these are labeled modular graphs with at most countably many edges.
			Since $\mathrm{MGR}_\ell$ is an uncountable set, $\pimcg$ has uncountably many objects. Since finite sequences of flips are at most countable in number,  $\pimcg$ has uncountably many connected components. If $|\G^*,\ell|\in \mathrm{MGR}_\ell $ then by $\pimcg(\G^*,{\ell})$ we denote the connected component of $\pimcg$ containing 
			$|\G^*,\ell|$. 
			Thanks to relabelings, this connected component is same as $\pimcg(\G^*,\ell')$, where $\ell'$ is any other labeling, and we will unambiguously denote  both as $\pimcg(\G^*)$. Finally, if $S$ is a surface of finite type, then by $\pimcg(S)$ we denote $\pimcg(\G^*)$ where $\G\hookrightarrow S$ is any trivalent spine of $S$. This notation is compatible with our previous definition of $\pimcg(S)$.
			
			\begin{figure}[h!]
				\begin{center}
					\includegraphics[width=2.5cm]{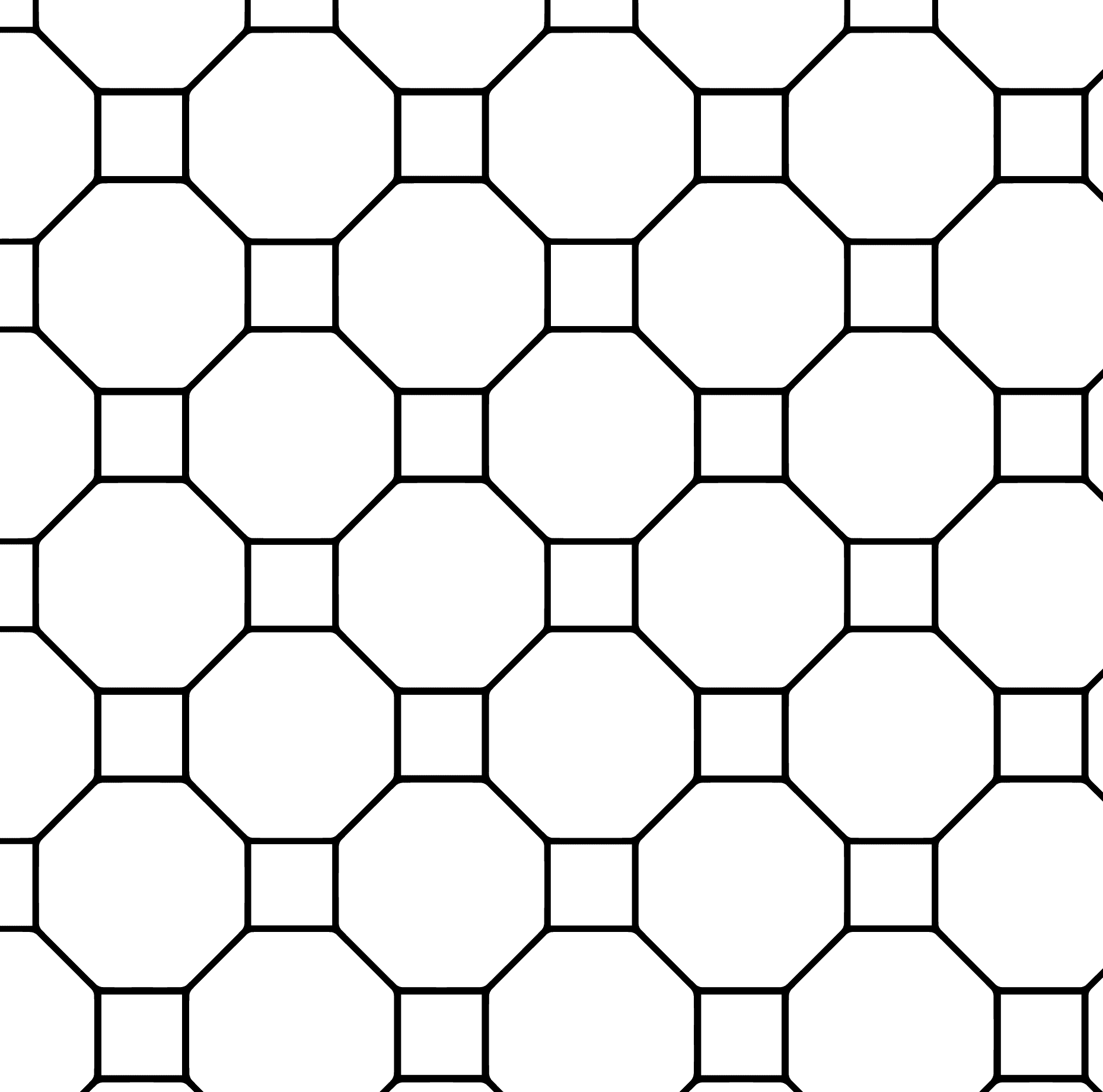}\qquad
					\includegraphics[width=2.5cm]{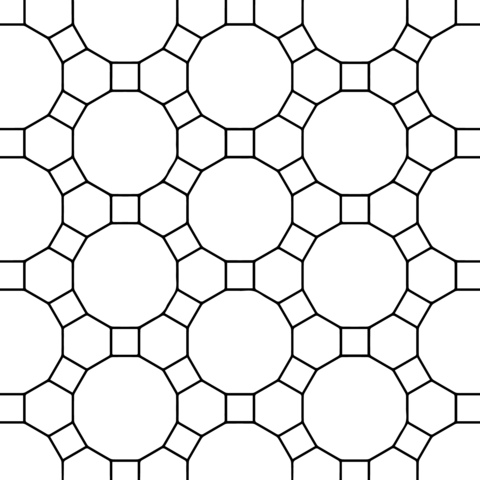}\qquad
					\includegraphics[width=2.9cm]{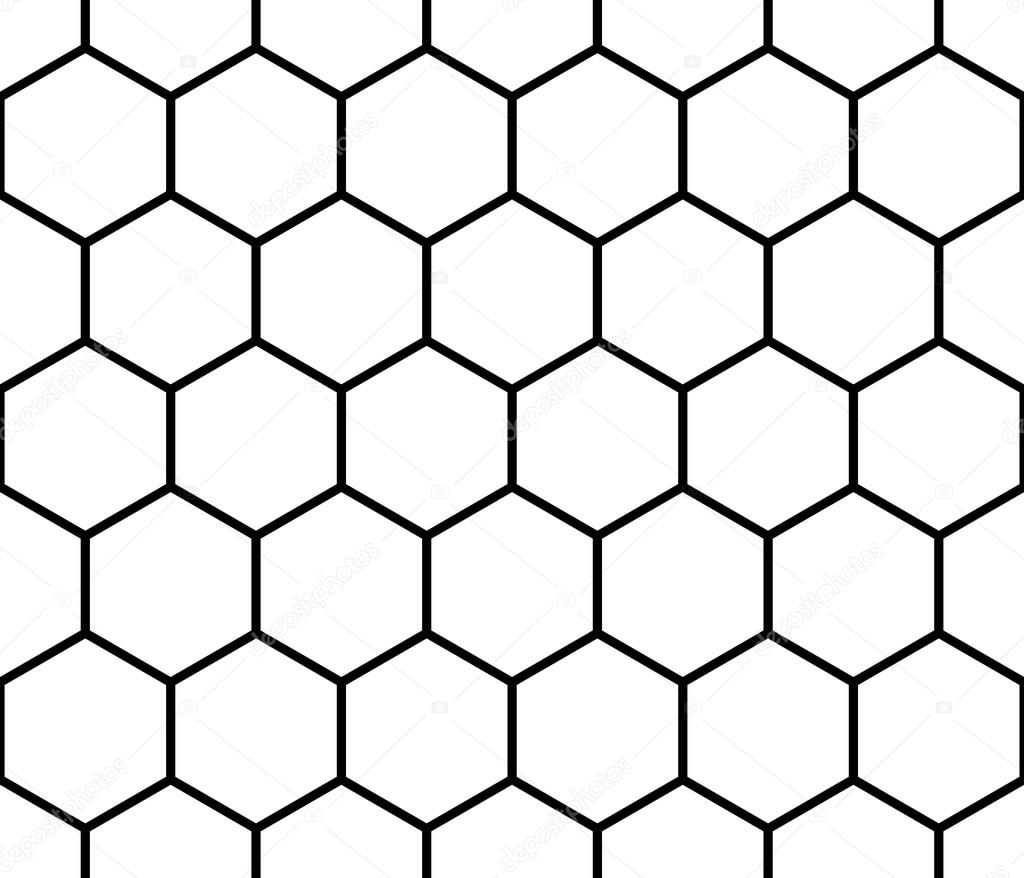}
				\end{center}
				\caption{Modular graphs in different flip orbits (connected components of $\pimcg$) with $g=0$, $n=\infty$. In these examples, there are no punctures at infinity, and there are infinitely many modular graphs in the flip-orbit. Farey tree (Figure~\ref{fareytree}) has no finite punctures, has infinitely many punctures at infinity, and is the only modular graph in its flip orbit.}
				\label{cccp}
			\end{figure}
			
			Since the isotropy group $\pmg(S)$ of $\pimcg(S)$ surjects onto $\mod(S)$, the isotropy groups of connected components of $\pimcg$ gives analogues of -extended- mapping class groups; there exists an uncountable number of them. Naturally, these will not be finitely generated nor presented in general. Figure~\ref{cccp} exhibits graphs from three different connected components.			Note that the mapping class group of the ambient surface $S^{g=0}_{n=\infty}$ is an uncountable group for such surfaces.
			
			Since $\pimcg(\G^* )$ is a connected groupoid, its isotropy groups are (non-canonically) isomorphic. This justifies the following the notation $\pmg(\G^* )$ for the isotropy group of $\pimcg(\G^* )$, which we shall adopt from this point on.
			
			Note that we still have the sequence~(\ref{miracle2}) (this time for graphs instead of surfaces):
			\begin{equation*}
				\pmg(\G^*){:=}\mathrm{Isot}(\pimcg, |\G^*,\ell |)\to
				{\textcolor{green}
					{\aut^\circ(\widetilde{\Pi}_1^{\G^*})}}\to
				\out^\circ(\pi_1(\G^*,x)).
			\end{equation*}

			\paragraph{Question.} Given some graph $\G^*$ as in Figure~\ref{cccp}, find a presentation of the  group $\mod(\G^* )$.

			\section{Farey tree and Thompson's group}
			Recall that, in case $\G^*$ is a finite modular graph, 
			the groupoid $\pimcg(\G^*)$ (i.e. the connected component of $\pimcg$ containing $\G$) admits the usual mapping class group $\mod(S)$ 
			as a quotient of its isotropy group $\pmg(S)$, where $\G$ is a spine of $S$. This section is devoted to exemplifying connected components of $\pimcg$, containing infinite modular graphs as their objects.
			
			Fundamental groups of the graphs in Figure~\ref{cccp} are not finitely generated. It is more interesting to start with infinite graphs of finite topology (i.e. with finitely generated fundamental groups). The simplest such graph is the Farey tree $\F$, which is the plane trivalent tree with no terminal vertices. In general, these are finite graphs with a bunch of infinite tree branches attached to them.
			
			\subsection{Farey tree and modular graphs}\label{fareytreee}
			We repeat some definitions from~\cite{UZD}. 
			Let  $\psl$ be the modular group. It is well-known that
			$\psl = \langle S\rangle \ast \langle L\rangle \cong\Z/2\Z \ast \Z / 3\Z,$ where $S (z)=-1/z$ and $L (z)= (z-1)/z$, where we view $S$ and $L$ as acting on the complex 
			upper half plane $\mathcal H$.
			
			From $\psl$ we construct the (bipartite) Farey tree $\F$ via
			\begin{eqnarray*}
				\mbox{edges of }\F: \quad E  (\F) & = & \{\{W\} \colon W \in \psl \} \\
				\mbox{vertices of }\F:\quad V  (\F) & = & V_{{\tens}}  (\F) \sqcup V_{\oasterisk}  (\F);
			\end{eqnarray*}
			\noindent where 
			\begin{eqnarray*}
				V_{\tens}  (\F) & = & \{ \{W,WS\} \colon W \in \psl \}, \\
				V_{\oasterisk}  (\F) & = & \{ \{W,WL,WL^{2}\} \colon W \in \psl \}.
			\end{eqnarray*}
			The edges incident to the vertex $\{W,WS\} \in V_{\tens}$ are $ \{W\}$ and $\{WS\} $.
			The edges incident to the vertex $\{W,WL,WL^{2}\} \in V_{\oasterisk}$ are $ \{W\}$, $\{WL\}$, and $\{WL^{2}\} $, and these edges inherit a natural cyclic ordering from the vertex. Thus the Farey tree $\F$ is an infinite bipartite ribbon graph, and $\psl$ acts on $\F$ from the left, by ribbon graph automorphisms, where $M\in \psl$ sends
			the edge $\{W\}$ to $ \{MW\}$.

			Note that the automorphism group $\aut(\F)$ is the free product $\Z/2\Z*\Z/3\Z\simeq \psl$, generated by an order-2 rotation around an edge and an order-3 rotation around a vertex adjacent to this edge.

			Let $\Gamma$ be any subgroup of $\psl$. Then $\Gamma$ acts on $\F$ from the left and to $\Gamma$ we associate a quotient graph $\Gamma\backslash\F$ as follows:
			
			\medskip\noindent
			\hspace{1.5cm}	 $E  (\Gamma\backslash\F) = \{\Gamma \ccdot \{W\} \colon W \in \psl \}$
			
			\noindent
			\hspace{1.5cm}	 $V  (\Gamma\backslash\F) = V_{\tens}  (\F/\Gamma) \cup V_{\oasterisk}  (\F/\Gamma)$;
			
			\noindent where 
			
			\noindent
			\hspace{1.5cm}	$V_{\tens}  (\Gamma\backslash\F) = \{ \Gamma \ccdot\{W , WS\} \colon W \in \psl \}$, and
			
			\noindent
			\hspace{1.5cm}	$V_{\oasterisk}  (\Gamma\backslash\F) = \{ \Gamma \ccdot\{W, WL, WL^{2}\} \colon W \in \psl \}$.

			\medskip\noindent The incidence relation induced from the Farey tree gives a well-defined incidence relation and gives us the  bipartite graph.
			If the group $\Gamma$ is torsion-free, then
			it is easy to see that $\Gamma\backslash\F$ has no pending edges. Hence, $\Gamma\backslash\F$
			is a {modular graph}. Conversely, any connected bipartite ribbon graph $G$, with $V (G)=V_{\tens}  (G) \sqcup V_{\oasterisk}  (G)$, such that every $\tens$-vertex is of degree 2 and every $\oasterisk$-vertex is of degree 3, is modular since the universal covering of $G$ is isomorphic to $\F$.

			The set of edges of $\Gamma\backslash\F$ is identified with the set of right-cosets of $\Gamma$, so that the graph $\Gamma\backslash\F$ has $[\psl : \Gamma]$ many edges. In case $\Gamma$ is a finite index subgroup, the graph $\Gamma\backslash\F$ is a finite modular graph.
			
			It takes a little effort to define the fundamental group of $\Gamma\backslash\F$ so that there is a canonical isomorphism $\pi_1 (\Gamma\backslash\F, \Gamma\ccdot \{I\})\simeq \Gamma<\psl$, with the canonical choice of $\Gamma\ccdot \{I\}$ as a base edge. In general, subgroups $\Gamma$ of the modular group  (or equivalently the fundamental groups $\pi_1 (\Gamma\backslash\F)$) are free products of copies of $\Z$, $\Z/2\Z$ and $\Z/3\Z$, see \cite{kulkarni}.  Note that two distinct isomorphic subgroups $\Gamma_1$, $\Gamma_2$ of the modular group may give rise to non-isomorphic ribbon graphs $\Gamma_1\backslash\F$ and $\Gamma_2\backslash\F$. In other words, the fundamental group does not characterize the graph. Another basic invariant of $\Gamma\backslash\F$ is its genus, which is defined to be the genus of the surface constructed by gluing discs along left-turn paths.  This genus is the same as the genus of the Riemann surface $\Gamma \backslash{\mathcal H} $,
			where $\mathcal H$ is the upper half plane. Finally, the number of punctures  is an invariant of $\Gamma\backslash\F$.

			\begin{proposition}
				If $\Gamma_1$ and $\Gamma_2$ are conjugate subgroups of $\psl$, then the graphs $\Gamma_1\backslash\F$ and $\Gamma_2\backslash\F$ are isomorphic as ribbon graphs. Hence there is a 1-1 correspondence between modular graphs and conjugacy classes of (torsion free) subgroups of the modular group. 
				\label{conjugatesubgroups}
			\end{proposition}

			\begin{theorem} \label{graphswithadoe}
				There is a 1-1 correspondence between modular graphs with a base edge $(\mathcal G,e)$  (modulo ribbon graph isomorphisms of pairs $ (\mathcal G,e)$) and subgroups of the modular group). 
				\label{thm:modular/graph/vs/subgroups}
			\end{theorem}
The $\psl$-action on the set of subgroups corresponds to the 
base-edge-moving action of $\psl$ on the set $\overrightarrow{\mathrm{MGR}}$ of modular graphs with a base edge.
			\begin{theorem} \label{cosetsss}
				{There is a 1-1 correspondence between modular graphs with two base edges $ (\mathcal G,e,e')$  (modulo ribbon graph isomorphisms of triples $ (\mathcal G,e,e')$) and cosets of subgroups of the modular group. 
					\label{thm:modular/graph/vs/cosets}}
			\end{theorem}

To get a description of the labeled modular graphs up to isomorphism, let $(\G, \ell)$ be a such graph. For simplicity, we assume that the labels are from the set $\{1,2,\dots |E(\G)|\}$ (and not from the set $\{1,2,\dots |E(\G)|/2\}\times
\{+,-\}$ as we used above).
We construct a pair of elements $(\sigma, \lambda)$ of the permutation group 
$\mathsf{Sym}(|E(\G)|)$ as follows.
\begin{itemize}
\item $\sigma$ is the product of all label transpositions 
$(i,j)$ such that the corresponding edges of $\G$ meet at a vertex of degree 2;
\item $\lambda$ is the product of all 3-cycles of transpositions $(i,j,k)$ of labels such that the corresponding edges of $\G$ meet at a vertex of degree 3, with this orientation.
\end{itemize}
One has then $\sigma^2=\lambda^3=1$. Note that the permutation group $\langle \sigma, \lambda \rangle<\mathsf{Sym}(|E(\G)|)$ is transitive since $\G$ is connected. 

Conversely, given $N$ let us call a {\it 2-3 permutation pair} a pair $(\sigma, \lambda)$ of elements of 
$\mathsf{Sym}(N)$ with $\sigma^2=\lambda^3=1$ 
and such that $\langle \sigma, \lambda \rangle$ is transitive. 
Note that any 2-3 permutation pair defines a transitive permutation representation 
$\psl\to\mathsf{Sym}(N)$ sending $S\to \sigma$ and $L \to\lambda$. The group $\psl$ acts on such representations by conjugation.

Given a 2-3 permutation pair, we can construct a labeled modular graph 
$(\G, \ell)$ by taking a collection of edges labeled $1,2,\dots N$; gluing the edges labeled $i$ and $j$ at a degree-2 vertex whenever $(i,j)$ is a transposition in $\sigma$; and finally gluing the edges labeled $i$, $j$ and $k$ at a degree-3 vertex with orientation $i\to j\to k$ whenever $(i,j,k)$ is a 3-cycle in $\lambda$ (pending edges of $\G$ or equivalently the torsion elements of a corresponding subgroup are avoided by asking $\sigma$ and $\lambda$ to be free of 1-cycles\footnote{Much of the ideas in this paper carry over to modular graphs with pending edges; equivalently to subgroups of $\psl$ with torsion; equivalently to 2-3 permutation pairs with fixed points. These correspond to surfaces $S$ with punctures and orbifold points with inertia $\Z/2\Z$ and $\Z/3\Z$.}). The graph $\G$ so obtained with its labeling is then connected since $\langle \sigma, \lambda \rangle$ is transitive.
Observe that the cycle decomposition of the product $\sigma\lambda$ then corresponds to the puncture set of $\G$  (see \cite{millington}).

Hence the following characterization of our basic object set $\mathrm{MGR}_\ell$.

\begin{theorem}
There is a 1-1 correspondence between 
the elements of $\mathrm{MGR}_\ell$ with $N$ edges 
($1\leq N \leq \infty$) and the set of pairs of fixed point free 2-3 permutation pairs; i.e. permutations 
$(\sigma, \lambda)$ from
$\mathsf{Sym}(N)$ with $\sigma^2=\lambda^3=1$ 
and such that $\langle \sigma, \lambda \rangle$ is transitive.
\end{theorem}

By Theorem \ref{graphswithadoe}, a subgroup $\mathsf G<\psl$ is specified by the choice of an edge of $\G$; which we universally pick to be the one labeled 1. By Theorem \ref{cosetsss}, all other edges of $\G$ corresponds to the cosets of $\mathsf G$. This shows that an element of 
$\mathrm{MGR}_\ell$ is nothing but a subgroup $\mathsf G<\psl$ with an enumeration of cosets of $\mathsf G$ in $\psl$, such that the coset $\mathsf G$ itself is labeled 1. Elements of 
$\widetilde{\Pi}_1^{\G}$ can then be interpreted as certain sequences of enumerated cosets of $\mathsf G$ (of arbitrary length); encoding paths on the corresponding graph.
As a side remark, note that $\psl$ acts on such enumerated cosets (i.e. on $\mathrm{MGR}_\ell$) on by conjugation from the left and from the right, though these won't extend to an action on paths.

In this description, a labeled flip $\psi: (\G, \ell)\to(\G', \ell')$ corresponding to the 2-3 permutation pair $(\sigma, \lambda)$ is specified by the choice of a transposition $(c,d)$ of $\sigma$. 
This flip then erases a pair of 3-cycles $(a,b,c)$ and $(d,e,f)$ from $\lambda$ and inserts $(a,d,f)$ and $(c,b,e)$ to give a new permutation $\lambda'$ of order 3. The final 2-3 permutation is then $(\sigma, \lambda')$. One can check that if $(\sigma, \lambda)$ is transitive then so is  $(\sigma, \lambda')$. There is a similar interpretation of the corresponding isomorphism 
$\psi^*:\widetilde{\Pi}_1^{\G}\to \widetilde{\Pi}_1^{\G'}$ in terms of coset sequences.

			\tikzset{
				common/.style={draw,name=#1,node contents={}, inner sep=0, minimum size=1.5},
				disc/.style={circle, fill=red, radius=.1mm, common=#1},
				square/.style={rectangle,common={#1}},
			}
			\begin{figure}
				\begin{center}
					\begin{tikzpicture}[scale=.6]
						\draw (0,0) node[disc=c-0-1];
						\xdef\radius{0cm}
						\xdef\level{0}
						\xdef\nbnodes{1}
						\xdef\degree{(3+1)} 
						\foreach \ndegree/\form in {3/disc,3/disc,3/disc,3/disc,3/disc,3/disc,3/disc}{
							\pgfmathsetmacro\nlevel{int(\level+1)}
							\pgfmathsetmacro\nnbnodes{int(\nbnodes*(\degree-1))}
							\pgfmathsetmacro\nradius{\radius+1cm}
							\foreach \div in {1,...,\nnbnodes} {
								\pgfmathtruncatemacro\src{((\div+\degree-2)/(\degree-1))}
								\path (c-0-1) ++({\div*(360/\nnbnodes)-180/\nnbnodes}:\nradius pt) node[\form=c-\nlevel-\div];
								\draw (c-\level-\src) -- (c-\nlevel-\div);
							}
							\xdef\radius{\nradius}
							\xdef\level{\nlevel}
							\xdef\nbnodes{\nnbnodes}
							\xdef\degree{\ndegree}
						}
				\end{tikzpicture}\end{center}
				\caption{The 3-valent Farey tree. To obtain the 2-3-valent graph, one must put extra vertices at midpoints of its edges.}
				\label{fareytree}
			\end{figure}

			\subsection{The fundamental modular groupoid of the Farey tree}
			Recall that the groupoid $\pimcg(\F)$ is connected by definition. It has uncountably many objects, consisting of all labeled trees 
			$|\F, \ell|$, where $\ell:\F\to\mathbb N\times\{+,-\} $ is a labeling. In other words, the object set 
			of $\pimcg(\F)$ is
			$$
			\mathrm{Obj}(\pimcg(\F))=\{(\F, \ell)\, |\, \ell:E(\mathcal \F)  \to \mathbb N \mbox{ is a bijection}\}/\psl.
			$$
			Since among the morphisms of $\pimcg(\F)$, in particular we have all relabelings, 
			the set of morphisms of $\pimcg(\F)$ is uncountable, too\footnote{The situation is similar if we have an infinite modular graph $\G$ in place of $\F$.}.
			However, the situation is not so desperate as it seems; since the set of morphisms among any pair of  objects $\pimcg(\F)$ is countable by definition 
			(cf. ``{finite sequences of}'' in (\ref{grrou1})). In particular, $\pmg(\F)$, 
			i.e. the
			isotropy group of any of its objects is the same (up to isomorphism) countable group. Any element of $\pmg(\F)$ can be uniquely written in the form $\lambda\varphi $, 
			where $\varphi$ is the composition of a finite sequence of flips followed by a relabeling $\lambda$. 
			
			Let us now show that $\pmg(\F)$ is an extension of $\ppsl$ (which is isomorphic to the famed Thompson's group $\mathsf{PPL_2(Z)}$) with a group of permutations of finite support on the set $\psl$.
			
			\definecolor{green}{RGB}{0, 153, 51}

			The tree $\F$ being simply connected, 
			$\widetilde{\Pi}_1^{\mathcal F}$ has trivial isotropy, and the last element of the sequence in (\ref{miracle}) vanish, yielding for any labeling $\ell$ the inclusion map
			\begin{equation*}
				\pmg(\F):= 
				\mathrm{Isot}(\pimcg, |\mathcal F, \ell |)\to
				{\textcolor{green}
					{\aut(\widetilde{\Pi}_1^{\mathcal F}) }}.
			\end{equation*}

			Since $\F$ is simply connected, there exists a unique path connecting any pair of edges, and the 
			groupoid $\widetilde{\Pi}_1^{\mathcal F}$ is identified with the pair groupoid $E(\F)\times E(\F)$. Recall that by definition $E(\F)$ is identified with $\psl$. Hence, 
			${\textcolor{green}
				{\aut(\widetilde{\Pi}_1^{\mathcal F})}}\simeq \mathsf{Sym}(E(\F))= \mathsf{Sym}(\psl)$ is an uncountable group. Note that 
			$\aut^!(\F)<{\textcolor{green}
				{\aut(\widetilde{\Pi}_1^{\mathcal F})}}$, where  $\aut^!(\F)$ is the group of graph automorphisms not necessarily preserving the ribbon structure (this is a locally compact, uncountable group).
			To sum up this discussion, we have
			$$
			\pmg(\F)<\mathsf{Sym}(\psl) \simeq {\textcolor{green}
				{\aut(\widetilde{\Pi}_1^{\mathcal F})}}.
			$$
			
			Let now $|\F, \ell|$ be an object of $\pimcg(\F)$. 
			Let us first show that the group of ribbon graph automorphisms $\aut(\F)\simeq \psl$ is
			a subgroup of the isotropy group of $|\F, \ell|$.
			Let $\alpha\in \aut(\F)$. Then $(\F, \ell)$  and $(\F, \ell\circ \alpha)$ represent the same object of 
			${\mathrm{MGR}_\ell}$, identified via 
			$\alpha^{-1}:\F\to \F$. 
			The induced automorphism 
			$\widetilde{\Pi}_1^{\mathcal F}
			\to \widetilde{\Pi}_1^{\mathcal F}$ simply sends a path to its image under $\alpha^{-1}$.
			To sum up, we have
			$$
			\psl\simeq \aut(\F)<\pmg(\F).
			$$
			(The same argument shows that $\aut(\G)<\pmg(\G)$ for every modular graph $\G$.)
			Note that $\aut(\F)=\aut^!(\F)\cap \pmg(\F)$, because elements in $\aut^!(\F)\backslash \aut(\F)$ does not respect the ribbon graph structure, whereas flips, although not being graph automorphisms, respects the ribbon graph structure.
			
			\subsubsection{The group of partial automorphisms}
			Recall that $\pmg(\F)$ is the isotropy group of $|\F,\ell|$ in $\pimcg$ for some labeling $\ell$. From now on we fix one such labeling.
			
			Now consider the subgroup 
			$\mathsf{Sym}^\pm(\psl)<
			\mathsf{Sym}(\psl)$ which consists of 
			permutations $\pi$ of finite support,  such that for all $W\in \psl$ one has 
			$$\{\pi(W), \pi(WS)\}=\{W', W'S\}$$ for some $W'$ (of course, depending on $\pi$ and $W$). 
			This means that on the Farey tree, $\pi$ should send twin edges to twin edges, i.e. those edges which meet at a degree-2 vertex. 
			Note that 
			$\mathsf{Sym}^\pm(\psl)<\pmg(\F)$ since any of the two `{transposition}' of  two given twin pairs of edges can be written in terms of flips.

			Moreover,
			$$
			\mathsf{Sym}^\pm(\psl)\vartriangleleft \pmg(\F),
			$$
			since the conjugate of a finite permutation is again a finite permutation. Note that an element of $\mathsf{Sym}^\pm(\psl)$ may permute two edges incident at a degree-two vertex.

			A {\it ribbon graph near-automorphism} of $\F$
			is a bijection $\sigma: E(\F) \to E(\F)$ such that the restriction
			$\sigma: \F\backslash \mathcal T_1 \to\F\backslash \mathcal T_2 $ is a ribbon graph isomorphism for some finite subtrees $\mathcal T_1, \mathcal T_2$, where we require that $\sigma$ respects the canonical cyclic order of the branches $\F\backslash \mathcal T_1$ and $\F\backslash \mathcal T_2$ induced by the planar structure  (see~\cite{yves}). Thus, $\sigma$ can permute the edges of $\mathcal T_1, \mathcal T_2$ arbitrarily, violating the graph structure (but keeping twin edges together). 
			
			Consider the equivalence relation on the set of near automorphisms; where two such automorphisms are  equivalent if they agree on the complement of a finite subtree. 
			The group of {\it partial $\F$-automorphisms} $\aut^\infty(\F)$  is the equivalence class of near automorphisms under this equivalence \cite{vlad}. The equivalence class of an automorphism is called its {\it germ}. Note that $\aut^\infty(\F)$ is a countable group.
			
			\label{germsof}
			\begin{theorem}
				There is an exact sequence
				\begin{align}\label{symfary}
					1\to \mathsf{Sym}^\pm(\psl) \to \pmg(\F) 
					\stackrel{\gamma}{\to} \aut^\infty(\F) \to 1
				\end{align}
				where $\aut^\infty(\F)$ is the group of partial $\F$-automorphisms and $\gamma$ is the germification map.
			\end{theorem}

			\begin{proof}
				It suffices to prove that $\pmg(\F)$ is the group of near $\F$-automorphisms. To this end, it suffices to observe that every flip followed by the corresponding relabeling is nothing but a near automorphism. Since every transposition of edges can be written as a sequence flips, the kernel is the group 
				$\mathsf{Sym}^\pm(\psl)$.
			\end{proof}

			We may relax the definition of the partial automorphism and consider instead arbitrary automorphisms (not just ribbon graph automorphisms) $\sigma: \F\backslash \mathcal T_1 \to \F\backslash \mathcal T_2$, where $\mathcal T_1, \mathcal T_2$ are  two finite subtrees of $\F$. The resulting group is Neretin's spheromorphism group, which we denote by \label{nerretin}
			$\aut^\infty(\F^c)$ (where the subscript `$c$' stands for ``{combinatorial}''). Observe that $\aut^\infty(\F^c)$ is an uncountable group. Note that there is an obvious analogue of (\ref{symfary}) for regular planar trees.
			
			Denoting the corresponding group of 
			near-automorphisms (not just ribbon graph near-automorphisms) of $\F$ by $\pmg(\F^c)$, we get the analogous exact sequence \label{modfc}
			\begin{align}\label{symfary7}
				1\to \mathsf{Sym}^\pm(\psl) \to \pmg(\F^c) 
				\stackrel{\gamma}{\to} \aut^\infty(\F^c) \to 1.
			\end{align}
			(We don't claim that $\pmg(\F^c)$  is the isotropy group of a groupoid in a way $\pmg(\F)$ is an isotropy group.)
			The group $\aut^\infty(\F^c)$ has a discrete subgroup which is isomorphic to Thompson's group $\mathsf V$, see Section~\ref{outmodgr} below.
			
			\begin{remark}
				Instead of our tree $\F$, it is possible to consider the planar tree $\T$ such that every vertex is of degree 3; i.e. without vertices of degree 2. The group $\pmg(\T)$ of near automorphisms can be defined in the same way. However; 
				since elements of $\pmg(\F)$ are allowed to permute twins; i.e. two edges that meet at a vertex of degree $2$, the group $\pmg(\F)$ is much bigger than the group $\pmg(\T)$. Nevertheless, the groups $\aut^\infty(\F)$ and $\aut^\infty(\T)$ are isomorphic.
			\end{remark}
			
			\begin{remark}
				There is a direct analogue of (\ref{symfary7}) for a regular tree where each vertex is of degree $p$, which provides the extension of Neretin's $p-$adic spheromorphism group by an infinite symmetric group.
			\end{remark}
			
			\subsubsection{Thompson's group and $\pmg(\F)$}
			Now let $|\F, \ell|$ be an object of $\pimcg(\F)$.
			We use the edge labeled $e:=(1,+)$ to identify the boundary $\partial|\F, \ell|$ with the circle $S^1_e$ via continued fractions~\cite{modulargroupanditsactions}. 
			We then extend the $\pimcg(\F)$ action on 
			fundamental groupoids (which consists of finite paths) to the boundaries (which consists of infinite paths). This defines an epimorphism
			$$
			\partial: \pimcg(\F) \to \ppsl.
			$$
			Since the relabelings has no effect on the paths whatsoever, the kernel of $\partial$ consists of the relabeling morphisms 
			$(|\F, \ell|, |\F, \ell'|)\in \pimcg(\F)$ 
			(observe that the set of relabeling morphisms does not constitute a subgroup of $\pimcg$).
			
			As for the flips, let 
			$\varphi:|\F, \ell|\to |\F, \ell'|$ be one.
			It induces a map 
			$$
			\partial\varphi: \partial|\F, \ell|\to \partial|\F, \ell'|,
			$$
			and since both 
			$\partial|\F, \ell| $ and 
			$\partial|\F, \ell'| $ are identified with $S^1$, 
			$\varphi$ induces a map
			$$
			\partial\varphi: S^1\to S^1,
			$$
			which is piece-wise $\psl$. 
			Restricting $\partial$ to the subgroup 
			$\pmg(\F)$ for some $\ell$, we get the following result.

			\begin{lemma} (\cite{vlad})
				$\aut^\infty(\F)$ is isomorphic to $\ppsl$. 
			\end{lemma}
			\label{thmpgr}
			\begin{proof}
				Any element of $\psl$ can be described as a finite composition of flips on $\F$. First, let's consider a single flip. Let $\phi:E\to E$ denote the permutation determined by a flip. In Figure \ref{xcxcxxx}, we observe that $\phi$ maps the branches stemming from the edges $a,b,c$ and $d$ on the left isomorphically to the branches stemming from the edges $a',b',c'$ and $d'$, respectively, on the right. Hence, for this map, we can choose $\mathcal T$ in the definition of the related element of $\aut^\infty(\F)$ as the subgraph consisting of the edges $a,b,c,d,e$ and $f$. For a composition of flips, $\mathcal T$ can be chosen as the smallest subtree containing the inverse images of the subtrees described for a single flip, by the permutations of prior flips.
				
				Any element of $\ppsl$ can be described by a finite sequence of pairs of intervals in $S^1$. The map between any such pair is given by a Möbius transformation. For any pair of intervals $I, I'\subset S^1$, we can consider the branches $\B, \B'$ of $\F$ such that $\partial \B=I$ and $\partial \B' = I'$. For such a pair $(I,I')$ there is a unique ribbon graph isomorphism sending $\B$ to $\B'$. Then, we complete these isomorphisms into a permutation of $E\to E$ by gluing the ribbon graph isomorphisms. Notice that the complementary tree $\mathcal T$ is finite and thus the resulting element is in $\aut^\infty(\mathcal F)$. This procedure describes an inverse map to the one in the previous paragraph.
			\end{proof}
			
			Recall that  $\ppsl$  is isomorphic to Thompson's group $\mathsf{PPL_2(Z)}$ via the Minkowski question mark function. This was conjectured by Kontsevich and proved by Imbert \cite{imbert}.
			In fact the last arrow in (\ref{symfary}) depends on our identification of $\partial \F$ with the circle; the continued fraction identification makes appear the group $\ppsl$ and the dyadic identification makes the appear group $\mathsf{PPL_2(Z)}$ as the last term in (\ref{symfary}).

			\section{Bushes}
			We say that $\G$ is of {\it finite topology} if $\pi_1(\G)$ is finitely generated. An infinite modular graph of finite topology is called a  {\it bush}. The simplest bush is the Farey tree $\F$, which has a trivial fundamental group. It is easy to see that any bush has a well-defined genus $g$ and number of finite punctures $n$, which are both finite and flip-invariant.
			
			A {\it branch} is a cycle-free connected graph, which is modular except that it has a single vertex of degree 1, called its {\it root}, such that the vertex adjacent to the root is of degree 2. 
			
			Unless $\pi_1(\G)\simeq\Z$, every bush $\G$ defines an underlying finite modular graph 
						$u(\G)$.  The graph $u(\G)$ can be obtained by plucking the maximal branches from $\G$; 
{removing the vertex that is root of the branch together with the two neighboring half-edges and identifying the two hanging 2-valent vertices, as shown in Figure \ref{fig:plucking}.}
			(If $\pi_1(\G)=\Z$, in which case $\G$ is called a {\it çark}, $u(G)$ is a `{graph}' which consists of just one cycle and no nodes.) Besides $n$ and $g$, a bush $\G$ has a third characteristic, which is the number of punctures of $u(\G)$ containing a branch, denoted $b$. These latter punctures are called the {\it boundary cycles} of $\G$. As it is the case with the punctures, each flip induces a well-defined bijection on sets of boundary cycles.

	\begin{figure}
		\newcommand{\depth}{1}
		\newcommand{\Step}{1}
		
	\begin{center}
	\begin{tikzpicture}
		\draw[emphline]\vertex{0}{1}--(2,-1);
		\draw[emphline](1,-1)--(3,-1);
		\foreach \level in {0,...,\depth}
		{
			\newcommand{\LevelNodeNumber}{\NiL\level}
			\foreach \i in {1,...,\LevelNodeNumber}
			{
				{
					\ifthenelse{\level=\depth}
					{
						\draw[emphline]\vertex{\level+1}{2*\i}--\vertex{2+\level}{4*\i-1};
						\draw[emphline]\vertex{\level+1}{2*\i}--\vertex{2+\level}{4*\i};
						\draw[emphline]\vertex{\level+1}{2*\i-1}--\vertex{2+\level}{4*\i-2};
						\draw[emphline]\vertex{\level+1}{2*\i-1}--\vertex{2+\level}{4*\i-3};
						\ghostfork{1+\level}{2*\i}
						\ghostfork{1+\level}{(2*\i-1)}
					}
					{}
				}
				\draw[emphline]\vertex{\level}{\i}--\vertex{1+\level}{2*\i};
				\draw[emphline]\vertex{\level}{\i}--\vertex{1+\level}{2*\i-1};
				\edge{\vertex{\level}{\i}}{\vertex{1+\level}{2*\i}}
				\edge{\vertex{\level}{\i}}{\vertex{1+\level}{2*\i-1}}
			}
		}
		\edge{\vertex{0}{1}}{(2,-1)}
		\edge{(0,-1)}{(2,-1)}
		\edge{(4,-1)}{(2,-1)}
		\node at (0,-.9) [anchor=south] {$a$};
		\node at (4,-.9) [anchor=south] {$b$};
		
		\node at (5,0) [anchor=south, scale=2] {$\rightsquigarrow$};
		
		\edge{(10,-1)}{(6,-1)}
		\node at (10,-.9) [anchor=south] {$b$};
		\node at (6,-.9) [anchor=south] {$a$};
		
	\end{tikzpicture}
	\end{center}
		\caption{Plucking a maximal branch from a bush to form $u(\G)$.}
		\label{fig:plucking}
	\end{figure}
			
			To each boundary cycle  of the bush $\G$ we associate a cyclic vector $(x_1, x_2, \dots, x_k)$ called the {\it boundary necklace} as follows: Travel along the boundary cycle in the positive sense,  starting from an arbitrary edge.
			Set $x_1=1$ if the first degree-3 node encountered is the root of a branch to the right with respect to the travel direction and $x_1=0$ otherwise. Then repeat the procedure. (In case $\G$ is a çark, there are by definition two boundary cycles, with complementary boundary necklaces).
			
			The number $b$ of boundary cycles is invariant under the flips on $\G$. On the other hand, boundary necklaces themselves are not flip-invariant.
			
			\subsection{Çarks}
			Modular graphs with fundamental groups isomorphic to $\Z$ provides the simplest\footnote{If we allow pending edges; then the rooted trees $\F/\langle S\rangle$ and $\F/\langle L\rangle$ are simpler.} examples of bushes after the Farey tree $\F$ itself.
			They can be described as follows.
			
			A {\it hyperbolic (resp. parabolic)} {\it \c{c}ark} is a modular graph of the form $\mathcal{C}_M:=\langle M \rangle \backslash \F$ where $M\in \psl$ is hyperbolic (resp. parabolic). Since
			$$
			\pi_1 (\langle M \rangle\backslash \F, M)=\langle M \rangle\simeq \Z,
			$$ 
			 the \c{c}ark $\langle W \rangle\backslash \F$ is a graph with only one circuit, which we call its {\it spine}. 
			By Proposition~\ref{conjugatesubgroups} the graphs $\mathcal{C}_M$ and $\mathcal{C}_{XMX^{-1}}$ are isomorphic for every $X\in\psl$ and from Theorem~\ref{thm:modular/graph/vs/subgroups} we deduce the following result:
			
			\begin{theorem} (\cite{merve})
				There are one-to-one correspondences between
				\begin{itemize}
					\item[i.] hyperbolic (resp. parabolic) \c{c}arks and conjugacy classes of subgroups of the modular group generated by a single hyperbolic (resp. parabolic) element, and
					\item[ii.] hyperbolic (resp. parabolic) \c{c}arks with a base edge and subgroups of the modular group generated by a single hyperbolic (resp. parabolic) element.
				\end{itemize}
				\label{cor:1/to/1/corr/between/undirected/carks}
			\end{theorem}
			
			\begin{figure}[h!]
				\centering
				\includegraphics[scale=0.23]{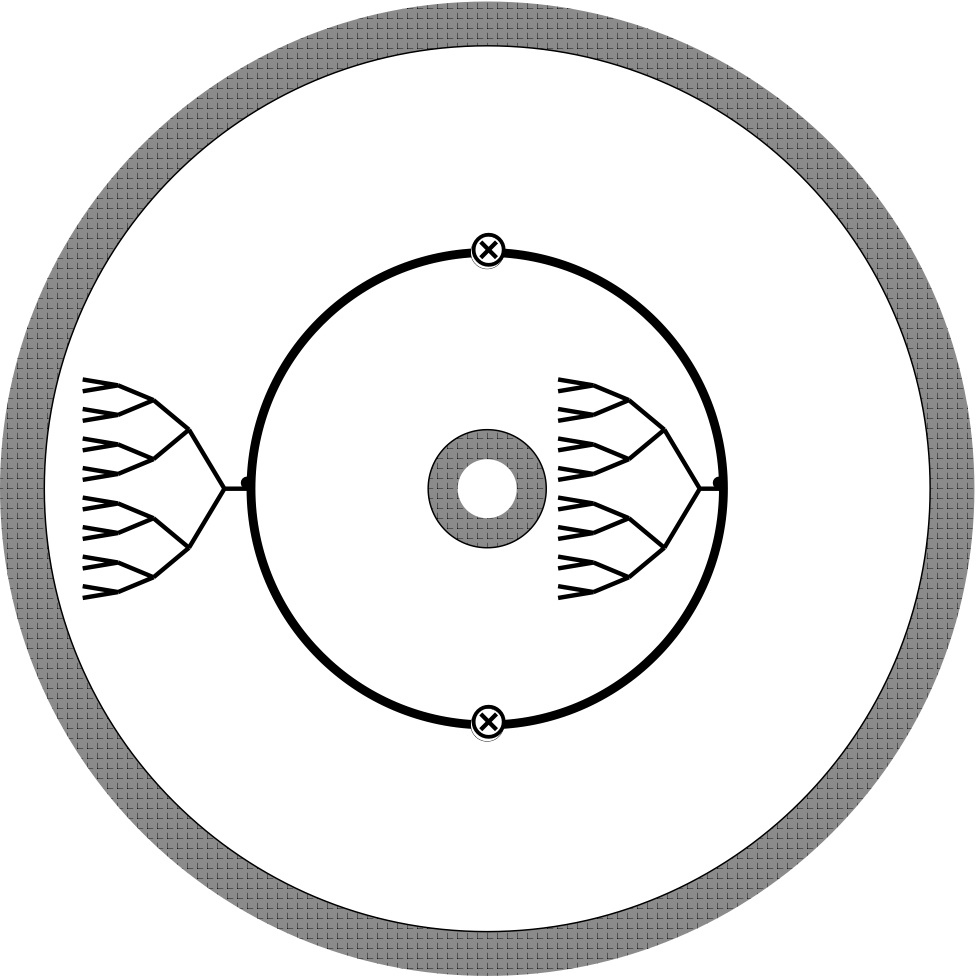}
				\caption{ The \c hyperbolic \c{c}ark $\F/\langle S L^2 S L \rangle$.}
				\label{fig:cark/1/-1}
			\end{figure}

			A \c{c}ark is said to be {\it directed }if we choose an orientation for its spine. 
			
			\begin{corollary} 
				There are one-to-one correspondences between
				\begin{itemize}
					\item[i.] hyperbolic (resp. parabolic) elements in the modular group and directed hyperbolic (resp. parabolic) \c carks with a base edge, and
					\item[ii.] conjugacy classes of hyperbolic (resp. parabolic) elements in the modular group and directed hyperbolic (resp. parabolic) \c{c}arks.
				\end{itemize}
				\label{cor:1/to/1/corr/between/directed/carks}
			\end{corollary}
			\vspace{-5mm}

			\subsection{Connected components of bushes in $\pimcg$}
	
	Let $\G$ be a bush of genus $g$ with $n$ punctures, $b$-many of which contain a branch. In this section we prove that the connected component of $\G$ in $\pimcg$ is determined by $(n,g,b)$. For this purpose, we first prove the following lemma.
	
	\begin{lemma} \label{lem:TransposePunctures}
d		Given a finite modular graph and two of its punctures, there exists a finite sequence of flips and relabelings that result in an isomorphic  modular graph, transposing the given two punctures and fixing the others.
	\end{lemma}
	
	\begin{proof}
		Let $\H$ be a finite modular graph with genus $g$ and with $n$ punctures.
		In case $n=1$ there is nothing to prove. For $n=2$, the surface has a spine as depicted in Figure~\ref{fig:TwoPunctureSurface}. By Lemma~\ref{whitehead}, we can transform $\H$ to this spine by a sequence of flips. We observe that a rotational symmetry gives an automorphism of the spine interchanging the punctures, thus the related relabeling transposes the punctures. A symmetric picture also exists for the case $(g,n)=(0,3)$ in the form of a sphere divided into three punctures by three equidistant meridians.
		\begin{figure}
			\begin{center}
				\tikzset{
					vviz/.style
					={green, fill=green},
					vshad/.style
					={green, fill=white},
					eviz/.style
					={thick, green},
					evizx/.style
					={thick, red},
					eshad/.style
					={green, very thin},
					eshadx/.style
					={red, very thin},
					surf/.style
					={thick}
				}
				\begin{tikzpicture}[%
					scale=1.4,
					decoration={markings,
						mark=at position 0.5 with {\coordinate (M);
							;}
					}]
					
					
					\draw[surf] (-1,1) -- (-3,1);
					\draw[surf,postaction={decorate}] (-3,1) .. controls (-5,1) and (-5,-1) .. (-3,-1);
					\draw[evizx] (-3.55, .05) .. controls (-3.6,0.1) and ($(M)+(0,.4)$) .. (M);
					\draw[eshadx] (-3.57, .06) 
					.. controls ++(.2,-.2) and ++(0,-.1) 
					.. (-3.38, -.23);
					\draw[eshadx] (-3.34, .13) .. controls (-3.5,0) and ($(M)-(0,.5)$) .. (M);
					\draw[evizx] (-3.35, .12) 
					.. controls ++(.1,.1) and ++(0,.06) 
					.. (-3.4, .24);
					\draw[surf]     (-3,-1) -- (-1,-1);
					\draw[dotted] (-1,-1) -- (1,-1);
					\draw[surf] (1,-1) -- (3,-1);
					\draw[postaction={decorate},surf]  (3,-1) .. controls (5,-1) and (5,1) .. (3,1);
					\draw[evizx] (3.35, .25) .. controls (3.5,.4) and ($(M)+(0,.4)$) .. (M);
					\draw[eshadx] (3.35, -.25) .. controls (3.5,-.4) and ($(M)-(0,.5)$) .. (M);
					\draw[surf] (3,1) -- (1,1);
					\draw[dotted] (1,1) -- (-1,1);
					
					\begin{scope}[shift={(3,0)}]
						\draw[eviz] (.55,.08) arc (45:135:.8 and 1);
						\draw[vviz] (.35, .25) circle (.05);
						\draw[eshad] (.55,.08) arc (20:-200:.6 and .27);
						\draw[vshad] (.35, -.23) circle (.05);
						\draw[surf] (.44,0) arc (45:135:.64 and .8);
						\draw[surf] (.7,.21) arc (315:225:1 and 1.2);
					\end{scope}

					\begin{scope}[shift={(-3,0)}]
						\draw[eviz] (.55,.08) arc (45:135:.8 and 1);
						\draw[vviz] (-.38, .25) circle (.05);
						\draw[eshad] (.55,.08) arc (20:-200:.6 and .27);
						\draw[vshad] (-.38, -.23) circle (.05);
						\draw[surf] (.44,0) arc (45:135:.64 and .8);
						\draw[surf] (.7,.21) arc (315:225:1 and 1.2);
					\end{scope}
					
					\begin{scope}[shift={(-1,0)}]
						\draw[eviz] (-.55,.08) arc (135:90:.8 and 1);
						\draw[eshad] (-.55,.08) arc (160:270:.6 and .27);
						\draw[surf] (-.44,0) arc (135:90:.64 and .8);
						\draw[surf] (-.7,.21) arc (225:270:1 and 1.2);
						
						\draw[eshadx] 
						(.35-2, -.23) .. controls (.69-2,-.35) and (1-2,-.35) .. (-.4, .06);
						\draw[evizx] 
						(-.38, .25) .. controls ++(.1,0) and ++(.0,-.05) .. (-.4, .07);
						
						\draw[evizx] 
						(.37-2, .23) .. controls (.69-2,.35) and (1-2,.25) .. (-.4, -.06);
						\draw[eshadx] 
						(-.38, -.23) .. controls ++(.1,.1) and ++(-.0,.05) .. (-.41, -.07);
						
						\draw[vviz] (.37-2, .25) circle (.05);
						\draw[vshad] (.35-2, -.23) circle (.05);
						\draw[vshad] (-.38, -.23) circle (.05);
						\draw[vviz] (-.36, .25) circle (.05);
					\end{scope}
					
					\begin{scope}[shift={(1,0)}]
						\draw[eviz] (.55,.08) arc (45:90:.8 and 1);
						\draw[eshad] (.55,.08) arc (20:-90:.6 and .27);
						\draw[surf] (.44,0) arc (45:90:.64 and .8);
						\draw[surf] (.7,.21) arc (315:270:1 and 1.2);
						
						\draw[eshadx] 
						(.35, -.23) .. controls (.69,-.35) and (1,-.35) .. (2-.4, .08);
						\draw[evizx] 
						(2-.38, .25) .. controls ++(.1,0) and ++(.0,-.05) .. (2-.4, .09);
						
						\draw[evizx] 
						(.35, .23) .. controls (.69,.35) and (1,.25) .. (2-.4, -.06);
						\draw[eshadx] 
						(2-.38, -.25) .. controls ++(.1,0) and ++(-.0,.05) .. (2-.41, -.06);
						
						\draw[vviz] (.35, .25) circle (.05);
						\draw[vshad] (.35, -.23) circle (.05);
						\draw[vshad] (2-.38, -.23) circle (.05);
						\draw[vviz] (2-.38, .25) circle (.05);
					\end{scope}
				\end{tikzpicture}
				\caption{A symmetric spine for surface of genus $g$ with two punctures. Both green and red curves denote full edges.}
				\label{fig:TwoPunctureSurface}
			\end{center}
		\end{figure}
		
		We assume that $\H$ is a finite modular graph such that $n\geq 3$ and $(g,n)\neq (0,3)$. We choose a puncture of $\H$, with at least two full-edges. We will  show that we can reduce the length of this puncture via flips so that the puncture is eventually comprised of just one edge (By puncture length, we refer to the number of full-edges, i.e. pairs of twin edges, on the puncture).
		
		\begin{enumerate}[label=\emph{Case \Roman*:}]
			\item
			The puncture contains full-edges whose ends are the same vertex.\\ In this case, we first choose $e$ as one of these edges. Let's call a subgraph consisting of two edges, one forming a loop and the other attached to the only end of the first one, a \emph{lollipop}. Since the ends of $e$ are the same vertex; $e$ forms a loop, giving rise to a lollipop. Consider the leftmost graph in Figure \ref{fig:ShortenPunctureSame}. The puncture of interest here is the region above the horizontal line and outside the loop, where $e$ is the edge forming the loop. The operation depicted in Figure \ref{fig:ShortenPunctureSame} can be viewed as a maneuver that relocates a lollipop within a puncture, sticking to an adjacent neighboring puncture. We observe that applying flips to the highlighted edges, we obtain the rightmost graph in the figure. If the region below the horizontal line is different from the region above the line, then this operation shortens the puncture of interest by 3 edges. In case the punctures are the same, then we first move the lollipop along the puncture using flips as shown in Figure~\ref{fig:LollipopMove} until the neighboring puncture is different. Repeating the procedure, we reach a graph with no full-edge whose ends are the same vertex.
			\begin{figure}
				\begin{center}
					\begin{tikzpicture}[scale=1]
						\tikzset
						{
							emphline/.style=
							{line width=2mm,yellow},
							emphline2/.style=
							{line width=2mm,pink}
						}
						
						\draw[dotted](0,-1)--(-.5,-1);
						\draw[dotted](2,-1)--(2.5,-1);
						\draw[emphline] (1,-1) -- (1,0);
						\halfedge{(1,-1)}{(0,-1)}
						\halfedge{(1,-1)}{(2,-1)}
						\edge{(1,-1)}{(1,0)}
						
						\def\factor{.7}
						\halfedgebezier{(1,0)}{(1,1)}{+(-\factor,0) and +(-\factor,0)} 
						\halfedgebezier{(1,0)}{(1,1)}{+(\factor,0) and +(\factor,0)} 
						
						\node at (3,-1) [scale=2] {$\rightsquigarrow$};

						\begin{scope}[shift={(4,0)}]
							\draw[dotted](0,-1)--(-.5,-1);
							\draw[dotted](3,-1)--(3.5,-1);
							\draw[emphline] (1,-1) .. controls +(0,\factor) and +(0,\factor) .. (2,-1);
							\halfedge{(1,-1)}{(0,-1)}
							\halfedge{(2,-1)}{(3,-1)}
							\draw[fedge] (1,-1) .. controls +(0,\factor) and +(0,\factor) .. (2,-1); 
							\draw[fedge] (1,-1) .. controls +(0,-\factor) and +(0,-\factor) .. (2,-1);
						\end{scope}
						
						\node at (8,-1) [scale=2] {$\rightsquigarrow$};
						
						\begin{scope}[shift={(9,0)}]
							\draw[dotted](0,-1)--(-.5,-1);
							\draw[dotted](2,-1)--(2.5,-1);
							\halfedge{(1,-1)}{(0,-1)}
							\halfedge{(1,-1)}{(2,-1)}
							\edge{(1,-1)}{(1,-2)}
							
							\halfedgebezier{(1,-2)}{(1,-3)}{+(\factor,0) and +(\factor,0)} 
							\halfedgebezier{(1,-2)}{(1,-3)}{+(-\factor,0) and +(-\factor,0)}
						\end{scope}
					\end{tikzpicture}
				\end{center}
				\caption{Shortening a puncture after a choice of an edge that forms a loop. At each step, a flip is applied on the highlighted edge. Note that the puncture bounding the region above the line and outside the loop in shortened, and the \emph{lollipop} is moved to the puncture below the horizontal line.}
				\label{fig:ShortenPunctureSame}
			\end{figure}
			\item For all edges of the puncture, the ends are different vertices.\\
			In this case, it is easy to observe that performing a flip on an edge $e$ shortens both neighboring punctures of $e$ by one, and increases the length of the other two punctures which are adjacent to the end vertices of $e$ by one. In Figure~\ref{fig:FlipRegions}, the puncture of interest is $A$, shortened punctures are $A$ and $B$, and lengthened punctures are $C$ and $D$. If the puncture of interest has no edge whose other neighbor is also the same puncture, then performing a flip on that edge would shorten the puncture. If some both neighbors are the puncture of interest for some of the edges, i.e. $B=A$ in the Figure~\ref{fig:FlipRegions}, then we choose an edge $e$ such that $B=A$ and at least one of $C$ or $D$ is different from $A$. Such an edge exist because $n>1$.
		\end{enumerate}
		\begin{figure}
			\begin{center}
				\begin{tikzpicture}
					\tikzset
					{
						emphline/.style=
						{line width=2mm,yellow},
					}
					
					\path[fill=yellow!50] (-1.83, 1.66) -- (-1,0) -- (1,0) -- (1.83, 1.66);
					\draw[dotted] (-1.5,1) -- (-1.83, 1.66);
					\draw[dotted] (-1.5,-1) -- (-1.83, -1.66);
					\draw[dotted] (1.5,1) -- (1.83, 1.66);
					\draw[dotted] (1.5,-1) -- (1.83, -1.66);
					
					\node at(0,.3){$e$};
					\edge{(-1,0)}{(1,0)}
					\halfedge{(-1,0)}{(-1.5,1)}
					\halfedge{(-1,0)}{(-1.5,-1)}
					\halfedge{(1,0)}{(1.5,1)}
					\halfedge{(1,0)}{(1.5,-1)}
					\node at (0,1.3){$A$};
					\node at (0,-1){$B$};
					\node at (-2,0){$C$};
					\node at (2,0){$D$};			
					
				\end{tikzpicture}
			\end{center}
			\caption{Punctures around an edge. The puncture to be shortened is $A$, colored in yellow.}
			\label{fig:FlipRegions}
		\end{figure}
		Note that each lollipop itself represents a puncture of length 1, and vice versa. Since we can shorten any puncture of length at least 2, we can turn any puncture into a lollipop by flips.

		As we see in Figure \ref{fig:LollipopMove}  we can move a lollipop within a puncture. Considering the ability to relocate the lollipop to an adjacent puncture, as illustrated in Figure \ref{fig:ShortenPunctureSame}, we have the freedom to select any lollipop, along with its accompanying half-edges, and position it at the 2-valent vertex of any twin edge, orienting it toward the neighboring puncture of choice. 
		
		Given any two punctures $A$ and $B$ of a modular graph $\H$, by flips we can transform them into lollipops and move them to adjacent edges as they face opposite directions, using the procedures described above. Then, as shown in Figure \ref{fig:LollipopSwitch}, we can swap them. After this, reversing the flips returns us to the original modular graph $\H$, only the punctures $A$ and $B$ changing places.
		
		\begin{figure}
			\begin{center}
				\begin{tikzpicture}
					\tikzset
					{
						emphline/.style=
						{line width=2mm,yellow},
					}
					
					\def\factor{.7}
					
					\begin{scope}[shift={(1,-1)}]   
						\edge{(0,0)}{(0,1)}
						\halfedgebezier{(0,1)}{(0,2)}{+(-\factor,0) and +(-\factor,0)} 
						\halfedgebezier{(0,1)}{(0,2)}{+(\factor,0) and +(\factor,0)}
					\end{scope}
					\begin{scope}[shift={(2,-1)}]   
						\edge{(0,0)}{(0,-1)}
						\halfedgebezier{(0,-1)} {(0,-2)}{+(\factor,0) and +(\factor,0)} 
						\halfedgebezier{(0,-1)}{(0,-2)}{+(-\factor,0) and +(-\factor,0)}
					\end{scope}
					\draw[dotted] (0,-1) -- (-.5,-1);
					\draw[dotted] (3,-1) -- (3.5,-1);
					\draw[emphline] (1,-1) -- (2,-1);
					\edge{(0,-1)}{(1,-1)}
					\edge{(1,-1)}{(2,-1)}
					\edge{(2,-1)}{(3,-1)}
					\node at (1,.5) {$A$};
					\node at (2,-2.5) {$B$};

					\node at (4,-1) [scale=2] {$\rightsquigarrow$};

					\begin{scope}[shift={(5,0)}]
						\begin{scope}[shift={(2,-1)}]   
							\edge{(0,0)}{(0,1)}
							\halfedgebezier{(0,1)}{(0,2)}{+(-\factor,0) and +(-\factor,0)} 
							\halfedgebezier{(0,1)}{(0,2)}{+(\factor,0) and +(\factor,0)}
						\end{scope}
						\begin{scope}[shift={(1,-1)}]   
							\edge{(0,0)}{(0,-1)}
							\halfedgebezier{(0,-1)} {(0,-2)}{+(\factor,0) and +(\factor,0)} 
							\halfedgebezier{(0,-1)}{(0,-2)}{+(-\factor,0) and +(-\factor,0)}
						\end{scope}
						\draw[dotted] (0,-1) -- (-.5,-1);
						\draw[dotted] (3,-1) -- (3.5,-1);
						\edge{(0,-1)}{(1,-1)}
						\edge{(1,-1)}{(2,-1)}
						\edge{(2,-1)}{(3,-1)}
						\node at (2,.5) {$A$};
						\node at (1,-2.5) {$B$};
						
					\end{scope}
					
					\node at (9,-1) [scale=2] {$\rightsquigarrow$};

					\begin{scope}[shift={(10,0)}]
						\begin{scope}[shift={(1,-1)}]   
							\edge{(0,0)}{(0,1)}
							\halfedgebezier{(0,1)}{(0,2)}{+(-\factor,0) and +(-\factor,0)} 
							\halfedgebezier{(0,1)}{(0,2)}{+(\factor,0) and +(\factor,0)}
						\end{scope}
						\begin{scope}[shift={(2,-1)}]   
							\edge{(0,0)}{(0,-1)}
							\halfedgebezier{(0,-1)} {(0,-2)}{+(\factor,0) and +(\factor,0)} 
							\halfedgebezier{(0,-1)}{(0,-2)}{+(-\factor,0) and +(-\factor,0)}
						\end{scope}
						\draw[dotted] (0,-1) -- (-.5,-1);
						\draw[dotted] (3,-1) -- (3.5,-1);
						\edge{(0,-1)}{(1,-1)}
						\edge{(1,-1)}{(2,-1)}
						\edge{(2,-1)}{(3,-1)}
						\node at (1,.5) {$B$};
						\node at (2,-2.5) {$A$};
					\end{scope}
				\end{tikzpicture}
			\end{center}
			\caption{Swapping lollipops. First transformation involves a flip on the highlighted edge. The second step is performing the operation described in Figure \ref{fig:LollipopMove} for both lollipops.}
			\label{fig:LollipopSwitch}
		\end{figure}	
	\end{proof}
	
	\begin{proposition}
		Any two labeled bushes with characteristics $(n,g,b)$ can be connected via a finite sequence of flips or relabelings.
	\end{proposition}
	
	\begin{proof} 
	Since relabelings are morphisms in $\pimcg$, it suffices to show that any two (non-labeled) bushes with the characteristics $(n,g,b)$ are related via flips.
	
		Let $\G$ be a bush.
		\begin{itemize}
			\item
			If a puncture of $u(\G)$ contains a branch, then that branch can be moved within the puncture by flips, sliding over the neighboring vertices if there is no branch rooted at that vertex pointing into the aforementioned puncture; as shown in Figure \ref{fig:BranchMove}. 
			\item 
			If two neighboring 3-valent vertices have branches pointing into the puncture, then these branches can be joined to form a single maximal branch by a flip, as shown in Figure \ref{fig:BranchJoin}.
		\end{itemize}

		\begin{figure}
			\newcommand{\depth}{1}
			\newcommand{\Step}{1}
			\begin{center}
				\begin{tikzpicture}[scale=.8]
					\tikzset
					{
						emphline/.style=
						{line width=2mm,yellow},
					}
					
					\edge{\vertex{0}{1}}{\vertex{1}{1}}
					\edge{\vertex{0}{1}}{\vertex{1}{2}}
					\foreach \level in {1,...,\depth}
					{
						\newcommand{\LevelNodeNumber}{\NiL\level}
						\foreach \i in {1,...,\LevelNodeNumber}
						{
							{
								\ifthenelse{\level=\depth}
								{
									\ghostfork{1+\level}{2*\i}
									\ghostfork{1+\level}{(2*\i-1)}
								}
								{}
							}
							\edge{\vertex{\level}{\i}}{\vertex{1+\level}{2*\i}}
							\edge{\vertex{\level}{\i}}{\vertex{1+\level}{2*\i-1}}
						}
					}
					\draw[emphline] (2,-1) -- (4,-1);
					\edge{\vertex{0}{1}}{(2,-1)}
					\edge{(0,-1)}{(2,-1)}
					\edge{(2,-1)}{(4,-1)}
					\edge{(4,-1)}{(6,-1)}
					\draw[dotted] (4,-2) -- (4,-2.5);
					\halfedge{(4,-1)}{(4,-2)}	
					
					\node at (7,0) [anchor=south, scale=2] {$\rightsquigarrow$};

					\begin{scope}[shift={(10,0)}]
						\edge{\vertex{0}{1}}{\vertex{1}{1}}
						\edge{\vertex{0}{1}}{\vertex{1}{2}}
						\foreach \level in {1,...,\depth}
						{
							\newcommand{\LevelNodeNumber}{\NiL\level}
							\foreach \i in {1,...,\LevelNodeNumber}
							{
								{
									\ifthenelse{\level=\depth}
									{
										\ghostfork{1+\level}{2*\i}
										\ghostfork{1+\level}{(2*\i-1)}
									}
									{}
								}
								\edge{\vertex{\level}{\i}}{\vertex{1+\level}{2*\i}}
								\edge{\vertex{\level}{\i}}{\vertex{1+\level}{2*\i-1}}
							}
						}
						\edge{\vertex{0}{1}}{(2,-1)}
						\edge{(-2,-1)}{(0,-1)}
						\edge{(0,-1)}{(2,-1)}
						\edge{(2,-1)}{(4,-1)}
						\draw[dotted] (0,-2) -- (0,-2.5);
						\halfedge{(0,-1)}{(0,-2)}
					\end{scope}
				\end{tikzpicture}
			\end{center}
			\caption{By a flip on the highlighted edge int the graph on the left, the branch rooted at vertex $b$ slides over the vertex $c$. }
			\label{fig:BranchMove}
		\end{figure}

		\begin{figure}
			\newcommand{\depth}{1}
			\newcommand{\Step}{1}
			\begin{center}
				\begin{tikzpicture}[scale=.8]
					\tikzset
					{
						emphline/.style=
						{line width=2mm,yellow},
						emphline2/.style=
						{line width=2mm,pink}
					}
					
					\edge{\vertex{0}{1}}{\vertex{1}{1}}
					\edge{\vertex{0}{1}}{\vertex{1}{2}}
					\foreach \level in {1,...,\depth}
					{
						\newcommand{\LevelNodeNumber}{\NiL\level}
						\foreach \i in {1,...,\LevelNodeNumber}
						{
							{
								\ifthenelse{\level=\depth}
								{
									\ghostfork{1+\level}{2*\i}
									\ghostfork{1+\level}{(2*\i-1)}
								}
								{}
							}
							\edge{\vertex{\level}{\i}}{\vertex{1+\level}{2*\i}}
							\edge{\vertex{\level}{\i}}{\vertex{1+\level}{2*\i-1}}
						}
					}
					\edge{\vertex{0}{1}}{(2,-1)}
					
					\begin{scope}[shift={(4,0)}]
						\edge{\vertex{0}{1}}{\vertex{1}{1}}
						\edge{\vertex{0}{1}}{\vertex{1}{2}}
						\foreach \level in {1,...,\depth}
						{
							\newcommand{\LevelNodeNumber}{\NiL\level}
							\foreach \i in {1,...,\LevelNodeNumber}
							{
								{
									\ifthenelse{\level=\depth}
									{
										\ghostfork{1+\level}{2*\i}
										\ghostfork{1+\level}{(2*\i-1)}
									}
									{}
								}
								\edge{\vertex{\level}{\i}}{\vertex{1+\level}{2*\i}}
								\edge{\vertex{\level}{\i}}{\vertex{1+\level}{2*\i-1}}
							}
						}
						\edge{\vertex{0}{1}}{(2,-1)}
					\end{scope}
					
					\draw[dotted](0,-1)--(-1,-1);
					\draw[dotted](8,-1)--(9,-1);
					\draw[emphline] (2,-1) -- (6,-1);
					\halfedge{(2,-1)}{(0,-1)}
					\edge{(2,-1)}{(6,-1)}
					\halfedge{(6,-1)}{(8,-1)}

					\node at (9.5,0) [anchor=south, scale=2] {$\rightsquigarrow$};

					\begin{scope}[shift={(11,0)}]
						\draw[emphline2] (2,-1) -- (2,0);
						\edge{\vertex{0}{1}}{\vertex{1}{1}}
						\edge{\vertex{0}{1}}{\vertex{1}{2}}
						\foreach \level in {1,...,\depth}
						{
							\newcommand{\LevelNodeNumber}{\NiL\level}
							\foreach \i in {1,...,\LevelNodeNumber}
							{
								{
									\ifthenelse{\level=\depth}
									{
										\ghostfork{1+\level}{2*\i}
										\ghostfork{1+\level}{(2*\i-1)}
									}
									{}
								}
								\edge{\vertex{\level}{\i}}{\vertex{1+\level}{2*\i}}
								\edge{\vertex{\level}{\i}}{\vertex{1+\level}{2*\i-1}}
							}
						}
						\edge{\vertex{0}{1}}{(2,-1)}
						
						\draw[dotted](0,-1)--(-1,-1);
						\draw[dotted](4,-1)--(5,-1);
						\halfedge{(2,-1)}{(0,-1)}
						\halfedge{(2,-1)}{(4,-1)}
					\end{scope}
				\end{tikzpicture}
			\end{center}
			\caption{Performing a flip on the highlighted edge in the graph on left results in the graph on right, where the edge on which the flip is applied corresponds to the pink edge and the two edges now forms a single maximal branch.}
			\label{fig:BranchJoin}
		\end{figure}

	Taking these two operations into account, we see we can modify $\G$ in such a way that any puncture of $u(\G)$ containing a branch in $\G$ will have a single maximal branch in it, and this singular branch can be freely relocated within the puncture. Hence, the connected component containing a bush $\G$ in $\pimcg$ is determined by $u(\G)$ and marking the punctures of $u(\G)$ that include a branch in $\G$, i.e. by a finite modular graph and a choice of a subset of its punctures. By Lemma~\ref{lem:TransposePunctures} we know that we can interchange any two punctures in a finite modular graph using flips and relabelings, and therefore we can conclude that two markings of punctures for a finite modular graph can be obtained from one another if and only if the number of marked punctures are the same.

	\begin{figure}
		\begin{center}
			\begin{tikzpicture}
				\tikzset
				{
					emphline/.style=
					{line width=2mm,yellow},
				}
				
				\def\factor{.7}
				
				\begin{scope}[shift={(1,-1)}]   
					\edge{(0,0)}{(0,1)}
					\halfedgebezier{(0,1)}{(0,2)}{+(-\factor,0) and +(-\factor,0)} 
					\halfedgebezier{(0,1)}{(0,2)}{+(\factor,0) and +(\factor,0)}
				\end{scope}
				\draw[dotted] (2,-1) -- (2,-2);
				\draw[dotted] (0,-1) -- (-.5,-1);
				\draw[dotted] (3,-1) -- (3.5,-1);
				\draw[emphline] (1,-1) -- (2,-1);
				\edge{(0,-1)}{(1,-1)}
				\edge{(1,-1)}{(2,-1)}
				\edge{(2,-1)}{(3,-1)}

				\node at (4,-1) [scale=2] {$\rightsquigarrow$};

				\begin{scope}[shift={(5,0)}]
					\begin{scope}[shift={(2,-1)}]   
						\edge{(0,0)}{(0,1)}
						\halfedgebezier{(0,1)}{(0,2)}{+(-\factor,0) and +(-\factor,0)} 
						\halfedgebezier{(0,1)}{(0,2)}{+(\factor,0) and +(\factor,0)}
					\end{scope}
					
					\draw[dotted] (1,-1) -- (1,-2);
					\draw[dotted] (0,-1) -- (-.5,-1);
					\draw[dotted] (3,-1) -- (3.5,-1);
					\edge{(0,-1)}{(1,-1)}
					\edge{(1,-1)}{(2,-1)}
					\edge{(2,-1)}{(3,-1)}

				\end{scope}
			\end{tikzpicture}
		\end{center}
		\caption{Moving lollipops.}
		\label{fig:LollipopMove}
	\end{figure}

	\end{proof}
	\subsection{Near isomorphism groupoids and $\pmg(\G)$ of a bush $\G$}
			We want to generalize the characterization of $\pmg(\F)$ as the group of near automorphisms of $\F$ to arbitrary $\pmg(\G)$, especially when $\G$ is a bush.
			As a first attempt, we introduce the notion of near ribbon-isomorphisms of labeled bushes.
			Let $\G_1$, $ \G_2$ be two labeled bushes of the same type $(n,g,b)$.
			A {\it near isomorphism} $\G_1 \to \G_2$ is a
			triple $(\sigma, \mathcal S_1, \mathcal S_2)$ where 
			$\mathcal S_1 \subset \G_1$ and $\mathcal S_2 \subset \G_2$ are finite connected subgraphs containing $u(\G_i)$ and $\sigma:E(\G_1)\to E(\G_2)$ is a bijection such that the restriction  
			$$\sigma: \G_1\setminus \mathcal S_1\to 
			\G_2\setminus \mathcal S_2$$  is a label-preserving ribbon graph isomorphism which respects boundary cycles (i.e. $\sigma$ sends the branches in one boundary cycle of $\G_1$ to branches in the same boundary cycle of $\G_2$, respecting their cyclic order on the respective boundary cycles).
			A {\it partial isomorphism} $\G_1 \to \G_2$ is the equivalence class of a near isomorphism, where two near isomorphisms which coincide on a cofinite subgraph are considered equivalent. Note that a rigid graph $\G$ may 
			have lots of near and partial automorphisms.
			
			We consider the groupoid $\mathbf{NMGR}_\ell$ whose objects are labeled bushes and morphisms are near ribbon isomorphisms between them, along with relabelings. Since each flip and relabeling induce a near isomorphism, we have morphisms of groupoids \label{nmgrl}
			\begin{equation*}
				\pimcg^b \to \mathbf{NMGR}_\ell \twoheadrightarrow \mathbf{PMGR}_\ell
			\end{equation*}
			where $\pimcg^b$ is the (disconnected) component of $\pimcg$ which admits bushes as its objects,
			the groupoid $\mathbf{PMGR}_\ell$ has the  labeled modular graphs as its objects and partial isomorphisms as its morphisms. \label{pmgrl}
			Set 
			$$
			\aut^{near}(\G):=\mathrm{Isot}(\mathbf{NMGR}_\ell, |\G, \ell|), \quad 
			\aut^\infty(\G):=\mathrm{Isot}(\mathbf{PMGR}_\ell, |\G, \ell|)
			$$ for some labeling $\ell$.  The latter is called the {\it group of partial automorphisms of $\G$}.
			Since by definition \label{modgdef} $\pmg(\G) :=\mathrm{Isot}(\pimcg, |\G,\ell|)$ for any labeled modular graph $\G$, one has for each labeled $\G$ the induced sequence of isotropy groups of connected components
			$$
			\pmg(\G) 
			\to 
			\aut^{near}(\G)
			\twoheadrightarrow
			\aut^\infty(\G).
			$$
			When $\G=\F$, the first homomorphism above is an isomorphism (cf. Eq.~\ref{symfary}). In fact, we have an isomorphism of connected components
			\begin{equation}\label{seq:nearnear}
				\pimcg(\F) \to \mathbf{NMGR}_\ell(\F) \twoheadrightarrow \mathbf{PMGR}_\ell(\F).
			\end{equation}
			However, when $\G\neq \F$, the first homomorphism in 
			(\ref{seq:nearnear}) is not an isomorphism  and we don't expect an exact sequence similar to (\ref{symfary}). Recovering (\ref{symfary}) requires the notion of groupoids of near and partial isomorphisms of fundamental groupoids of bushes, which we intend to discuss elsewhere. As for the group of partial automorphisms of  a bush $\G$ with characteristics $(n,g,b)$, we conjecture that 
 $\aut^\infty(\G)\simeq \mathsf T^b$ and there is a surjection
				$\mod(\G)\to \mathsf T^b$, where $\mathsf T$ is Thompson's group.

		\section{Outer Modular Groupoids}\label{outmodgr}
			Recall that a modular graph is a graph with a cyclic ordering of edges emanating from each one of its vertices.
			A {\it shuffle} of a modular graph $\G$ is the operation of reversing this orientation at a given vertex, as in 
			Figure~\ref{shufffle}.
			
			\begin{figure}[h!]
				\begin{center}
				\begin{tikzpicture}[font=\small]
					\begin{scope}[shift={(-2,0)},scale=.5]
						\draw [hedge] (-1,0)-- (1.37,0)
							node[pos=1,right]{3};
						\draw [hedge] (-1,0)-- (-2,2)
						node[pos=1,above left]{1};
						\draw [hedge] (-1,0)-- (-2,-2)
						node[pos=1,below left]{2};
					\end{scope}
					\node at (0,0){\huge $\rightsquigarrow$};
					\begin{scope}[shift={(2,0)},scale=.5]
						\draw [hedge] (2,0)-- (-0.3,0)
						node[pos=1,left]{3};
						\draw [hedge] (2,0)-- (3,2)
						node[pos=1,above right]{1};
						\draw [hedge] (2,0)-- (3,-2)
						node[pos=1,below right]{2};
					\end{scope}
				\end{tikzpicture}
				\end{center}
				\caption{A shuffle.}\label{shufffle}
			\end{figure}

			A shuffle of a labeled modular graph is defined likewise.
			Since shuffles do not modify the topological graph underlying the modular graph, they act trivially on fundamental groupoids.
			However, they do change the genus and the number of punctures of $\G$. By applying flips and shuffles to $\G $, we can obtain every trivalent ribbon graph whose fundamental group is isomorphic to $\pi_1(\G )$, provided that $\pi_1(\G )$ is finitely generated. Hence we may define the following (disconnected) groupoid:
			\begin{definition}\label{oomg}
				{\rm The} {\it outer modular groupoid}  {\rm is the groupoid $\omg$ whose object set is ${\mathrm{MGR}_\ell}$. Morphisms of $\omg$ are defined as
					\begin{equation}\label{grroupp}
						\mor_{\omg}(|\G_1^*,\ell_1|,|\G_2^*,\ell_2|):=
						\left\{ 
						\begin{array}{l} 
							\mbox{isomorphisms } \widetilde{\Pi}_1^{\mathcal \G_1}\to \widetilde{\Pi}_1^{\mathcal \G_2} \mbox{ induced by }\\
							\mbox{finite sequences of flips, shuffles \& }\\
							\mbox{permutations of labels}
						\end{array}
						\right\}
				\end{equation}}
			\end{definition}
			Obviously, the groupoids $\pimcg$ and $\omg$ share the same set of objects.
			The natural map $\pimcg\to \omg$ is an inclusion, since the morphisms of both groupoids are $\widetilde{\Pi}_1^{\mathcal \G}$-isomorphisms. In other words,
			$\pimcg < \omg$. Several connected components of $\pimcg$ becomes fused into one connected component inside $\omg$.
			
\begin{example} \label{ex:ThetaShuffles}
Let's consider the \emph{theta graph}, leftmost graph in Figure~\ref{fig:ThetaShuffles}. Applying shuffles on first vertex $x$ and then on vertex $y$, we obtain an isomorphic graph. If we identify the resulting graph with the original one so that this isomorphism fixes the edge $a$, then  we obtain the automorphism of the theta graph that transposes the edges $b$ and $c$. 
\begin{figure}[h!]
\begin{center}
	\begin{tikzpicture}[scale=.8]
	\node at (2,0){$\rightsquigarrow$};
	\node at (7,0){$\rightsquigarrow$};
	\def\factor{2}
	\begin{scope}[shift={(-1,0)},rotate=90]
		\draw[emphline](0,1) .. controls +(-\factor,0) and +(-\factor,0) .. (0,-1 ); 
		\draw[fedge]((0,1) .. controls +(-\factor,0) and +(-\factor,0) .. (0,-1 )
			node[pos=0, left]{$x$}
			node[pos=.5, below]{$a$}
			node[pos=1, right]{$y$};
		\draw[fedge](0,1)--(0,-1 )
			node[pos=.5, above]{$b$};
		\draw[fedge]((0,1) .. controls +(\factor,0) and +(\factor,0) .. (0,-1 )
			node[pos=.5, above]{$c$};
	\end{scope}
	
	\begin{scope}[shift={(5,0)},rotate=90]
		\draw[emphline](0,1) .. controls +(-\factor,0) and +(-\factor,0) .. (0,-1 ); 
		\draw[fedge]((0,1) .. controls +(-\factor,0) and +(-\factor,0) .. (0,-1 )
			node[pos=0, left]{$x$}
			node[pos=.5, below]{$a$}
			node[pos=1, right]{$y$};
		\draw[fedge](0,1) .. controls +(\factor,0) and +(0,.5*\factor) .. (0,-1 )
			node[pos=.25, above]{$b$};
		\fill[white](0.46,0) circle (.1);
		\draw[fedge](0,1) .. controls +(0,-.5*\factor) and +(\factor,0) .. (0,-1 )
			node[pos=.75, above]{$c$};

		\end{scope}
	
		\begin{scope}[shift={(10,0)},rotate=90]
		\draw[emphline](0,1) .. controls +(-\factor,0) and +(-\factor,0) .. (0,-1 ); 
		\draw[fedge]((0,1) .. controls +(-\factor,0) and +(-\factor,0) .. (0,-1 )
			node[pos=0, left]{$x$}
			node[pos=.5, below]{$a$}
			node[pos=1, right]{$y$};
		\draw[fedge](0,1)--(0,-1 )
			node[pos=.5, above]{$c$};
		\draw[fedge]((0,1) .. controls +(\factor,0) and +(\factor,0) .. (0,-1 )
			node[pos=.5, above]{$b$};
		\end{scope}
	
	\end{tikzpicture}
\end{center}
\caption{Two shuffles applied to the theta graph; first to the vertex $x$, then to the vertex $y$.}\label{fig:ThetaShuffles}
\end{figure}
\end{example}

			By $\omg(\G )$, we denote the connected component of $\omg$ containing a labeled graph $|\G, \ell|$ and by $\oomg(\G)$ we denote its isotropy group.
Hence, $\pmg(\G)<\oomg(\G)$ for any $\G$.
			
			In the particular case $\G=\F$, the group $\oomg(\F)$ consists of near automorphisms of $\F$ generated by ribbon graph near-automorphisms and shuffles. One has
			$$
			\pmg(\F)<\oomg(\F)<\pmg(\F^c),
			$$
where $\mod(\F^c)$ is the group of near automorphisms of $\F$ as a combinatorial graph (cf. pg. \pageref{nerretin}).
			One has the exact sequence
				\begin{align}\label{symfary9}
					1\to \mathsf{Sym}^\pm(\psl) \to \out(\F) 
					\stackrel{\gamma}{\to} \aut^\infty_s(\F) \to 1,
				\end{align}
				where $\aut^\infty_s(\F)$ is the group of partial automorphisms of $\F$ generated by ribbon graph near-automorphisms and shuffles.
			\begin{theorem}  
$\aut^\infty_s(\F)$ is isomorphic to Thompson's group $\mathsf V$.
			\end{theorem} 
Observe that $\mathsf T\simeq\aut^\infty(\F)<\aut^\infty_s(\F)<\aut^\infty(\F^c)$, where the latter is Neretin's spheromorphism group.		
\begin{proof}
As in the case of $\pmg(\F)$, it is easy to see that all representatives of an element of a near automorphism in $\oomg(\F)$ induce the same map (ignoring the discontinuity points which are finite in number) on the boundary, where to be definite we use the dyadic model for the boundary. These maps sends the intervals of some finite dyadic subdivision linearly onto the intervals of another, except that via shuffles the order of the intervals may be arbitrarily permuted. But this is exactly the description of elements of Thompson’s group 
$\mathsf V$ (see~\cite{belksthesis}).
\end{proof}

In case $\G$ is finite, the connected component $\omg(\G)$ is same for every $\G=\G_n^g$ with $r=2g+n-1$ which we denote as $\omg(r)$. Hence, $\omg(r)$ contains every $\pimcg(\G_n^g)$ with $r=2g+n-1$. Recall that $\pi_1(\G_n^g)$ is isomorphic to 
$\mathsf F_r$, the free group of rank $r$, and
\begin{equation}
	r=2g+n-1
	\implies 
	\out^\circ \left(\pi_1(\G)\right) \simeq\mod(\G)<\mathrm{Isot}(\omg, |\G,\ell|).
\end{equation}

In general, we have the diagram
\begin{center}
\begin{tikzpicture}[>=stealth, scale=2]
  \node (A) at (2,0) {$\omg(\G)$};
  \node (B) at (0,0) {$\oomg(\G)$};
  \node (C) at (0,1) {$\pmg(\G)$};
  \node (D) at (2,1) {$\pimcg(\G)$};

  \draw[->] (B) -- (A);
  \draw[right hook->] (C) -- (B);
  \draw[->] (C) -- (D);
  \draw[right hook->] (D) -- (A);
\end{tikzpicture}
\end{center}

Recall that there is a surjection $\pmg(\G) \twoheadrightarrow \mod(\G)$. 
In a similar vein, we have the following theorem:
\begin{theorem}\label{thm:OmgSurjection}
			If $\G:=\G_n^g$ is finite,
	then the isotropy group $\oomg(\G)$ surjects onto 
	$\out(\mathsf \pi_1(\G))$; i.e. 
	$\oomg(\G)\twoheadrightarrow\out(\mathsf F_r)$, where 
	$r=2g+ n-1$.
		\end{theorem}	
		It is of interest to know if there is an exact sequence,
\begin{equation*}
	1 
	\rightarrow 
	\mathsf{Sym}^\pm(E(\G))
	\rightarrow
	\oomg(\G)
	\rightarrow
	\out(\mathsf \pi_1(\G))
	\rightarrow
	1
\end{equation*}	
or if the kernel is actually larger than just $\mathsf{Sym}^\pm(E(\G))$.

We will use the following lemma in the proof of the theorem.

\begin{lemma}\label{lem:LollipopTree}
	Given a finite modular graph $\G$ of genus $g$ with $n$ punctures, there exists a modular graph $\H$ such that
	\begin{itemize}
		\item $\H$ is in the same connected component with $\G$ on $\omg$, i.e. there is a morphism $\G \to \H$ in $\omg$.
		\item $\H$ contains a tree whose complement consists of $2g+n-1$ loops.
	\end{itemize}
\end{lemma}

\begin{proof}
The top graph in Figure~\ref{fig:GraphPlanarization} has genus $g$ and $n$ punctures. By Whitehead's Lemma (Lemma~\ref{whitehead}), we can obtain this graph from $\G$ by a sequence of flips. Then, applying shuffles at vertices $x_1,  \dots, x_g$, we obtain the planar graph that is the bottom graph in Figure~\ref{fig:GraphPlanarization}, with $2g+n$ punctures. Then, again by Whitehead's Lemma, applying flips to this planar ribbon graph we can produce a ribbon graph $\H$ which contains a tree whose complement consists of $2g+n-1$ loops. (see~Figure~\ref{fig:LollipopTree}).
\begin{figure}[h!]
	\begin{center}  
		\begin{tikzpicture}[scale=.65,font=\small]
			\def\factor{1.4}
			\def\scopestep{1.6}
			\begin{scope}
				\begin{scope}[shift={(-6*\scopestep,0)}]
					\draw[fedge](0,1) .. controls +(-\factor,0) and +(0-\factor,0) .. (0,-1 )
					node[pos=0,above,minimum size=.6cm]{$x_1$};
					\draw[fedge](0,1) .. controls +(\factor,0) and +(0,\factor) .. (0,-1 );
					\draw[white, fill=white] (.16,-.18) circle (.12);
					\draw[fedge](0,1) .. controls +(-.6*\factor,-\factor) and +(-.4*\factor,-.2*\factor) .. (.7*\factor,0 );
					\draw[fedge](0,-1) .. controls +(.4*\factor,0) and +(0,-.4*\factor,0) .. (.7*\factor,0 );
				\end{scope}
				\begin{scope}[shift={(-5*\scopestep,0)}]
					\draw[fedge](-\scopestep+.7*\factor,0)--(\scopestep-.7*\factor,0);
				\end{scope}
				\begin{scope}[shift={(-4*\scopestep,0)}]
					\draw[fedge](0,1) .. controls +(-.4*\factor,0) and +(0,.4*\factor,0) .. (-.7*\factor,0 )
					node[pos=0,above,minimum size=.6cm]{$x_2$};
					\draw[fedge](0,-1) .. controls +(-.4*\factor,0) and +(0,-.4*\factor,0) .. (-.7*\factor,0 );		\draw[fedge](0,1) .. controls +(\factor,0) and +(0,\factor) .. (0,-1 );
					\draw[white, fill=white] (.16,-.18) circle (.12);
					\draw[fedge](0,1) .. controls +(-.6*\factor,-\factor) and +(-.4*\factor,-.2*\factor) .. (.7*\factor,0 );
					\draw[fedge](0,-1) .. controls +(.4*\factor,0) and +(0,-.4*\factor,0) .. (.7*\factor,0 );
				\end{scope}
				\begin{scope}[shift={(-3*\scopestep,0)}]
					\draw[dotted](\scopestep,0)--(0,0);
					\draw[hedge](-\scopestep+.7*\factor,0)--(0,0);
				\end{scope}
				\begin{scope}[shift={(-2*\scopestep,0)}]
					\draw[dotted](-\scopestep,0)--(0,0);
					\draw[hedge](\scopestep-.7*\factor,0)--(0,0);
				\end{scope}
				\begin{scope}[shift={(-1*\scopestep,0)}]
					\draw[fedge](0,1) .. controls +(-.4*\factor,0) and +(0,.4*\factor,0) .. (-.7*\factor,0 )
					node[pos=0,above,minimum size=.6cm]{$x_g$};
					\draw[fedge](0,-1) .. controls +(-.4*\factor,0) and +(0,-.4*\factor,0) .. (-.7*\factor,0 );		\draw[fedge](0,1) .. controls +(\factor,0) and +(0,\factor) .. (0,-1 );
					\draw[white, fill=white] (.16,-.18) circle (.12);
					\draw[fedge](0,1) .. controls +(-.6*\factor,-\factor) and +(-.4*\factor,-.2*\factor) .. (.7*\factor,0 );
					\draw[fedge](0,-1) .. controls +(.4*\factor,0) and +(0,-.4*\factor,0) .. (.7*\factor,0 );
				\end{scope}
				\begin{scope}[shift={(0,0)}]
					\draw[fedge](-\scopestep+.7*\factor,0)--(\scopestep-.7*\factor,0);
				\end{scope}
				\begin{scope}[shift={(\scopestep,0)}]
					\draw[fedge](-.7*\factor,0) .. controls +(0,\factor) and +(0,\factor,0) .. (.7*\factor,0);\draw[fedge](-.7*\factor,0) .. controls +(0,-\factor) and +(0,-\factor,0) .. (.7*\factor,0);
				\end{scope}
				\begin{scope}[shift={(2*\scopestep,0)}]
					\draw[dotted](\scopestep,0)--(0,0);
					\draw[hedge](-\scopestep+.7*\factor,0)--(0,0);
				\end{scope}
				\begin{scope}[shift={(3*\scopestep,0)}]
					\draw[dotted](-\scopestep,0)--(0,0);
					\draw[hedge](\scopestep-.7*\factor,0)--(0,0);
				\end{scope}
				\begin{scope}[shift={(4*\scopestep,0)}]
					\draw[fedge](-.7*\factor,0) .. controls +(0,\factor) and +(0,\factor,0) .. (.7*\factor,0);\draw[fedge](-.7*\factor,0) .. controls +(0,-\factor) and +(0,-\factor,0) .. (.7*\factor,0);
				\end{scope}
				\begin{scope}[shift={(5*\scopestep,0)}]
					\draw[fedge](-\scopestep+.7*\factor,0)--(\scopestep-.7*\factor,0);
				\end{scope}
				\begin{scope}[shift={(6*\scopestep,0)}]
					\draw[hedge](-.7*\factor,0) .. controls +(0,\factor) and +(0,\factor,0) .. (.7*\factor,0);
					\draw[hedge](-.7*\factor,0) .. controls +(0,-\factor) and +(0,-\factor,0) .. (.7*\factor,0);
				\end{scope}
				\draw [brace] (-7.8,-1.4) -- node [below, pos=0.5]{$(g-1)$ -many}(-.2,-1.4);
				\draw [brace] (.2,-1.4) -- node [below, pos=0.5]{$(n-2)$ -many}(7.8,-1.4);
			\end{scope}
			
			\begin{scope}[shift={(0,-5)}]
				\begin{scope}[shift={(-6*\scopestep,0)}]
					\draw[fedge](0,1) .. controls +(-\factor,0) and +(0-\factor,0) .. (0,-1 )
					node[pos=0,above,minimum size=.6cm]{$x_1$};
					\draw[fedge](0,1)--(0,-1 );
					\draw[fedge](0,1) .. controls +(.4*\factor,0) and +(0,.4*\factor,0) .. (.7*\factor,0 );
					\draw[fedge](0,-1) .. controls +(.4*\factor,0) and +(0,-.4*\factor,0) .. (.7*\factor,0 );
				\end{scope}
				\begin{scope}[shift={(-5*\scopestep,0)}]
					\draw[fedge](-\scopestep+.7*\factor,0)--(\scopestep-.7*\factor,0);
				\end{scope}
				\begin{scope}[shift={(-4*\scopestep,0)}]
					\draw[fedge](0,1) .. controls +(-.4*\factor,0) and +(0,.4*\factor,0) .. (-.7*\factor,0 )
					node[pos=0,above,minimum size=.6cm]{$x_2$};
					\draw[fedge](0,-1) .. controls +(-.4*\factor,0) and +(0,-.4*\factor,0) .. (-.7*\factor,0 );		\draw[fedge](0,1)--(0,-1 );
					\draw[fedge](0,1) .. controls +(.4*\factor,0) and +(0,.4*\factor,0) .. (.7*\factor,0 );
					\draw[fedge](0,-1) .. controls +(.4*\factor,0) and +(0,-.4*\factor,0) .. (.7*\factor,0 );
				\end{scope}
				\begin{scope}[shift={(-3*\scopestep,0)}]
					\draw[dotted](\scopestep,0)--(0,0);
					\draw[hedge](-\scopestep+.7*\factor,0)--(0,0);
				\end{scope}
				\begin{scope}[shift={(-2*\scopestep,0)}]
					\draw[dotted](-\scopestep,0)--(0,0);
					\draw[hedge](\scopestep-.7*\factor,0)--(0,0);
				\end{scope}
				\begin{scope}[shift={(-1*\scopestep,0)}]
					\draw[fedge](0,1) .. controls +(-.4*\factor,0) and +(0,.4*\factor,0) .. (-.7*\factor,0 )
					node[pos=0,above,minimum size=.6cm]{$x_g$};
					\draw[fedge](0,-1) .. controls +(-.4*\factor,0) and +(0,-.4*\factor,0) .. (-.7*\factor,0 );		\draw[fedge](0,1)--(0,-1 );
					\draw[fedge](0,1) .. controls +(.4*\factor,0) and +(0,.4*\factor,0) .. (.7*\factor,0 );
					\draw[fedge](0,-1) .. controls +(.4*\factor,0) and +(0,-.4*\factor,0) .. (.7*\factor,0 );
				\end{scope}
				\begin{scope}[shift={(0,0)}]
					\draw[fedge](-\scopestep+.7*\factor,0)--(\scopestep-.7*\factor,0);
				\end{scope}
				\begin{scope}[shift={(\scopestep,0)}]
					\draw[fedge](-.7*\factor,0) .. controls +(0,\factor) and +(0,\factor,0) .. (.7*\factor,0);\draw[fedge](-.7*\factor,0) .. controls +(0,-\factor) and +(0,-\factor,0) .. (.7*\factor,0);
				\end{scope}
				\begin{scope}[shift={(2*\scopestep,0)}]
					\draw[dotted](\scopestep,0)--(0,0);
					\draw[hedge](-\scopestep+.7*\factor,0)--(0,0);
				\end{scope}
				\begin{scope}[shift={(3*\scopestep,0)}]
					\draw[dotted](-\scopestep,0)--(0,0);
					\draw[hedge](\scopestep-.7*\factor,0)--(0,0);
				\end{scope}
				\begin{scope}[shift={(4*\scopestep,0)}]
					\draw[fedge](-.7*\factor,0) .. controls +(0,\factor) and +(0,\factor,0) .. (.7*\factor,0);\draw[fedge](-.7*\factor,0) .. controls +(0,-\factor) and +(0,-\factor,0) .. (.7*\factor,0);
				\end{scope}
				\begin{scope}[shift={(5*\scopestep,0)}]
					\draw[fedge](-\scopestep+.7*\factor,0)--(\scopestep-.7*\factor,0);
				\end{scope}
				\begin{scope}[shift={(6*\scopestep,0)}]
					\draw[hedge](-.7*\factor,0) .. controls +(0,\factor) and +(0,\factor,0) .. (.7*\factor,0);
					\draw[hedge](-.7*\factor,0) .. controls +(0,-\factor) and +(0,-\factor,0) .. (.7*\factor,0);
				\end{scope}
				\draw [brace] (-7.5,-1.4) -- node [below, pos=0.5]{$(g-1)$ -many}(-.5,-1.4);
				\draw [brace] (.5,-1.4) -- node [below, pos=0.5]{$(n-2)$ -many}(7.5,-1.4);
			\end{scope}
		\end{tikzpicture}
	\end{center}
	\caption{Above: a modular graph with genus $g$ and $n$ punctures. Below: its `planarization'.}\label{fig:GraphPlanarization}
\end{figure}

\begin{figure}[h!]
	\begin{center}
		\begin{tikzpicture}
			\def\factor{.4}
			\draw[fedge] (0,0) -- (1,1) node[pos=1,left]{$y_1$};
			\draw[hedge] (1,1) ..  controls +(-\factor,\factor) and +(-\factor,\factor) .. (1.6,1.6);
			\draw[hedge] (1,1)..  controls +(\factor,-\factor) and +(\factor,-\factor) .. (1.6,1.6);
			\draw[fedge] (0,0) -- (-1,1) node[pos=1,right]{$y_2$};
			\draw[hedge] (-1,1)..  controls +(-\factor,-\factor) and +(-\factor,-\factor) .. (-1.6,1.6);
			\draw[hedge] (-1,1)..  controls +(\factor,\factor) and +(\factor,\factor) .. (-1.6,1.6);
			\draw[fedge] (0,0) -- (0,-1);
			\draw[fedge] (0,-1) -- (-1,-1);
			\draw[fedge] (-1,-1) -- (-2,0) node[pos=1,right]{$y_3$};
			\draw[hedge] (-2,0)..  controls +(\factor,\factor) and +(\factor,\factor) ..  (-2.6,.6);
			\draw[hedge] (-2,0)..  controls +(-\factor,-\factor) and +(-\factor,-\factor) .. (-2.6,.6);
			\draw[fedge] (-1,-1) -- (-2,-2) node[pos=1,right]{$y_4$};
			\draw[hedge]  (-2,-2)..  controls +(\factor,-\factor) and +(\factor,-\factor) .. (-2.6,-2.6);
			\draw[hedge] (-2,-2)..  controls +(-\factor,\factor) and +(-\factor,\factor) .. (-2.6,-2.6);
			\draw[fedge] (0,-1) -- (1,-2) node[pos=1,left]{$y_5$};
			\draw[hedge] (1,-2)..  controls +(\factor,\factor) and +(\factor,\factor) .. (1.6,-2.6);
			\draw[hedge] (1,-2)..  controls +(-\factor,-\factor) and +(-\factor,-\factor) .. (1.6,-2.6);		\end{tikzpicture}
	\end{center}
	\caption{The graph $\H$.}\label{fig:LollipopTree}
\end{figure}
\end{proof}

\begin{proof}[Proof of Theorem~\ref{thm:OmgSurjection}]
Let $\mathsf F_r=\langle \alpha_1, \dots, \alpha_r\rangle$.
Recall that $\out(\mathsf F_r)$ is generated by isomorphisms of the following 3 types
(\cite{MKS}, Theorem 3.2, see also \cite{FR}, Theorem 5.2).
	\begin{enumerate}[label=Type \Roman*:]
		\item inversions $\alpha_i \mapsto \alpha_i^{-1}$ for $i=1,\dots,r$, 
		\item permutations of generators   $\alpha_i \leftrightarrow \alpha_j$ for $i,j=1,\dots,r$,
		\item the maps  $\alpha_i \mapsto \alpha_i\alpha_j$ for  $i,j=1,\dots,r, i\neq j$.
	\end{enumerate}

%
%

For $\G$, let $\H$ be a modular graph with properties described in Lemma~\ref{lem:LollipopTree}.
 Observe that $\pi_1(\H) \cong \mathsf F_r$ for $r=2g+n-1$.

We pick an edge $e$ of $\H$ which is not adjacent to a loop as the base edge. The loops in $\H$ are cyclically ordered such that two consecutive loops are joined by a sequence of left-turning edges. We number the loops by $1, \dots, r$ consistent with this order. We denote the 3-valent vertex on the $i$'th loop by $y_i$. For each $i\in\{1,\dots,r\}$, let $\alpha_i$ denote the path going from $e$ to $y_i$ on $\H$, then traversing the $i$'th loop counter-clockwise and lastly returning to $e$ again on $\H$. Note that the fundamental group $\pi_1(\H)$ is freely generated by $\alpha_1, \dots, \alpha_r$.

Shuffles at the vertices $y_i$ reverse the orientations of the corresponding loops sending $\alpha_i$ to $\alpha_i^{-1}$ and therefore yield automorphisms of  Type I.

Now let's handle the automorphisms of Type II. We choose two consecutive loops of $\H$ with respect to the cyclic order.
Applying flips to edges between these loops other than the ones directly connected to them, we obtain a local picture as in Figure~\ref{fig:TwinLoops}. A shuffle at the common root vertex $x$ of the \emph{neighboring lollipops} transposes the loops. Undoing the flips we applied to shorten the distance between the loops, we return to the graph $\H$; with the sole difference that the consecutive loops are transposed, and so are the corresponding generators of the fundamental group. Since the permutation group is generated by transpositions $(\alpha_i\alpha_{i+1})$, we can obtain any automorphism of Type II.
\begin{figure}[h!]
	\begin{center}
		\begin{tikzpicture}
			\def\factor{.4}
			\edge{(0,0)}{(1,1)}
			\node[anchor=north west](X){$x$};
			\halfedgebezier{(1,1)}{(1.6,1.6)}{+(-\factor,\factor) and +(-\factor,\factor)} 
			\halfedgebezier{(1,1)}{(1.6,1.6)}{+(\factor,-\factor) and +(\factor,-\factor)}
			\edge{(0,0)}{(-1,1)}
			\halfedgebezier{(-1,1)}{(-1.6,1.6)}{+(-\factor,-\factor) and +(-\factor,-\factor)} 
			\halfedgebezier{(-1,1)}{(-1.6,1.6)}{+(\factor,\factor) and +(\factor,\factor)}
			\halfedge{(0,0)}{(0,-.75)}
			\draw[dotted] (0,-.75) -- (0,-1.5);
		\end{tikzpicture}
	\end{center}
	\caption{Neighboring lollipops.}\label{fig:TwinLoops}
\end{figure}

In order to obtain automorphism of Type III, we first choose $i \in \{1,\dots,r\}$ and apply flips as above so that the $i-1$'st and $i$'th lollipops become neighbors. Then we implement the following procedure depicted in Figure~\ref{fig:AddLoop}. In graphs (I) and (II), we apply flips to the edges highlighted by yellow. Then we apply shuffles at the vertices highlighted by pink in graph (III), similar to shuffles in Figure~\ref{fig:ThetaShuffles}. Then, again applying flips to edges highlighted by yellow in graphs (IV) and (V), we return to graph (VI) which is isomorphic to graph (I). Lastly, by a relabeling we return to the original graph (I).

\begin{figure}[h!]
\begin{center}
\begin{tikzpicture}[scale=.8, font=\footnotesize]
\def\factor{.4}
\begin{scope}
	\draw[emphline](0,0)--(-1,1);
	\draw[fedge] (0,0)--(0,-1.2)
		node[pos=.5,left]{$a$};
	\draw[fedge] (0,0) -- (-1,1)
		node[pos=.5,below left]{$b$};
	\draw[hedge] (-1,1) .. controls +(-\factor,-\factor) and +(-\factor,-\factor) .. (-1.6,1.6)
		node[pos=1,above left]{$c$};
	\draw[hedge] (-1,1) .. controls +(\factor,\factor) and +(\factor,\factor) .. (-1.6,1.6);
	\draw[fedge] (0,0) -- (1,1)
		node[pos=.5,below right]{$d$};
	\draw[hedge] (1,1) .. controls +(-\factor,\factor) and +(-\factor,\factor) .. (1.6,1.6)
		node[pos=1,above right]{$e$};
	\draw[hedge] (1,1) .. controls +(\factor,-\factor) and +(\factor,-\factor) .. (1.6,1.6);
	\draw[dotted] (0,-1.2) -- (.5,-1.7);
	\draw[dotted] (0,-1.2) -- (-.5,-1.7);
	\node at (0,-2){(I)};
\end{scope}

\begin{scope}[shift={(5,0)}]
	\draw[emphline](.5,.5)--(1,1);
	\draw[fedge] (.5,.5)--(1,1)
		node[pos=.5,below right]{$d$};
	\draw[hedge] (1,1) .. controls +(-\factor,\factor) and +(-\factor,\factor) .. (1.6,1.6)
		node[pos=1,above right]{$e$};
	\draw[hedge] (1,1) .. controls +(\factor,-\factor) and +(\factor,-\factor) .. (1.6,1.6);
	\draw[fedge] (0,-.5) .. controls +(-2*\factor,\factor) and +(-\factor,\factor) .. (.5,.5)
		node[pos=.5,above left]{$c$};
	\draw[fedge] (0,-.5) .. controls +(2*\factor,-.5*\factor) and +(\factor,-.5*\factor) .. (.5,.5)
		node[pos=.5,below right]{$b$};
	\draw[fedge] (0,-.5)--(0,-1.2)
		node[pos=.5,left]{$a$};;
	\draw[dotted] (0,-1.2) -- (.5,-1.7);
	\draw[dotted] (0,-1.2) -- (-.5,-1.7);
	\node at (0,-2){(II)};
\end{scope}

\begin{scope}[shift={(11,0)}]
	\coordinate (A) at (0,-.5);	
	\coordinate (B) at (.7,.5);	
	\coordinate (C) at (0,1.5);
	\coordinate (D) at (-.7,.5);
	
	\fill[pink] (B) circle (8pt);
	\fill[pink] (D) circle (8pt);
	\draw[fedge] (B) -- (D)
		node[pos=.5,above]{$d$};
	\draw[fedge] (A) .. controls +(\factor,0) and +(0,-\factor) .. (B)
		node[pos=.5,right]{$b$};
	\draw[hedge] (B) .. controls +(0,\factor) and +(\factor,0) .. (C)
		node[pos=1,above]{$e$};
	\draw[fedge] (A) .. controls +(-\factor,0) and +(0,-\factor) .. (D)
		node[pos=.5,left]{$c$};;
	\draw[hedge] (D) .. controls +(0,\factor) and +(-\factor,0) .. (C);
	\draw[fedge] (A) -- (0,-1.2)
		node[pos=.5,left]{$a$};;
	\draw[dotted] (0,-1.2) -- (.5,-1.7);
	\draw[dotted] (0,-1.2) -- (-.5,-1.7);
	\node at (0,-2){(III)};
\end{scope}

\begin{scope}[shift={(0,-5)}]
	\coordinate (A) at (0,-.5);	
	\coordinate (B) at (.7,.5);	
	\coordinate (C) at (0,1.5);
	\coordinate (D) at (-.7,.5);
	
	\draw[emphline] (B) -- (D);
	
	\draw[fedge] (B) -- (D)
	node[pos=.5,above]{$e$};
	\draw[fedge] (A) .. controls +(\factor,0) and +(0,-\factor) .. (B)
	node[pos=.5,right]{$b$};
	\draw[hedge] (B) .. controls +(0,\factor) and +(\factor,0) .. (C)
	node[pos=1,above]{$d$};
	\draw[fedge] (A) .. controls +(-\factor,0) and +(0,-\factor) .. (D)
	node[pos=.5,left]{$c$};;
	\draw[hedge] (D) .. controls +(0,\factor) and +(-\factor,0) .. (C);
	\draw[fedge] (A) -- (0,-1.2)
	node[pos=.5,left]{$a$};;
	\draw[dotted] (0,-1.2) -- (.5,-1.7);
	\draw[dotted] (0,-1.2) -- (-.5,-1.7);
	\node at (0,-2){(IV)};
\end{scope}

\begin{scope}[shift={(5,-5)}]
	\draw[emphline] (0,-.5) .. controls +(2*\factor,-.5*\factor) and +(\factor,-.5*\factor) .. (.5,.5)
	node[pos=.5,below right]{$b$};
	
	\draw[fedge] (.5,.5)--(1,1)
	node[pos=.5,below right]{$e$};
	\draw[hedge] (1,1) .. controls +(-\factor,\factor) and +(-\factor,\factor) .. (1.6,1.6)
	node[pos=1,above right]{$d$};
	\draw[hedge] (1,1) .. controls +(\factor,-\factor) and +(\factor,-\factor) .. (1.6,1.6);
	\draw[fedge] (0,-.5) .. controls +(-2*\factor,\factor) and +(-\factor,\factor) .. (.5,.5)
	node[pos=.5,above left]{$c$};
	\draw[fedge] (0,-.5) .. controls +(2*\factor,-.5*\factor) and +(\factor,-.5*\factor) .. (.5,.5)
	node[pos=.5,below right]{$b$};
	\draw[fedge] (0,-.5)--(0,-1.2)
	node[pos=.5,left]{$a$};;
	\draw[dotted] (0,-1.2) -- (.5,-1.7);
	\draw[dotted] (0,-1.2) -- (-.5,-1.7);
	\node at (0,-2){(V)};
\end{scope}

\begin{scope}[shift={(11,-5)}]
	\draw[fedge] (0,0)--(0,-1.2)
	node[pos=.5,left]{$a$};
	\draw[fedge] (0,0) -- (-1,1)
	node[pos=.5,below left]{$b$};
	\draw[hedge] (-1,1) .. controls +(-\factor,-\factor) and +(-\factor,-\factor) .. (-1.6,1.6)
	node[pos=1,above left]{$c$};
	\draw[hedge] (-1,1) .. controls +(\factor,\factor) and +(\factor,\factor) .. (-1.6,1.6);
	\draw[fedge] (0,0) -- (1,1)
	node[pos=.5,below right]{$e$};
	\draw[hedge] (1,1) .. controls +(-\factor,\factor) and +(-\factor,\factor) .. (1.6,1.6)
	node[pos=1,above right]{$d$};
	\draw[hedge] (1,1) .. controls +(\factor,-\factor) and +(\factor,-\factor) .. (1.6,1.6);
	\draw[dotted] (0,-1.2) -- (.5,-1.7);
	\draw[dotted] (0,-1.2) -- (-.5,-1.7);
	\node at (0,-2){(VI)};
\end{scope}
\end{tikzpicture}
\end{center}
\caption{Morphisms in $\oomg$ on neighboring lollipops. Yellow highlight indicate flips and pink highlights indicate shuffles.}\label{fig:AddLoop}
\end{figure}

The letters $a,b,c,d$ and $e$ in the graph diagrams denote labels of the full edges. We assume counterclockwise order for the half edges on the loops, and we indicate the clockwise order by the inverse notation, as in $a^{-1}$. The fundamental group generators related to the $i-1$'st and $i$'th loops can be described by the full-edge sequences $adeda$ and $abcba$, respectively, in graph (I). In the fundamental group of graph (II), they are mapped to paths $abdedba$ and $acba$, respectively. The paths they are mapped in each graph are described in the following table:

\begin{center}
	\begin{tabular}{l||c|c|c|c|c|c}
		&	(I)		&	(II)		&	(III)	&	(IV)	&	(V)				&	(VI)	\\ [1ex]	\hline  
$i-1$	&	$adeda$	&	$abdedba$	&	$abedba$&	$abedba$&	$abed^{-1}eba$	&	$aed^{-1}ea$	\rule[3pt]{0pt}{10pt}\\
$i$		&	$abcba$	&	$acba$		&	$acdba$	&	$acdba$	&	$acedeba$		&	$abcbedea$
	\end{tabular}
\end{center}

Reversing the initial procedure to obtain the neighboring lollipops, we end up with an isomorphism on the fundamental group mapping $\alpha_{i-1}$ to $\alpha_{i-1}^{-1}$ and $\alpha_{i}$ to $\alpha_{i}\alpha_{i-1}$. Composing this morphism with a morphism of Type I gives us an isomorphism of fundamental groups fixing every generator other than $\alpha_i$, which is mapped to $\alpha_i\alpha_{i-1}$. Conjugate of this morphism with a morphism of Type II transposing $\alpha_{i-1}$ and an $\alpha_j$ for any choice of $j$ produces the desired morphism of Type III.

\end{proof}			

			
			\section{$\pimcg$-representation on punctures}
			Recall that flips and relabelings maps punctures to punctures. 
			Consider the map $\psi$ sending each element of ${\mathrm{MGR}}_\ell$ to its set of punctures and each morphism
			of $\pimcg$ to the bijection it induces between the sets of punctures. We have an exact sequence of groupoids
			\begin{equation} \label{puregrp}
				1\to \mathbf{PUR} \to \pimcg \stackrel{\psi}{\longrightarrow} \mathbf{PUN}
			\end{equation}
			where $\mathbf{PUN}$ is the groupoid whose object set is ${\mathrm{MGR}}_\ell$ and morphisms are bijections among their puncture sets. \label{purrpunn}
			 $\mathbf{PUR}$ is			the kernel of $\psi$: it is the totally disconnected groupoid whose object set is ${\mathrm{MGR}}_\ell$ and whose set of morphisms is a subset of the flip-induced fundamental groupoid automorphisms that keep each puncture fixed. (Since flips and relabelings preserve the set of finite (resp. infinite) punctures; one may also consider the version $\psi_{finite}$ (resp. $\psi_\infty$) 
of $\psi$ sending the morphisms of $\pimcg$ to 
bijections among finite (resp. infinite) punctures.)
			
			If $\mathcal G$ is finite, then the isotropy group of $\mathbf{PUR}(\mathcal G)$ surjects onto the pure mapping class group.
			The case of  infinite $\mathcal G$ is most interesting when $\mathcal G$  has no finite punctures.			For instance, one has
			\begin{equation}\label{purpun}
				1\to \mathbf{PUR}(|\F,\ell|) \to \pmg(|\F,\ell|) \to \mathbf{PUN}(|\F,\ell|)
			\end{equation}
			for every labeling $\ell$. 
Since we may identify (via the continued fraction map) the set of punctures of $\F$ by 
$\overline{\Q}:=\Q\cup\{\infty\}$, the groupoid $\mathbf{PUN}(|\F,\ell|)$
is a subgroupoid of $\mathsf{Sym}(\overline{\Q})$.
These permutations are exactly those which admits an extension to the circle as elements of $\ppsl$, i.e.
$$
\mathbf{PUN}(|\F,\ell|)\simeq \aut^\infty(\F) \simeq \ppsl\simeq \mathsf T.
$$
Comparing  (\ref{purpun}) with (\ref{symfary}) shows that 
$\mathbf{PUR}(|\F,\ell|)$ is the group $\mathsf{Sym}^\pm(\psl)$.

 For other infinite elements of ${\mathrm{MGR}_\ell}$ without finite punctures (e.g. hyperbolic çarks), we suspect that the group $\mathsf{Sym}^\pm(|\G,\ell|)$ of finite edge permutations is normal in $\mathbf{PUR}(|\G,\ell|)$; the quotient being related to the mapping class group of the compact surface $S$ obtained from an ambient surface of $\mathcal G$ by filling in the boundary components by discs. 

			\small
			\bibliographystyle{abbrv}
\section*{Acknowledgements}
This research has been funded by the grants GSÜBAP FIR-2022-1089, 
TÜBİTAK BOSPHORUS 221N171,
TÜBİTAK 115F412, and TÜBİTAK  119F405. \"O. \"Ulkem was supported by TÜBİTAK project no. 119F405.
			We are indebted to Athanase Papadopoulos, Louis Funar and to Vlad Sergiescu for their comments on the work.

		\end{document}